\documentclass[12pt,a4paper]{article}
\usepackage[utf8]{inputenc}
\usepackage[T1]{fontenc}
\usepackage[numbers]{natbib}
\usepackage[french,english]{babel}
\usepackage{amsmath, amssymb, amsthm}
\usepackage{mathtools}
\usepackage{bm}
\usepackage{geometry}
\usepackage{etoolbox}
\usepackage{float}
\IfFileExists{lmodern.sty}{\usepackage{lmodern}}{}
\IfFileExists{microtype.sty}{\usepackage{microtype}}{}
\IfFileExists{setspace.sty}{\usepackage{setspace}\setstretch{1.03}}{}
\IfFileExists{hyperref.sty}{\usepackage[hidelinks]{hyperref}}{}

\allowdisplaybreaks
\numberwithin{equation}{section}

\newtheorem{proposition}{Proposition}[section]
\newtheorem{lemma}{Lemma}[section]
\newtheorem{theorem}{Theorem}[section]
\newtheorem{corollary}{Corollary}[section]
\newtheorem{remark}{Remark}[section]
\theoremstyle{definition}
\newtheorem{definition}{Definition}[section]
\newtheorem{hypothesis}{Hypothesis}[section]

\newcommand{\R}{\mathbb{R}}

\title{Robust McKean--Vlasov Variational Systems with Asymmetric Loss Aversion: Well-Posedness, Stability, and Propagation of Chaos for the Forward and Regularized Backward Systems}
\author{
  Kayembe Tshiswaka Tcheick$^{1}$,
  Mabela Matendo Rostin$^{1}$,
  Bosonga Bofeki$^{2}$,\\
  Mbuyi Mukendi Eugene$^{1}$\\[2ex]
  $^{1}$Faculty of Science and Technology,\\
  \hspace*{1cm}University of Kinshasa, Kinshasa, D.R. Congo\\[1ex]
  $^{2}$Faculty of Economics and Management Sciences,\\
  \hspace*{1cm}University of Kinshasa, Kinshasa, D.R. Congo
}
\date{}

\begin{document}

\selectlanguage{english}
\maketitle

\begin{abstract}
We study a class of robust forward--backward McKean--Vlasov variational systems under model uncertainty represented by a non-dominated family of probability measures. Mean-field interactions are described through nonlinear collective observables acting on the laws of the forward and backward components. To model asymmetric loss aversion, we introduce a nonsmooth convex functional whose subdifferential defines a law-dependent maximal monotone operator acting on the forward state. We establish existence, uniqueness, and stability of the robust forward dynamics by a fixed-point argument in Wasserstein space. The backward component is formulated as a selected backward variational system rather than a classical backward stochastic variational inequality. Our analysis relies on Yosida regularization, uniform a priori estimates, convergence of the regularized solutions, and a Minty--Br\'ezis identification argument, yielding a canonical solution associated with the minimal norm selection. We further construct a particle approximation and prove propagation of chaos for the forward dynamics with explicit convergence rates uniformly over the non-dominated family. For each fixed regularization parameter, we also establish quantitative propagation of chaos for the regularized backward component and explain why estimates uniform in both the number of particles and the regularization parameter require additional non-contact assumptions near the nonsmooth threshold.
\end{abstract}



\newpage

\begin{otherlanguage}{french}
\medskip
\noindent\textbf{Résumé en français.}

Nous \'etudions une classe de syst\`emes variationnels de McKean--Vlasov forward--backward robustes, dans un cadre d'incertitude de mod\`ele repr\'esent\'e par une famille non domin\'ee de mesures de probabilit\'e. Les interactions de champ moyen sont d\'ecrites par des observables collectives non lin\'eaires agissant sur les lois des composantes forward et backward. Pour mod\'eliser l'aversion asym\'etrique aux pertes, nous introduisons une fonctionnelle convexe non lisse dont le sous-diff\'erentiel d\'efinit un op\'erateur maximal monotone d\'ependant de la loi et agissant sur l'\'etat forward. Nous \'etablissons l'existence, l'unicit\'e et la stabilit\'e de la dynamique forward robuste par un argument de point fixe dans l'espace de Wasserstein. La composante backward est formul\'ee comme un syst\`eme variationnel backward s\'electionn\'e plut\^ot que comme une in\'egalit\'e variationnelle stochastique backward classique. Notre analyse repose sur la r\'egularisation de Yosida, des estimations a priori uniformes, la convergence des solutions r\'egularis\'ees et un argument d'identification de Minty--Br\'ezis, conduisant \`a une solution canonique associ\'ee \`a la s\'election de norme minimale. Nous construisons \'egalement une approximation particulaire et d\'emontrons la propagation du chaos pour la dynamique forward avec des vitesses de convergence explicites, uniform\'ement sur la famille non domin\'ee. Pour chaque param\`etre de r\'egularisation fix\'e, nous \'etablissons \'egalement une propagation du chaos quantitative pour la composante backward r\'egularis\'ee et expliquons pourquoi des estim\'ees uniformes \`a la fois en le nombre de particules et en le param\`etre de r\'egularisation n\'ecessitent des hypoth\`eses de non-contact suppl\'ementaires pr\`es du seuil non lisse.
\end{otherlanguage}

\section{Introduction}

\subsection{Background and motivation}

The theory of McKean--Vlasov systems addresses a central question in statistical physics and probability theory: how can one describe the dynamics of a large population when each agent reacts not to its neighbors one by one, but to the aggregate state of the population? Since the foundational work of McKean \cite{McKean1966} and the rigorous formulation of propagation of chaos by Sznitman \cite{Sznitman1991}, this framework has become a reference tool for the study of large systems of interacting agents. It now underpins a significant part of the mean field game theory introduced by Lasry and Lions \cite{LasryLions2007}, as well as many models in mathematical finance, stochastic control, and multi-agent systems analysis.

It also naturally leads to coupled models.

Forward--backward McKean--Vlasov systems enrich this perspective by coupling state dynamics with a backward component. The latter can represent a future cost, an adjoint variable, a continuation value, or an optimality signal. The controlled case was developed, among others, by Carmona and Delarue \cite{CarmonaDelarue2015}, and their monographs \cite{CarmonaDelarue2018I,CarmonaDelarue2018II} provide a major synthesis of the theory. Recent developments, notably those of Bayraktar and Zhang \cite{BayraktarZhang2023} and Hua and Luo \cite{HuaLuo2024}, further extend its range of applications.

Three limitations, however, remain important.

The first concerns the way individual preferences are modeled. In classical models, gains and losses are often treated symmetrically. Prospect theory, due to Kahneman and Tversky \cite{KahnemanTversky1979}, shows that a loss is generally felt more strongly than a gain of the same magnitude. Introducing this loss aversion into a forward--backward variational system naturally leads to nonsmooth convex functionals and multivalued monotone operators.

The second limitation concerns the form of collective interactions. In many works, law dependence is reduced to a single average. This reduction is sometimes too poor to represent situations where agents react to more elaborate aggregate indicators: overall risk level, confidence index, collective benchmark, weighted performance, or systemic vulnerability measure. It then becomes necessary to work with nonlinear collective observables acting directly on distributions.

Finally, model uncertainty requires a more flexible framework than a single reference probability. Volatility parameters, diffusion mechanisms, or certain structural characteristics may be ambiguous. The works of Denis, Hu and Peng \cite{DenisHuPeng2011}, Soner, Touzi and Zhang \cite{SonerTouziZhang2011,SonerTouziZhang2012}, Nutz and van Handel \cite{NutzVanHandel2013}, and Peng \cite{Peng2019} laid the foundations for robust stochastic analysis under non-dominated families of probabilities. The interaction between this robust framework, variational systems and McKean--Vlasov dynamics is still only partially understood.

This observation is the starting point of the paper. We combine mean-field interactions described by collective observables, a selected backward variational dynamics driven by a nonsmooth monotone signal evaluated at the forward state, and robustness formulated with respect to a non-dominated family of admissible probabilities.

\subsection{The model under study}

We consider a class of robust forward--backward variational McKean--Vlasov systems whose backward component is driven by a monotone ALA signal evaluated at the forward state:

\begin{equation}\label{eq:intro_system}
\begin{cases}
dX_t = b\bigl(t,X_t,\Psi_\varphi(\mu_t^P),u_t\bigr)dt + \sigma\bigl(t,X_t,\Psi_\beta(\mu_t^P),u_t\bigr)dB_t,\\[4pt]
-dY_t =
\Bigl[f\bigl(t,X_t,\Psi_\gamma(\mu_t^P),Y_t,Z_t,\Psi_\delta(\nu_t^P),u_t\bigr) + \rho\,\Gamma_t\Bigr]dt - Z_t dB_t,\\[4pt]
X_0 = \xi,\qquad Y_T = g\bigl(X_T,\Psi_\lambda(\mu_T^P)\bigr),\\[4pt]
\mu_t^P = \mathcal L^P(X_t),\qquad \nu_t^P = \mathcal L^P(Y_t),\\[4pt]
\Gamma_t\in A_{\mathrm{ALA}}\bigl(X_t,\Upsilon(\mu_t^P)\bigr).
\end{cases}
\end{equation}

The canonical formulation studied below corresponds to the minimal norm selection
\[
\Gamma_t=A_{\mathrm{ALA}}^0\bigl(X_t,\Upsilon(\mu_t^P)\bigr)
:=
\operatorname*{arg\,min}_{\xi\in A_{\mathrm{ALA}}(X_t,\Upsilon(\mu_t^P))}|\xi|.
\]

The maps $\Psi_\varphi,\Psi_\beta,\Psi_\gamma,\Psi_\delta,\Psi_\lambda$ are collective observables of the form
\[
\Psi_h(\mu) = \int h(x)\,\mu(dx),
\]
where $h$ is a Lipschitz function. The ALA reference observable is denoted by
\[
\Upsilon(\mu)=\int \chi(x)\,\mu(dx),
\]
with \(\chi\) Lipschitz. This formulation encompasses the classical mean case, but also allows for finer aggregated indicators.

Asymmetric loss aversion is carried by the convex functional

\begin{equation}\label{eq:intro_Phi}
\Phi_{\mathrm{ALA}}(x,\mu)
=
\sum_{i=1}^d
\left[
\kappa_+ \bigl(x_i-\Upsilon(\mu)\bigr)_+
+
\kappa_- \bigl(x_i-\Upsilon(\mu)\bigr)_-
\right]
+
\frac{\varepsilon}{2}|x|^2
+
\frac{\beta}{2}W_2^2(\mu,\mu^\star),
\end{equation}

with $\kappa_- > \kappa_+ > 0$. The positive and negative parts are taken componentwise.

The term $\Upsilon(\mu)$ plays the role of a collective reference. When $\chi(x)=x$, one recovers a mean reference; other choices of $\chi$ allow modeling a risk index, a performance measure, or a nonlinear benchmark. The inequality $\kappa_- > \kappa_+$ expresses the behavioral idea that an unfavorable deviation from the collective reference is penalized more severely than a favorable deviation of the same magnitude. The Wasserstein distance term introduces a comparison with a reference distribution $\mu^\star$, which can be interpreted as a regulatory constraint, a social norm, or a collective objective.

The functional \(\Phi_{\mathrm{ALA}}\) is not differentiable on the critical hyperplane \(\{x=\Upsilon(\mu)\mathbf 1\}\). Its \(x\)-subdifferential defines a maximal monotone multi-valued operator in the sense of Brézis \cite{Brezis1973,Brezis2011} and Rockafellar \cite{Rockafellar1970}, but this operator depends on the law only through the reference \(\Upsilon(\mu)\). In the present model, it is evaluated at the forward state \(X_t\), and not at the backward variable \(Y_t\). Hence the backward component is not a classical BSVI governed by a monotone operator in \(Y\). It is a selected backward variational system driven by the monotone signal \(A_{\mathrm{ALA}}(X_t,\Upsilon(\mu_t^P))\). The terminology and techniques of BSVI remain relevant for the variational and Yosida arguments, but the well-posedness and uniqueness statements below concern the selected formulation, in particular the canonical minimal norm selection.

This distinction is important for positioning the paper. Unlike classical backward stochastic variational inequalities, where the maximal monotone operator constrains the backward state itself, the present model uses a monotone variational signal generated by the forward state and by a collective reference. Thus the backward equation remains a valuation or adjoint-type dynamics, while the nonsmooth behavioral force is transmitted through the exogenous signal
\[
\Gamma_t^P\in A_{\mathrm{ALA}}\bigl(X_t^P,\Upsilon(\mu_t^P)\bigr).
\]
This is precisely what allows the analysis to combine BSVI-inspired tools, McKean--Vlasov interactions, and robust non-dominated probabilities without claiming to solve a classical BSVI in the variable \(Y\).

\subsection{Analytical difficulties and strategy}

The analysis of the above system involves three difficulties that cannot be treated separately.

The first stems from the non-differentiable nature of $\Phi_{\mathrm{ALA}}$. On the critical set, the subdifferential $\partial_x\Phi_{\mathrm{ALA}}$ is no longer a singleton. Since this subdifferential is evaluated at \(X\), the difficulty is not a monotonicity issue in \(Y\), but a selection and stability issue for the monotone signal entering the backward variational dynamics. We circumvent this difficulty by Yosida regularization \cite{Yosida1965}, which replaces the monotone signal by a family of Lipschitz functions before passing to the limit by means of a Minty--Brezis argument.

The second difficulty concerns the generalized law dependence. The coefficients do not depend only on a mean, but on several distinct collective observables. The natural space is then the Wasserstein space \cite{Villani2009,AmbrosioGigliSavare2005}, and the stability of observables with respect to $W_2$ becomes an essential ingredient of the McKean--Vlasov fixed point.

The third difficulty is related to model uncertainty. A priori estimates, stability results, compactness arguments and limit theorems must be uniform over the non-dominated family $\mathcal P$. The robust spaces $L^2_{\mathcal P}$, $\mathcal S^2_{\mathcal P}$ and $\mathcal H^2_{\mathcal P}$ provide the appropriate functional framework.

Finally, the particle approximation requires simultaneously controlling the mean-field error, the variational structure, and the robust uniformity.

\subsection{Main contributions}

The present work introduces a new class of robust McKean--Vlasov variational systems driven by a maximal monotone operator associated with asymmetric loss aversion. Within this framework, we establish well-posedness and stability under non-dominated model uncertainty, develop a Yosida regularization procedure preserving the variational structure, and prove propagation of chaos for both the forward dynamics and the regularized backward component. We also identify the mathematical obstruction to estimates that are simultaneously uniform with respect to the number of particles and the regularization parameter, thereby giving a precise description of the scope and limitations of the proposed approach.

The contributions of the paper can be summarized in four main points.

\textbf{A robust McKean--Vlasov variational formulation with a maximal monotone ALA signal.}
We introduce a robust forward--backward McKean--Vlasov system in which the collective interaction is described by nonlinear observables of the law and the behavioral component is generated by an asymmetric loss-aversion functional. The associated \(x\)-subdifferential defines a law-dependent maximal monotone signal evaluated at the forward state. This formulation is not a classical BSVI in the backward variable; it is a selected backward variational system driven by
\[
\Gamma_t^P\in A_{\mathrm{ALA}}\bigl(X_t^P,\Upsilon(\mu_t^P)\bigr).
\]

\textbf{Well-posedness, canonical uniqueness and stability.}
We establish the analytical properties of the ALA functional, including convexity, coercivity, strong monotonicity, controlled growth and Wasserstein stability. We then prove existence, uniqueness and stability for the robust forward dynamics through a fixed point on the space of measure flows. For the backward component, we construct Yosida-regularized selected systems, obtain uniform estimates, identify the limit by a Minty--Brezis graph-closure argument, and prove canonical uniqueness of the backward pair associated with the minimal norm selection.

\textbf{Forward and regularized backward propagation of chaos.}
We analyze a particle approximation of the system and prove a quantitative forward propagation of chaos with an explicit Fournier--Guillin rate, uniformly over the non-dominated family of admissible probabilities. In addition, for every fixed \(\lambda>0\), we prove a quantitative propagation of chaos result for the Yosida-regularized backward component. This result is stated separately as Theorem~\ref{thm:backward-chaos-fixed-lambda}, so that the distinction between non-regularized forward propagation and fixed-\(\lambda\) backward propagation is visible from the outset.

\textbf{Limitations of the method and future directions.}
We identify the precise limitation of the approach: a rate that is uniform simultaneously in \(N\) and \(\lambda\) cannot be expected without additional control of the nonsmooth threshold. We therefore discuss the role of a possible non-contact condition and outline future directions, including the simultaneous limit \(N\to\infty\), \(\lambda\downarrow0\), robust mean field games, robust control, and financial applications involving ambiguity-sensitive behavioral preferences.

\subsection{Organization of the paper}

Section 2 sets the robust probabilistic framework, introduces the functional spaces, the collective observables and the structural assumptions of the model. Section 3 establishes the existence, uniqueness and stability of the robust forward dynamics. Section 4 develops the variational analysis of the selected backward component: ALA functional, Yosida regularization, uniform estimates, graph closure by the Minty--Brezis argument, uniqueness and stability. Section 5 is devoted to the particle system, forward propagation of chaos, and backward propagation of chaos for fixed \(\lambda\). Section 6 discusses the limits of the simultaneous passage \(N\to\infty\), \(\lambda\downarrow0\), and the role of possible non-contact assumptions. The conclusion summarizes the scope of the results and presents several research directions.

\begin{figure}[H]
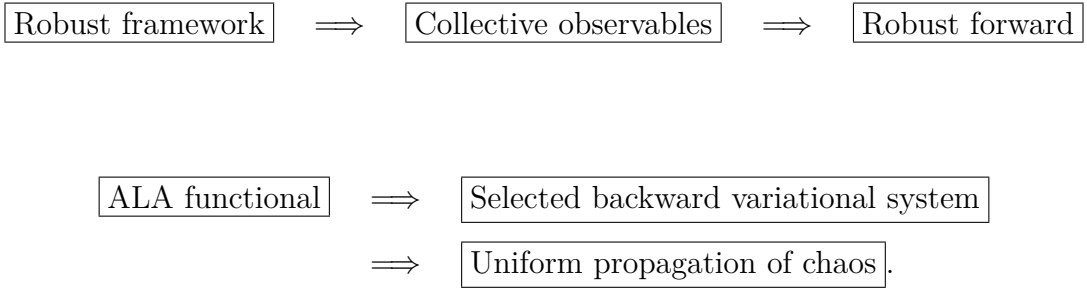

\centering
\[
\boxed{\text{Robust framework}}
\quad\Longrightarrow\quad
\boxed{\text{Collective observables}}
\quad\Longrightarrow\quad
\boxed{\text{Robust forward}}
\]

\vspace{0.5cm}

\[
\begin{aligned}
\boxed{\text{ALA functional}}
&\quad\Longrightarrow\quad
\boxed{\text{Selected backward variational system}}\\[3pt]
&\quad\Longrightarrow\quad
\boxed{\text{Uniform propagation of chaos}}.
\end{aligned}
\]

\caption{Mathematical structure of the model and main results.}
\label{fig:structure}
\end{figure}

\section{Robust probabilistic framework, collective observables and assumptions}

We fix in this section the basic objects of the model: the robust probabilistic framework, the functional spaces adapted to non-dominated families, the collective observables and the structural assumptions. The existence, uniqueness and stability results for the forward dynamics will be established in Section 3.

\subsection{Canonical space and admissible probabilities}

Let $T>0$. Set $\Omega = C([0,T];\mathbb{R}^d)$, $\mathcal{F}$ its Borel $\sigma$-algebra, and $B_t(\omega)=\omega(t)$ the canonical process. The filtration $\mathbb{F}=(\mathcal{F}_t)_{0\le t\le T}$ is the natural filtration of $B$, completed and made right-continuous.

Let $\underline a,\overline a\in\mathbb{S}_d^+$ with $0<\underline a\le\overline a$. The set of admissible volatility processes is
\[
\mathcal{A} = \{ a : [0,T]\times\Omega\to\mathbb{S}_d^+ \mid a \text{ progressively measurable},\ \underline a\le a_t\le\overline a \}.
\]
For $a\in\mathcal{A}$, let $P^a$ be the probability under which $d\langle B\rangle_t = a_t dt$. The family of admissible probabilities is
\[
\mathcal{P} = \{P^a : a\in\mathcal{A}\}.
\]
$\mathcal{P}$ is not dominated; all estimates will be uniform in $P$ (cf. \cite{DenisHuPeng2011, SonerTouziZhang2011}).

\subsection{Robust spaces and capacity}

\begin{definition}\label{def:L2P}
$L^2_{\mathcal{P}} = \{ X \mid \|X\|_{L^2_{\mathcal{P}}} := (\sup_{P\in\mathcal{P}}\mathbb{E}^P[|X|^2])^{1/2}<\infty \}$ (modulo $\mathcal{P}$-q.s. equality).
\end{definition}

\begin{definition}\label{def:S2P}
$\mathcal{S}^2_{\mathcal{P}} = \{ Y \text{ progressively measurable} \mid \|Y\|_{\mathcal{S}^2_{\mathcal{P}}} := (\sup_{P\in\mathcal{P}}\mathbb{E}^P[\sup_{0\le t\le T}|Y_t|^2])^{1/2}<\infty \}$.
\end{definition}

\begin{definition}\label{def:H2P}
For stochastic integrals with respect to the canonical process \(B\), the natural robust norm is the quadratic-variation weighted norm
\[
\|Z\|_{\mathcal{H}^2_{\mathcal P,a}}
:=
\left(
\sup_{P=P^a\in\mathcal P}
\mathbb E^P
\left[
\int_0^T
\operatorname{Tr}(Z_ta_tZ_t^\top)\,dt
\right]
\right)^{1/2}.
\]
We set
\[
\mathcal H^2_{\mathcal P}
:=
\{Z\ \text{progressively measurable}\mid
\|Z\|_{\mathcal{H}^2_{\mathcal P,a}}<\infty\}.
\]
\end{definition}

Since \(a_t\) is uniformly elliptic, this weighted norm is equivalent to the usual norm based on \(\int_0^T |Z_t|^2dt\). In all It\^o identities, however, the weighted quantity \(\operatorname{Tr}(Z_ta_tZ_t^\top)\) is the natural one; this is the robust counterpart of the standard quadratic-variation calculus for Itô integrals \cite{KaratzasShreve1991}. For readability, we shall write
\[
|z|_{a_t}^2:=\operatorname{Tr}(za_tz^\top).
\]

These spaces are Banach spaces \cite{DenisHuPeng2011, SonerTouziZhang2011}. Their functional properties are recalled here only to fix notation; the complete proofs are part of the standard theory of robust spaces under non-dominated families of probabilities.

\begin{definition}\label{def:capacity}
$c(A)=\sup_{P\in\mathcal{P}}P(A)$ for $A\in\mathcal{F}$. $c$ is the upper capacity.
\end{definition}

\begin{definition}\label{def:polar}
$A$ is $\mathcal{P}$-polar if $c(A)=0$, i.e. $P(A)=0$ for all $P\in\mathcal{P}$.
\end{definition}

A property is $\mathcal{P}$-quasi-sure ($\mathcal{P}$-q.s.) if it holds outside a polar set \cite{NutzVanHandel2013}. These spaces provide the functional framework in which the robust forward--backward dynamics will be formulated. We now introduce the nonlinear collective observables and the structural objects of the model.

\subsection{Wasserstein distance and collective observables}

$\mathcal{P}_2(\mathbb{R}^d)$ denotes the set of measures on $\mathbb{R}^d$ with finite second moment. The Wasserstein distance $W_2$ is given by
\[
W_2(\mu,\nu)^2 = \inf_{\pi\in\Pi(\mu,\nu)}\int|x-y|^2\pi(dx,dy),
\]
where $\Pi(\mu,\nu)$ is the set of couplings \cite{Villani2009}. For a complete presentation, see \cite{AmbrosioGigliSavare2005}. For any process $X$, we denote $\mu_t^P=\mathcal{L}^P(X_t)$.

$U\subset\mathbb{R}^k$ is closed convex. We set
\[
\mathcal{U} = \Big\{ u \text{ progressively measurable},\ u_t\in U,\ \sup_{P\in\mathcal{P}}\mathbb{E}^P\Big[\int_0^T|u_t|^2dt\Big]<\infty \Big\}.
\]

Let $\varphi:\mathbb R^d\to\mathbb R^{d_\varphi},\ \beta:\mathbb R^d\to\mathbb R^{d_\beta},\ \gamma:\mathbb R^d\to\mathbb R^{d_\gamma},\ \delta:\mathbb R\to\mathbb R^{d_\delta},\ \lambda:\mathbb R^d\to\mathbb R^{d_\lambda}$ be given Lipschitz functions, where $d_\varphi,d_\beta,d_\gamma,d_\delta,d_\lambda$ are positive integers.

For every $P\in\mathcal{P}$ and every process $X$ (respectively $Y$), we define the collective observables
\begin{align*}
M_t^{\varphi,P} &= \mathbb E^P[\varphi(X_t)], &
M_t^{\beta,P} &= \mathbb E^P[\beta(X_t)], \\
M_t^{\gamma,P} &= \mathbb E^P[\gamma(X_t)], &
M_t^{\delta,P} &= \mathbb E^P[\delta(Y_t)], \\
M_T^{\lambda,P} &= \mathbb E^P[\lambda(X_T)].
\end{align*}

\begin{remark}
Although these observables are entirely determined by the marginal law $\mu_t^P = \mathcal{L}^P(X_t)$ (or by the law of $Y_t$ for $M_t^{\delta,P}$), their explicit introduction allows us to distinguish the probabilistic interaction mechanism from the economic quantities that enter the coefficients. This separation will be particularly useful in the stability analysis and in the formulation of the fixed point.
\end{remark}

\begin{lemma}[Wasserstein stability of collective observables]
\label{lem:stability-observables}
Let $\mu,\nu\in\mathcal P_2(\mathbb R^d)$ and $\psi:\mathbb R^d\to\mathbb R^{d_\psi}$ be an $L_\psi$-Lipschitz function. Define $\Psi_\psi(\mu)=\int\psi\,d\mu$. Then
\[
|\Psi_\psi(\mu)-\Psi_\psi(\nu)| \le L_\psi W_2(\mu,\nu).
\]
In particular, for the observables $\varphi,\beta,\gamma,\lambda$, we respectively have
\[
|\Psi_\varphi(\mu)-\Psi_\varphi(\nu)| \le L_\varphi W_2(\mu,\nu),\qquad
|\Psi_\beta(\mu)-\Psi_\beta(\nu)| \le L_\beta W_2(\mu,\nu),
\]
\[
|\Psi_\gamma(\mu)-\Psi_\gamma(\nu)| \le L_\gamma W_2(\mu,\nu),\qquad
|\Psi_\lambda(\mu)-\Psi_\lambda(\nu)| \le L_\lambda W_2(\mu,\nu).
\]
The reference observable used in the ALA functional is a particular case of the previous lemma corresponding to the choice of the test function $\psi$.
\end{lemma}

\begin{proof}
Let $\pi\in\Pi(\mu,\nu)$ be an arbitrary coupling. By definition,
\[
\Psi_\psi(\mu)-\Psi_\psi(\nu) = \int (\psi(x)-\psi(y))\,\pi(dx,dy).
\]
Thus,
\[
|\Psi_\psi(\mu)-\Psi_\psi(\nu)| \le \int |\psi(x)-\psi(y)|\,\pi(dx,dy) \le L_\psi \int |x-y|\,\pi(dx,dy).
\]
By Cauchy-Schwarz,
\[
\int |x-y|\,\pi(dx,dy) \le \left(\int |x-y|^2\,\pi(dx,dy)\right)^{1/2}.
\]
Taking the infimum over all couplings $\pi\in\Pi(\mu,\nu)$, we obtain
\[
|\Psi_\psi(\mu)-\Psi_\psi(\nu)| \le L_\psi W_2(\mu,\nu).
\]
The particular cases are immediate.
\end{proof}

\begin{corollary}[$L^2$ stability of collective observables]
\label{cor:Psi-law}
Let $X$ and $Y$ be two square-integrable random variables defined on the same probability space. Then for any $L_\psi$-Lipschitz function $\psi$,
\[
\bigl| \mathbb E[\psi(X)] - \mathbb E[\psi(Y)] \bigr| \le L_\psi \bigl( \mathbb E[|X-Y|^2] \bigr)^{1/2}.
\]
\end{corollary}

\begin{proof}
Set $\mu=\mathcal L(X)$ and $\nu=\mathcal L(Y)$. The previous lemma gives
\[
|\Psi_\psi(\mu)-\Psi_\psi(\nu)| \le L_\psi W_2(\mu,\nu).
\]
Since the coupling $(X,Y)$ is admissible for $W_2$,
\[
W_2^2(\mu,\nu) \le \mathbb E[|X-Y|^2].
\]
We conclude immediately.
\end{proof}

\begin{remark}[Economic interpretation of collective observables]
\label{rem:economic-observables}
$M_t^{\varphi,P}$ is a collective indicator (average confidence, market sentiment). $M_t^{\beta,P}$ represents a perceived dispersion or volatility. $M_t^{\gamma,P}$ is a benchmark or aggregated wealth reference. $M_t^{\delta,P}$ aggregates future costs or anticipated risks. $M_T^{\lambda,P}$ is a social or regulatory performance criterion. The particular case $\varphi(x)=\beta(x)=\gamma(x)=\delta(y)=\lambda(x)=x$ recovers classical models.
\end{remark}

The various ingredients of the model being now defined, we can formulate the asymmetric loss aversion functional which constitutes the core of the variational component of the model.

\subsection{Asymmetric loss aversion functional}

\subsubsection{Reference collective observable}

\begin{definition}[Reference collective observable]
\label{def:Psi}

Let $\psi:\mathbb R^d\longrightarrow \mathbb R$ be a Lipschitz function with constant $L_\psi$.

For any measure $\mu\in\mathcal P_2(\mathbb R^d)$, we define the collective observable

\[
\Psi(\mu) := \int_{\mathbb R^d}\psi(y)\,\mu(dy).
\]

The quantity $\Psi(\mu)$ represents an aggregated indicator of the population of agents. Depending on the choice of $\psi$, it can model an average wealth level, a dispersion measure, a systemic risk indicator, or a market sentiment index. When this observable is used as the reference of the asymmetric loss aversion functional, we write it as \(\Upsilon(\mu)\). Thus \(\Upsilon\) is not an additional object; it is the ALA reference observable among the family of collective observables.
\end{definition}

\begin{remark}[Reduced collective interactions]
\label{rem:reduced-interactions}
The considered model belongs to the class of McKean--Vlasov systems with reduced law dependence (reduced mean-field interactions). The collective interaction is not described by the entire measure $\mu$ but by a finite number of Lipschitz observables obtained in the form $\Psi(\mu) = \int \psi(y)\mu(dy)$. This formulation allows modeling collective behaviors richer than the simple empirical mean while preserving an analytical structure compatible with Wasserstein stability estimates.
\end{remark}

\subsubsection[Functional Phi ALA]{Functional $\Phi_{\mathrm{ALA}}$}

\begin{definition}[Asymmetric loss aversion functional]
\label{def:ALA}

The asymmetric loss aversion functional is defined by

\[
\Phi_{\mathrm{ALA}} : \mathbb R^d\times\mathcal P_2(\mathbb R^d) \longrightarrow \mathbb R
\]

and

\[
\Phi_{\mathrm{ALA}}(x,\mu) = \vartheta\bigl(x - \Upsilon(\mu)\mathbf{1}\bigr) + \frac{\varepsilon}{2}|x|^2 + \frac{\beta}{2}W_2^2(\mu,\mu^\star),
\]

where $\mathbf{1}=(1,\ldots,1)\in\mathbb R^d$ and

\[
\vartheta(z) = \sum_{i=1}^d \vartheta_0(z_i),\qquad
\vartheta_0(r) = \kappa_+ r_+ + \kappa_- r_-,
\]
with $\kappa_- > \kappa_+ > 0$. The operators $(\cdot)_+$ and $(\cdot)_-$ are understood componentwise.

\end{definition}

\begin{remark}[Variational interpretation]
\label{rem:ALA-variational}

The nonsmooth contribution induced by asymmetric loss aversion is described first at the level of the convex functional \(\Phi_{\mathrm{ALA}}\). We define the associated maximal monotone operator by
\[
A_{\mathrm{ALA}}(x,\mu)
:=
\partial_x\Phi_{\mathrm{ALA}}(x,\mu),
\]
where the subdifferential is taken with respect to the state variable \(x\), while the probability measure \(\mu\) acts as a parameter.

In the present work, the dependence on the probability distribution is entirely encoded through the collective observable
\[
r=\Upsilon(\mu).
\]
Consequently, once this observable has been computed, the operator admits the reduced representation
\[
A_{\mathrm{ALA}}(x,\mu)
=
A_{\mathrm{ALA}}(x,r),
\qquad
r=\Upsilon(\mu).
\]
Since \(\Phi_{\mathrm{ALA}}\) is convex with respect to the state variable, this subdifferential can be computed explicitly. Denoting \(r=\Upsilon(\mu)\), one obtains
\[
A_{\mathrm{ALA}}(x,r)
=
\partial\vartheta(x-r\mathbf 1)+\varepsilon x
=
\prod_{i=1}^d \partial\vartheta_0(x_i-r)+\varepsilon x,
\]
where the product is understood as a Cartesian product of one-dimensional subdifferentials and is identified with a subset of \(\mathbb R^d\). Equivalently,
\[
\partial_x\Phi_{\mathrm{ALA}}(x,\mu)
=
A_{\mathrm{ALA}}\bigl(x,\Upsilon(\mu)\bigr).
\]
The operator \(x\mapsto A_{\mathrm{ALA}}(x,r)\) is multi-valued, maximal monotone and \(\varepsilon\)-strongly monotone for every fixed \(r\). These properties will be rigorously proved in Section 4. The asymmetry \(\kappa_->\kappa_+\) reflects the fact that a relative loss is perceived as more important than a relative gain of the same magnitude.

This point of view is fundamental throughout the paper. Although the operator is ultimately evaluated through the observable \(r=\Upsilon(\mu)\), its mathematical definition remains that of the convex subdifferential of the asymmetric loss-aversion functional. This distinction plays a central role in the variational analysis, the Yosida regularization, and the Minty--Brezis identification argument developed below.

\end{remark}

\begin{remark}[Role of the Wasserstein term]
\label{rem:wasserstein-role}

The term $\frac{\beta}{2}W_2^2(\mu,\mu^\star)$ does not depend on the state variable $x$. Consequently, it does not appear in the subdifferential with respect to $x$ :

\[
\partial_x\Phi_{\mathrm{ALA}}(x,\mu)
=
A_{\mathrm{ALA}}\bigl(x,\Upsilon(\mu)\bigr)
=
\partial\vartheta\bigl(x-\Upsilon(\mu)\mathbf 1\bigr)+\varepsilon x.
\]

Its role is solely to penalize distributions that deviate from a reference distribution $\mu^\star$.
\end{remark}

\begin{proposition}[Wasserstein stability of the ALA functional]
\label{prop:Phi-wasserstein}
Let $\mu,\nu\in\mathcal P_2(\mathbb R^d)$. Assume that their second moments are bounded by a common constant $R>0$. Then, for all $x\in\mathbb R^d$, there exists a constant $C_R>0$, depending only on $R$, $\mu^\star$ and the parameters of the functional, such that
\[
\bigl|\Phi_{\mathrm{ALA}}(x,\mu)-\Phi_{\mathrm{ALA}}(x,\nu)\bigr|
\le
C_R\,(1+|x|)\,W_2(\mu,\nu).
\]
In particular, on families of laws with uniformly bounded second moment, the functional $\Phi_{\mathrm{ALA}}$ is stable with respect to the Wasserstein distance.
\end{proposition}

\begin{proof}
The function $\vartheta$ is Lipschitz, with constant controlled by $\max(\kappa_+,\kappa_-)$ and the dimension $d$. By Lemma~\ref{lem:stability-observables},
\[
|\Upsilon(\mu)-\Upsilon(\nu)| \le L_\psi W_2(\mu,\nu).
\]
We deduce that the asymmetric part of $\Phi_{\mathrm{ALA}}$ varies at most linearly in $W_2(\mu,\nu)$. The term $\frac{\varepsilon}{2}|x|^2$ does not depend on the law. Finally,
\[
\bigl|W_2^2(\mu,\mu^\star)-W_2^2(\nu,\mu^\star)\bigr|
\le
\bigl(W_2(\mu,\mu^\star)+W_2(\nu,\mu^\star)\bigr)W_2(\mu,\nu),
\]
and the factor in parentheses is bounded as soon as the second moments of $\mu$ and $\nu$ are. Gathering constants, we obtain the stated estimate.
\end{proof}

\begin{remark}
Proposition~\ref{prop:Phi-wasserstein} establishes the continuity of the functional $\Phi_{\mathrm{ALA}}$ with respect to the measure variable in the Wasserstein topology. We emphasize however that this property alone does not imply the convergence of the graphs of the subdifferentials $\partial_x\Phi_{\mathrm{ALA}}(\cdot,\mu)$. In Section 4, the variational identification will be obtained for a fixed measure $\mu_t^P$, which allows applying directly the classical theory of maximal monotone operators.
\end{remark}

\subsection{Formulation of the selected robust forward--backward variational system}

\begin{remark}[Origin of the variational term]
\label{rem:variational-origin}
The term \(A_{\mathrm{ALA}}(X_t^P,\Upsilon(\mu_t^P))\) introduced in the backward dynamics represents the variational component of the model. It acts as a correction mechanism that penalizes trajectories incompatible with the collective preference described by the asymmetric loss aversion functional. This signal is evaluated at the forward state \(X^P\), not at the backward variable \(Y^P\). This structure constitutes one of the main specificities of the proposed model.
\end{remark}

For $P\in\mathcal{P}$, the selected robust forward--backward variational system is written as follows:
\begin{equation}\label{eq:system_robuste}
\begin{cases}
dX_t^P = b\bigl(t,X_t^P, M_t^{\varphi,P}, u_t\bigr)dt + \sigma\bigl(t,X_t^P, M_t^{\beta,P}, u_t\bigr)dB_t,\\[4pt]
-dY_t^P =
\Bigl[f\bigl(t,X_t^P, M_t^{\gamma,P}, Y_t^P, Z_t^P, M_t^{\delta,P}, u_t\bigr) + \rho\,\Gamma_t^P\Bigr]dt - Z_t^P dB_t,\\[4pt]
X_0^P=\xi,\quad Y_T^P=g\bigl(X_T^P, M_T^{\lambda,P}\bigr),\quad \mu_t^P=\mathcal{L}^P(X_t^P),\\[4pt]
\Gamma_t^P\in A_{\mathrm{ALA}}\bigl(X_t^P,\Upsilon(\mu_t^P)\bigr).
\end{cases}
\end{equation}
The inclusion only concerns the monotone signal \(\Gamma^P\), not the backward variable \(Y^P\). The canonical version used for uniqueness is obtained by taking
\[
\Gamma_t^P=A_{\mathrm{ALA}}^0\bigl(X_t^P,\Upsilon(\mu_t^P)\bigr),
\]
the element of minimal norm in \(A_{\mathrm{ALA}}(X_t^P,\Upsilon(\mu_t^P))\).

\begin{remark}
The system \eqref{eq:system_robuste} is studied under each admissible probability $P\in\mathcal P$ considered separately. The analysis developed in this paper therefore does not require the full theory of 2BSDEs in the sense of Soner--Touzi--Zhang. The results are established uniformly in $P$ and then aggregated at the level of the family $\mathcal P$.
\end{remark}

\subsection{Model coefficients and assumptions}

\subsubsection{Model coefficients}

Throughout the paper, we consider the measurable deterministic coefficients

\[
b : [0,T] \times \mathbb{R}^{d} \times \mathbb{R}^{d_\varphi} \times U \longrightarrow \mathbb R^d,
\]
\[
\sigma : [0,T] \times \mathbb{R}^{d} \times \mathbb{R}^{d_\beta} \times U \longrightarrow \mathbb R^{d \times m},
\]
\[
f : [0,T] \times \mathbb{R}^{d} \times \mathbb{R}^{d_\gamma} \times \mathbb R \times \mathbb R^{m} \times \mathbb{R}^{d_\delta} \times U \longrightarrow \mathbb R,
\]
as well as the terminal cost function
\[
g : \mathbb R^{d} \times \mathbb{R}^{d_\lambda} \longrightarrow \mathbb R.
\]

\subsubsection{Assumptions}

The Lipschitz constants are uniform in $P\in\mathcal{P}$.

\begin{hypothesis}[Initial data]\label{H1}
$\xi$ is $\mathcal{F}_0$-measurable, $\sup_P\mathbb{E}^P[|\xi|^2]<\infty$.
\end{hypothesis}

\begin{hypothesis}[Fourth moment]\label{H1bis}
\[
\sup_{P\in\mathcal P}\mathbb E^P[|\xi|^4]<\infty .
\]
\end{hypothesis}

\begin{hypothesis}[Drift and diffusion]\label{H2}
$b,\sigma$ are measurable, Lipschitz in $(x, m_\varphi, m_\beta)$, and have linear growth in $|x|+|m_\varphi|+|m_\beta|+|u|$, where $m_\varphi\in\R^{d_\varphi}$ and $m_\beta\in\R^{d_\beta}$.
\end{hypothesis}

\begin{hypothesis}[Backward generator]\label{H3}
$f$ is measurable, Lipschitz in $(x, m_\gamma, y, z, m_\delta)$, and has linear growth in $|x|+|m_\gamma|+|y|+|z|+|m_\delta|$, where $m_\gamma\in\R^{d_\gamma}$ and $m_\delta\in\R^{d_\delta}$.
\end{hypothesis}

\begin{hypothesis}[Terminal condition]\label{H4}
$g$ is Lipschitz in $(x, m_\lambda)$, with $|g(x,m_\lambda)|\le C_g(1+|x|+|m_\lambda|)$, where $m_\lambda\in\R^{d_\lambda}$.
\end{hypothesis}

\begin{hypothesis}[Monotonicity of $f$]\label{H5}
$\exists\gamma>0$ such that $\langle f(y,z)-f(y',z), y-y'\rangle\le -\gamma|y-y'|^2$, uniformly.
\end{hypothesis}

\begin{hypothesis}[Fourth-order integrability of the control]\label{Hu4}
\[
\sup_{P\in\mathcal P} \mathbb E^P\left[\int_0^T |u_t|^4 dt\right] < \infty.
\]
\end{hypothesis}

\begin{hypothesis}[Peng--Wu type monotone domination]\label{H6}
There exist constants $\varepsilon>0$, $\gamma>0$, $\rho>0$ such that:
\begin{itemize}
\item the operator $A_{\mathrm{ALA}}$ is strongly monotone with constant $\varepsilon$;
\item the generator $f$ satisfies the monotonicity assumption \ref{H5} with constant $\gamma$;
\item the Lipschitz constants of the coefficients $b$, $\sigma$ and $f$ are sufficiently small relative to the total dissipation induced by $\rho\varepsilon$ and $\gamma$.
\end{itemize}
More precisely, there exists a universal constant $C_{\mathrm{PW}}>0$, depending only on the Peng--Wu energy calculations, such that
\begin{equation}\label{eq:H6}
\rho\varepsilon+\gamma \;>\; C_{\mathrm{PW}}\bigl( L_b^2+L_\sigma^2+L_f^2+1 \bigr).
\end{equation}
This assumption will be used only in the analysis of the coupled backward component (Section 4). The results of Section 3 (existence and uniqueness of the forward component) require only assumptions {\rm(H1)}--{\rm(H2)}.
\end{hypothesis}

\begin{remark}\label{rem:PW-monotonicity}
Assumption \ref{H6} is a classical monotone domination condition in the theory of Peng--Wu type monotone FBSDE. It allows closing the coupled energy estimates obtained after applying It\^o's formula to a functional of the type
\[
|\delta Y_t|^2 + \alpha |\delta X_t|^2 + 2\beta\langle\delta X_t,\delta Y_t\rangle,
\]
and guarantees the absorption of terms arising from the Lipschitz constants of the system coefficients.
\end{remark}

Throughout the paper, generic constants $C$ may vary from line to line. Constants $C_T$ depend on $T$, the coefficients and the parameters of the model, but are independent of $P\in\mathcal P$, $N$ and $\lambda$. The constants $C_{FG}$ and $C_{prop}$ are respectively associated with the Fournier--Guillin theorem and the propagation of chaos estimates.

\section{Existence, uniqueness and stability of the robust forward dynamics}

We analyze the robust forward component by first freezing the collective observables, then reconstructing the law dependence by a fixed point in the Wasserstein space. This strategy provides the existence, uniqueness and stability of the McKean--Vlasov dynamics under constants uniform in $P\in\mathcal P$.

\subsection{Frozen forward equation}

Let $M^\varphi=(M_t^\varphi)_{0\le t\le T}$ and $M^\beta=(M_t^\beta)_{0\le t\le T}$ be given processes, where for each $t$, $M_t^\varphi\in\R^{d_\varphi}$ and $M_t^\beta\in\R^{d_\beta}$ are fixed, with $\sup_t(|M_t^\varphi|+|M_t^\beta|)<\infty$. Under assumptions {\rm(H1)--(H2)}, the coefficients are progressively measurable, Lipschitz in $x$ and have linear growth. The classical existence and uniqueness theorem for Lipschitz SDEs (see \cite{CarmonaDelarue2018I}) then implies the existence of a unique strong solution
\[
X_t^{P,M^\varphi,M^\beta} = \xi + \int_0^t b\bigl(s,X_s^{P,M^\varphi,M^\beta}, M_s^\varphi, u_s\bigr)ds + \int_0^t \sigma\bigl(s,X_s^{P,M^\varphi,M^\beta}, M_s^\beta, u_s\bigr)dB_s \tag{3.1}\label{eq:frozen-forward}
\]
in $L^2(P;C([0,T];\mathbb{R}^d))$. The uniform estimate
\[
\sup_{P\in\mathcal{P}}\mathbb{E}^P\Big[\sup_{0\le t\le T}|X_t^{P,M^\varphi,M^\beta}|^2\Big]
\le C_T\Big(1+\sup_P\mathbb{E}^P[|\xi|^2]+\sup_P\mathbb{E}^P\Big[\int_0^T|u_t|^2dt\Big]+\sup_t(|M_t^\varphi|^2+|M_t^\beta|^2)\Big) \tag{3.2}\label{eq:frozen-estimate}
\]
follows directly from the Lipschitz and linear growth assumptions.

\subsection{Construction of the McKean--Vlasov fixed point}

The analysis of the frozen forward equation constitutes the first step of the construction. To reintroduce the law dependence characteristic of McKean--Vlasov systems, it remains to solve a fixed point equation in the space of probability flows.

For each admissible probability $P\in\mathcal P$, we consider the space
\[
\mathcal M
=
\Bigl\{
\mu=(\mu_t)_{0\le t\le T}
\subset \mathcal P_2(\mathbb R^d)
:
t\mapsto\mu_t
\ \text{continuous for }W_2
\Bigr\},
\]
equipped with the distance
\[
d_{\mathcal M}(\mu,\nu)
=
\sup_{0\le t\le T}
W_2(\mu_t,\nu_t).
\]

\begin{lemma}[Completeness of the space of measure flows]
\label{lem:completeness-M}

The space $(\mathcal M,d_{\mathcal M})$ is a complete metric space.
\end{lemma}

\begin{proof}
Let $(\mu^n)_{n\ge1}$ be a Cauchy sequence in $(\mathcal M,d_{\mathcal M})$. Then, for every $t\in[0,T]$, the sequence $(\mu_t^n)_{n\ge1}$ is Cauchy in $(\mathcal P_2(\mathbb R^d),W_2)$. Since $(\mathcal P_2(\mathbb R^d),W_2)$ is complete, there exists, for each $t\in[0,T]$, a measure $\mu_t\in\mathcal P_2(\mathbb R^d)$ such that $W_2(\mu_t^n,\mu_t)\longrightarrow0$.

It remains to show that the convergence is uniform in time and that the limiting flow is continuous. Since $(\mu^n)$ is Cauchy in $d_{\mathcal M}$, for every $\varepsilon>0$, there exists $N_\varepsilon\ge1$ such that, for all $n,m\ge N_\varepsilon$, $\sup_{0\le t\le T} W_2(\mu_t^n,\mu_t^m) \le \varepsilon$. Letting $m\to\infty$, we obtain, for every $n\ge N_\varepsilon$, $\sup_{0\le t\le T} W_2(\mu_t^n,\mu_t) \le \varepsilon$. Thus $d_{\mathcal M}(\mu^n,\mu)\longrightarrow0$.

It remains to verify that $t\mapsto\mu_t$ is continuous for $W_2$. Let $s,t\in[0,T]$. For every $n$,
\[
W_2(\mu_t,\mu_s)
\le
W_2(\mu_t,\mu_t^n)
+
W_2(\mu_t^n,\mu_s^n)
+
W_2(\mu_s^n,\mu_s).
\]
Fix $\varepsilon>0$. First choose $n$ large enough so that $\sup_{r\in[0,T]}W_2(\mu_r^n,\mu_r) \le \varepsilon$. Since $\mu^n\in\mathcal M$, the map $t\mapsto\mu_t^n$ is continuous for $W_2$. There exists therefore $\delta>0$ such that, if $|t-s|<\delta$, then $W_2(\mu_t^n,\mu_s^n)\le\varepsilon$. Consequently, $W_2(\mu_t,\mu_s)\le 3\varepsilon$. The continuity of $t\mapsto\mu_t$ is proved. Thus $\mu\in\mathcal M$, and $(\mathcal M,d_{\mathcal M})$ is complete.
\end{proof}

Let now $\mu\in\mathcal M$ be a flow of probabilities. We define the associated collective observables
\[
M_t^{\varphi,\mu}
=
\Psi_\varphi(\mu_t)
=
\int_{\mathbb R^d}\varphi(x)\,\mu_t(dx),
\qquad
M_t^{\beta,\mu}
=
\Psi_\beta(\mu_t)
=
\int_{\mathbb R^d}\beta(x)\,\mu_t(dx).
\]

The solution of the frozen forward equation \eqref{eq:frozen-forward} associated with these observables will be denoted $X^{P,\mu}$. We then define the map
\[
\mathcal T^P:\mathcal M\longrightarrow\mathcal M
\]
by
\[
\mathcal T^P(\mu)_t
=
\mathcal L^P(X_t^{P,\mu}),
\qquad
0\le t\le T.
\]

Thus, a fixed point of \(\mathcal T^P\) corresponds exactly to a solution of the robust McKean--Vlasov equation.

\begin{lemma}[Invariance of the space of measure flows]
\label{lem:invariance}

Assume that {\rm(H1)} and {\rm(H2)} are satisfied. Then, for every $P\in\mathcal P$,
\[
\mathcal T^P(\mathcal M) \subset \mathcal M.
\]

\end{lemma}

\begin{proof}
Let $\mu\in\mathcal{M}$ and $P\in\mathcal{P}$ be fixed. Denote $X = X^{P,\mu}$ the solution of \eqref{eq:frozen-forward} with $M_t^\varphi = \Psi_\varphi(\mu_t)$ and $M_t^\beta = \Psi_\beta(\mu_t)$.

\medskip
\noindent\textbf{Verification of membership in $\mathcal{P}_2(\mathbb{R}^d)$.}
For every $t\in[0,T]$, $\mathcal T^P(\mu)_t = \mathcal{L}^P(X_t)$ is a probability measure on $\mathbb{R}^d$. By estimate \eqref{eq:frozen-estimate},
\[
\int |x|^2 \,\mathcal T^P(\mu)_t(dx) = \mathbb{E}^P[|X_t|^2] \le \mathbb{E}^P\Big[\sup_{0\le s\le T}|X_s|^2\Big] \le R_0^2 < \infty.
\]
Thus, $\mathcal T^P(\mu)_t \in \mathcal{P}_2(\mathbb{R}^d)$ for every $t$.

\medskip
\noindent\textbf{Verification of time continuity for $W_2$.}
Let $s,t\in[0,T]$. By the Wasserstein inequality and the $L^2$-continuity of trajectories,
\[
W_2^2(\mathcal T^P(\mu)_t, \mathcal T^P(\mu)_s) \le \mathbb{E}^P[|X_t - X_s|^2].
\]
As $t\to s$, we have $\mathbb{E}^P[|X_t - X_s|^2] \to 0$ because $X$ is a solution of an SDE with Lipschitz coefficients. Continuity is therefore established.

\medskip
\noindent\textbf{Uniform control of $W_2(\mathcal T^P(\mu)_t,\delta_0)$.}
By estimate \eqref{eq:frozen-estimate},
\[
W_2^2(\mathcal T^P(\mu)_t,\delta_0) \le \mathbb{E}^P[|X_t|^2] \le R_0^2,
\]
so $\sup_t W_2(\mathcal T^P(\mu)_t,\delta_0) \le R_0$.

\medskip
\noindent\textbf{Conclusion.}
The three points show that $\mathcal T^P(\mu) \in \mathcal{M}$, which completes the proof.
\end{proof}

The following result constitutes the decisive point of the fixed point construction.

\begin{proposition}[Uniform contraction]
\label{prop:contraction-uniforme}

Assume that assumptions {\rm(H1)}--{\rm(H2)} are satisfied. Then there exists a time horizon $T_0>0$ and a constant $\kappa\in(0,1)$, independent of $P\in\mathcal P$, such that
\[
d_{\mathcal M}\bigl( \mathcal T^P(\mu), \mathcal T^P(\nu) \bigr) \le \kappa\, d_{\mathcal M}(\mu,\nu)
\]
for all $\mu,\nu\in\mathcal M$ on the interval $[0,T_0]$.

\end{proposition}

\begin{proof}
Let $\mu, \nu \in \mathcal{M}$ and $P \in \mathcal{P}$ be fixed. Denote $X = X^{P,\mu}$ and $\bar X = X^{P,\nu}$ the two solutions of the frozen forward equation \eqref{eq:frozen-forward} associated respectively with the flows $\mu$ and $\nu$, with observables $M_t^{\varphi,\mu}=\Psi_\varphi(\mu_t)$, $M_t^{\beta,\mu}=\Psi_\beta(\mu_t)$, and similarly for $\nu$. Their differences satisfy
\[
\begin{aligned}
\delta X_t
&:= X_t-\bar X_t\\
&= \int_0^t
\bigl[
b(s,X_s,M_s^{\varphi,\mu},u_s)
-b(s,\bar X_s,M_s^{\varphi,\nu},u_s)
\bigr]\,ds\\
&\quad
+ \int_0^t
\bigl[
\sigma(s,X_s,M_s^{\beta,\mu},u_s)
-\sigma(s,\bar X_s,M_s^{\beta,\nu},u_s)
\bigr]\,dB_s.
\end{aligned}
\]

\medskip
\noindent\textbf{Energy inequality.}
We apply It\^o's formula to $|\delta X_t|^2$ and take expectation under $P$. The stochastic term yields, via the Burkholder--Davis--Gundy inequality uniform in $P$ (see \cite{SonerTouziZhang2011}, Lemma 3.1):
\[
\mathbb{E}^P\!\left[\sup_{0 \le r \le t}|\delta X_r|^2\right]
\;\le\; C_{\mathrm{BDG}}\,\mathbb{E}^P\!\left[\int_0^t
\bigl(|\delta b_s|^2 + |\delta\sigma_s|^2\bigr)\,ds\right],
\]
where the constant $C_{\mathrm{BDG}}$ depends only on the dimension and is uniform over the non-dominated family $\mathcal P$. By (H2), $b$ and $\sigma$ are $L$-Lipschitz. We therefore obtain
\[
|\delta b_s|^2 + |\delta\sigma_s|^2
\;\le\; 2L^2\bigl(|\delta X_s|^2 + |M_s^{\varphi,\mu}-M_s^{\varphi,\nu}|^2 + |M_s^{\beta,\mu}-M_s^{\beta,\nu}|^2\bigr).
\]

By Lemma~\ref{lem:stability-observables},
\[
|M_s^{\varphi,\mu}-M_s^{\varphi,\nu}| \le L_\varphi W_2(\mu_s,\nu_s),\qquad
|M_s^{\beta,\mu}-M_s^{\beta,\nu}| \le L_\beta W_2(\mu_s,\nu_s).
\]

Thus,
\[
|\delta b_s|^2 + |\delta\sigma_s|^2 \le 2L^2\bigl(|\delta X_s|^2 + (L_\varphi^2+L_\beta^2) W_2^2(\mu_s,\nu_s)\bigr).
\]

Substituting and setting $\Phi(t) = \mathbb{E}^P[\sup_{0 \le r \le t} |\delta X_r|^2]$,
\[
\Phi(t) \;\le\; 2L^2 C_{\mathrm{BDG}} \int_0^t \Phi(s)\,ds
            + 2L^2 C_{\mathrm{BDG}} (L_\varphi^2+L_\beta^2) \int_0^t \sup_{0\le r\le s} W_2^2(\mu_r,\nu_r)\,ds. \tag{3.3}\label{eq:forward-contraction-integral}
\]

\medskip
\noindent\textbf{Gronwall inequality.}
Applying Gronwall's lemma to \eqref{eq:forward-contraction-integral} on $[0, T_0]$ gives
\[
\Phi(T_0) \;\le\; 2L^2 C_{\mathrm{BDG}} (L_\varphi^2+L_\beta^2) T_0\, e^{2L^2 C_{\mathrm{BDG}} T_0}
\cdot \sup_{0 \le t \le T_0} W_2^2(\mu_t,\nu_t). \tag{3.4}\label{eq:forward-contraction-gronwall}
\]

\medskip
\noindent\textbf{Passage to the Wasserstein distance.}
For every $t \in [0, T_0]$, the natural coupling $(X_t, \bar X_t)$ is admissible for $W_2(\mathcal{L}^P(X_t), \mathcal{L}^P(\bar X_t))$, hence
\[
W_2^2\!\bigl(\mathcal{L}^P(X^{P,\mu}_t),\,\mathcal{L}^P(X^{P,\nu}_t)\bigr)
\;\le\; \mathbb{E}^P[|\delta X_t|^2]
\;\le\; \mathbb{E}^P\!\left[\sup_{0 \le r \le t}|\delta X_r|^2\right].
\]

Taking the supremum over $t \in [0, T_0]$, we conclude via \eqref{eq:forward-contraction-gronwall}:
\[
\sup_{0 \le t \le T_0}
W_2^2\!\bigl(\mathcal{L}^P(X^{P,\mu}_t),\,\mathcal{L}^P(X^{P,\nu}_t)\bigr)
\;\le\; \kappa(T_0)\cdot \sup_{0 \le t \le T_0} W_2^2(\mu_t, \nu_t),
\]
with
\[
\kappa(T_0) := 2L^2 C_{\mathrm{BDG}} (L_\varphi^2+L_\beta^2) T_0\, e^{2L^2 C_{\mathrm{BDG}} T_0}.
\]

The function $T_0 \mapsto \kappa(T_0)$ is continuous, zero at $T_0 = 0$, and strictly increasing. There exists therefore $T_0 > 0$ small enough, \textbf{depending only on $L$, $L_\varphi$, $L_\beta$ and $C_{\mathrm{BDG}}$} (and not on $P$), such that $\kappa(T_0) < 1$. The constant $\kappa = \kappa(T_0) \in (0,1)$ is indeed \textbf{uniform in $P \in \mathcal{P}$}.
\end{proof}

\subsection{Existence and global uniqueness of the robust forward dynamics}

\begin{lemma}[Pasting of local fixed points]
\label{lem:pasting}

Let \(T_0>0\) be such that the McKean--Vlasov map \(\mathcal T^P\) is contractive on every interval of length \(T_0\), with a contraction constant independent of \(P\in\mathcal P\). Assume that, for each \(k=0,\ldots,n-1\), there exists a unique local fixed point \(\mu^{(k)}\) of \(\mathcal T^P\) on the interval
\[
I_k=[kT_0,(k+1)T_0]\cap[0,T],
\]
constructed with the initial condition transmitted by the previous solution. Then the local flows \(\mu^{(k)}\) paste together into a unique global flow \(\mu^P\in\mathcal M\) on \([0,T]\), which is a fixed point of \(\mathcal T^P\).

\end{lemma}

\begin{proof}
We construct the global flow \(\mu^P\) piecewise. On the first interval \(I_0=[0,T_0]\), the contraction of \(\mathcal T^P\) provides a unique local fixed point \(\mu^{(0)}\). Let \(X^{(0)}\) be the corresponding forward solution.

Assume now that the solution is constructed up to time $kT_0$. The a priori estimate of the forward equation gives $\mathbb E^P[|X_{kT_0}^{(k-1)}|^2]<\infty$. We can therefore restart the equation on the interval $I_k=[kT_0,(k+1)T_0]\cap[0,T]$ with initial condition $X_{kT_0}^{(k-1)}$. The coefficients keep the same Lipschitz and growth constants. The local contraction Proposition therefore applies again on $I_k$, with the same contraction constant.

We thus obtain a unique local fixed point $\mu^{(k)}$ on $I_k$. By construction, the initial condition of the problem on $I_k$ is exactly the terminal value of the solution constructed on $I_{k-1}$. Consequently,
\[
\mu^{(k)}_{kT_0} = \mu^{(k-1)}_{kT_0}.
\]

The local flows therefore paste together without discontinuity at the junction times.

We then define
\[
\mu_t^P = \mu_t^{(k)} \quad \text{if } t\in I_k.
\]

Since each $t\mapsto\mu_t^{(k)}$ is continuous for $W_2$, and since the values coincide at the junction points, the global flow $t\mapsto\mu_t^P$ is continuous for $W_2$. Thus $\mu^P\in\mathcal M$.

It remains to verify that $\mu^P$ is indeed a global fixed point. On each interval $I_k$, by construction,
\[
\mathcal T^P(\mu^P)_t = \mu_t^P, \qquad t\in I_k.
\]

Since the intervals $I_k$ cover $[0,T]$, we obtain
\[
\mathcal T^P(\mu^P)=\mu^P \quad \text{on }[0,T].
\]

Finally, global uniqueness follows from local uniqueness. Indeed, if $\nu^P$ is another global fixed point, then its restrictions to $I_0$ and those of $\mu^P$ are two fixed points of the same local contraction. They therefore coincide on $I_0$. In particular, the two solutions have the same value at time $T_0$. Restarting on $I_1$, local uniqueness gives equality on $I_1$. Iterating the argument, we obtain $\nu_t^P=\mu_t^P$, $0\le t\le T$.

The proof is complete.
\end{proof}

\begin{theorem}[Existence and uniqueness of the robust forward dynamics]
\label{thm:fixed-point-global}

Assume {\rm(H1)--(H2)}. Then, for each probability $P\in\mathcal P$, there exists a unique flow $\mu^P\in\mathcal M$ and a unique process $X^P$ solution of the McKean--Vlasov forward equation
\[
X_t^P = \xi + \int_0^t b\bigl(s,X_s^P, M_s^{\varphi,P}, u_s\bigr)ds + \int_0^t \sigma\bigl(s,X_s^P, M_s^{\beta,P}, u_s\bigr)dB_s,
\]
with $M_t^{\varphi,P} = \Psi_\varphi(\mu_t^P)$, $M_t^{\beta,P} = \Psi_\beta(\mu_t^P)$ and $\mu_t^P = \mathcal{L}^P(X_t^P)$.

\end{theorem}

\begin{proof}

The proof relies on the application of Banach's fixed point theorem on sufficiently short time subintervals, followed by a continuation argument.

\medskip
\noindent
\textbf{Local existence and uniqueness.}

By Lemma~\ref{lem:invariance}, the map $\mathcal T^P:\mathcal M\longrightarrow \mathcal M$ is well-defined. By Proposition~\ref{prop:contraction-uniforme}, there exists a time horizon $T_0>0$ and a constant $\kappa\in(0,1)$ such that
\[
d_{\mathcal M}\!\left( \mathcal T^P(\mu), \mathcal T^P(\nu) \right) \le \kappa\, d_{\mathcal M}(\mu,\nu)
\]
for all $\mu,\nu\in\mathcal M$ on the interval $[0,T_0]$. Since $(\mathcal M,d_{\mathcal M})$ is a complete metric space (Lemma~\ref{lem:completeness-M}), Banach's fixed point theorem implies the existence of a unique flow $\mu^{(1)} \in \mathcal M$ such that $\mathcal T^P(\mu^{(1)}) = \mu^{(1)}$ on $[0,T_0]$. We then set $X^{(1)} = X^{P,\mu^{(1)}}$.

\medskip
\noindent
\textbf{Extension of the solution.}

Let $X^{(1)} := X^{P,\mu^{(1)}}$. The uniform estimate obtained in \eqref{eq:frozen-estimate} ensures that $\sup_{0\le t\le T_0} \mathbb E^P[|X_t^{(1)}|^2] < \infty$. In particular, $X_{T_0}^{(1)} \in L^2(P)$.

We can therefore consider the same problem on the interval $[T_0,2T_0]$ with initial condition $X_{T_0}^{(1)}$. Assumptions {\rm(H1)}--{\rm(H2)} remain valid with this new initial condition. The Lipschitz constants are unchanged.

Consequently, Proposition~\ref{prop:contraction-uniforme} applies again on $[T_0,2T_0]$. Banach's theorem then provides a unique fixed point $\mu^{(2)}$ on this interval. Moreover, $\mu^{(2)}_{T_0} = \mu^{(1)}_{T_0}$, since both constructions use the same initial value $X_{T_0}^{(1)}$. The two solutions therefore paste together naturally.

\medskip
\noindent
\textbf{Iteration of the argument.}

We repeat the previous argument on the intervals $[kT_0,(k+1)T_0]$, $k=0,1,\dots,N-1$, where $N = \lceil T/T_0 \rceil$. At each iteration, the a priori estimate \eqref{eq:frozen-estimate} guarantees that the initial condition transmitted to the following subinterval belongs to $L^2(P)$. Banach's theorem then provides a unique local fixed point.

By successive pasting (Lemma~\ref{lem:pasting}), we obtain a flow $\mu^P = (\mu_t^P)_{0\le t\le T}$ defined on the whole interval $[0,T]$ and satisfying $\mathcal T^P(\mu^P) = \mu^P$. Thus, $\mu_t^P = \mathcal L^P(X_t^{P,\mu^P})$ for every $t\in[0,T]$. Existence is proved.

\medskip
\noindent
\textbf{Global uniqueness.}

Let \(\mu^P\) and \(\nu^P\) be two fixed points of \(\mathcal T^P\) on \([0,T]\), and let \(X^P\) and \(\tilde X^P\) be the corresponding forward processes. Their restrictions to \([0,T_0]\) are two fixed points of the same contraction. Uniqueness given by Banach's theorem implies \(\mu_t^P = \nu_t^P\) for \(0\le t\le T_0\). In particular, \(X_{T_0}^P = \tilde X_{T_0}^P\).

The two problems restarted on $[T_0,2T_0]$ therefore have the same initial condition. Local uniqueness then gives $\mu_t^P = \nu_t^P$ for $T_0\le t\le 2T_0$. Repeating this argument on each of the subintervals $[kT_0,(k+1)T_0]$, we obtain $\mu_t^P = \nu_t^P$ for every $t\in[0,T]$, and consequently $X^P = \tilde X^P$. Global uniqueness is established.

\medskip

The theorem is proved.

\end{proof}

\subsection{Uniform estimate for the forward solution}

\begin{proposition}[Uniform estimate of the forward solution]
\label{prop:fixed-point-P-uniform}

For each $P\in\mathcal P$, let $(X^P,\mu^P)$ be the unique solution given by Theorem~\ref{thm:fixed-point-global}. Then
\[
\sup_{P\in\mathcal P}\;\mathbb E^P\!\left[\sup_{0\le t\le T}\bigl|X^P_t\bigr|^2\right] \le C_T \;<\; \infty,
\]
where $C_T$ is independent of $P$.
\end{proposition}

\begin{proof}
Let $X^P$ denote the forward solution. It satisfies equation \eqref{eq:frozen-forward} with observables $M_t^{\varphi,\mu^P} = \Psi_\varphi(\mu^P_t)$ and $M_t^{\beta,\mu^P} = \Psi_\beta(\mu^P_t)$.

By It\^o's formula and the Burkholder--Davis--Gundy inequality,
\[
\begin{aligned}
\mathbb{E}^P\!\left[\sup_{0\le s\le t}|X^P_s|^2\right]
&\le C_T\Bigg(
1+\mathbb{E}^P[|\xi|^2]
+\int_0^t \mathbb{E}^P[|X^P_s|^2]\,ds\\
&\qquad
+\int_0^t |M_s^{\varphi,\mu^P}|^2\,ds
+\int_0^t |M_s^{\beta,\mu^P}|^2\,ds
+\int_0^t |u_s|^2\,ds
\Bigg).
\end{aligned}
\]

By Lemma~\ref{lem:stability-observables} and the linear growth of $\varphi,\beta$,
\[
|M_t^{\varphi,\mu^P}| \le |\varphi(0)| + L_\varphi \mathbb{E}^P[|X^P_t|] \le C_T(1 + \mathbb{E}^P[|X^P_t|^2]^{1/2}),
\]
and similarly for $\beta$. Thus,
\[
|M_t^{\varphi,\mu^P}|^2 + |M_t^{\beta,\mu^P}|^2 \le C_T\left(1 + \mathbb{E}^P[|X^P_t|^2]\right).
\]

Setting $\Phi(t) = \mathbb{E}^P[\sup_{0\le s\le t}|X^P_s|^2]$, we obtain
\[
\Phi(t) \le C_T\left(1 + \mathbb{E}^P[|\xi|^2] + \int_0^t \Phi(s)ds + \int_0^T |u_s|^2 ds\right).
\]

Gronwall's lemma then gives
\[
\Phi(T) \le C_T\left(1 + \mathbb{E}^P[|\xi|^2] + \int_0^T |u_s|^2 ds\right) e^{C_T T}.
\]

Taking the supremum over $P\in\mathcal P$, we obtain
\[
\sup_{P\in\mathcal P}\Phi(T) \le C_T\left(1 + \sup_P\mathbb{E}^P[|\xi|^2] + \sup_P\mathbb{E}^P\left[\int_0^T |u_s|^2 ds\right]\right).
\]

The announced uniform estimate is proved.
\end{proof}

\begin{remark}\label{rem:family-muP}
The family $\{\mu^P\}_{P\in\mathcal{P}}$ depends on $P$ (no dominating measure). All estimates are uniform in $P$.
\end{remark}

\subsection{Robust stability of the forward dynamics}

This subsection establishes the continuous dependence of the McKean--Vlasov forward solution with respect to initial data and control.

\subsubsection{Fundamental stability inequality}

Consider two systems associated respectively with data $(\xi,u)$ and $(\bar\xi,\bar u)$. For $P\in\mathcal{P}$, denote $(X^P,\mu^P)$ and $(\bar X^P,\bar\mu^P)$ the solutions of Theorem~\ref{thm:fixed-point-global}. Set
\[
\Delta X_t := X_t^P - \bar X_t^P,\quad \Delta\xi := \xi - \bar\xi,\quad \Delta u_t := u_t - \bar u_t.
\]

By difference of the forward equations,
\[
\Delta X_t = \Delta\xi + \int_0^t \Delta b_s\,ds + \int_0^t \Delta\sigma_s\,dB_s,
\]
with
\[
\Delta b_s = b\bigl(s,X_s^P, \Psi_\varphi(\mu_s^P), u_s\bigr) - b\bigl(s,\bar X_s^P, \Psi_\varphi(\bar\mu_s^P), \bar u_s\bigr),
\]
and $\Delta\sigma_s$ defined similarly.

Applying successively Jensen's inequality, the uniform Burkholder--Davis--Gundy inequality (see \cite{SonerTouziZhang2011}) and the Lipschitz assumption (H2), we obtain, for each $P\in\mathcal{P}$,
\[
\begin{aligned}
\mathbb{E}^P\Big[\sup_{0\le r\le t}|\Delta X_r|^2\Big]
&\le C_T\mathbb{E}^P[|\Delta\xi|^2]
+ C_T\int_0^t\mathbb{E}^P\Big[\sup_{0\le q\le s}|\Delta X_q|^2\Big]ds \\
&\quad + C_T\int_0^t W_2^2(\mu_s^P,\bar\mu_s^P)ds
+ C_T\int_0^t\mathbb{E}^P[|\Delta u_s|^2]ds,
\end{aligned}
\tag{3.5}\label{eq:forward-stability-basic}
\]
where the constant $C_T$ is independent of $P$.

By the Wasserstein inequality, the canonical coupling gives
\[
W_2^2(\mu_t^P,\bar\mu_t^P) \le \mathbb{E}^P[|X_t^P-\bar X_t^P|^2] \le \mathbb{E}^P\Big[\sup_{0\le r\le t}|\Delta X_r|^2\Big]. \tag{3.6}\label{eq:forward-stability-wasserstein}
\]

Define the function
\[
\Phi(t) := \sup_{P\in\mathcal{P}}\mathbb{E}^P\Big[\sup_{0\le r\le t}|\Delta X_r|^2\Big].
\]

Taking the supremum over $P$ in \eqref{eq:forward-stability-basic} and using \eqref{eq:forward-stability-wasserstein}, we obtain
\[
\Phi(t) \le C_T\|\Delta\xi\|_{L^2_{\mathcal P}}^2 + 2C_T\int_0^t\Phi(s)\,ds + C_T\|\Delta u\|_{\mathcal H^2_{\mathcal P}}^2. \tag{3.7}\label{eq:forward-stability-gronwall}
\]

\subsubsection{Robust stability theorem}

\begin{theorem}[Robust stability of the forward dynamics]
\label{thm:stability-robust}

Under assumptions {\rm(H1)}--{\rm(H2)} of Section 2 concerning the forward coefficients and collective observables, there exists a constant $C_T>0$, depending only on $T$, the Lipschitz constants of the coefficients and the uniform Burkholder--Davis--Gundy constant, such that
\[
\sup_{P\in\mathcal P} \mathbb E^P\Bigl[ \sup_{0\le t\le T} |X_t^P-\bar X_t^P|^2 \Bigr] \le C_T \Bigl( \|\Delta\xi\|_{L^2_{\mathcal P}}^2 + \|\Delta u\|_{\mathcal H^2_{\mathcal P}}^2 \Bigr).
\]

In particular, the map associating to each datum $(\xi,u)$ the corresponding robust forward solution is Lipschitz.

\end{theorem}

\begin{proof}
Define $\Phi(t)$ as above. From estimate \eqref{eq:forward-stability-gronwall}, for all $t\in[0,T]$:
\[
\Phi(t) \le C_T \|\Delta\xi\|_{L^2_{\mathcal P}}^2 + 2C_T \int_0^t \Phi(s)\,ds + C_T \|\Delta u\|_{\mathcal H^2_{\mathcal P}}^2.
\]

Gronwall's lemma then gives
\[
\Phi(t) \le e^{2C_T t} \Bigl( C_T \|\Delta\xi\|_{L^2_{\mathcal P}}^2 + C_T \|\Delta u\|_{\mathcal H^2_{\mathcal P}}^2 \Bigr).
\]

For $t = T$, we obtain
\[
\Phi(T) \le C_T \Bigl( \|\Delta\xi\|_{L^2_{\mathcal P}}^2 + \|\Delta u\|_{\mathcal H^2_{\mathcal P}}^2 \Bigr),
\]
where $C_T$ denotes, as usual, a generic constant $C_T := C_T e^{2C_T T}$. By definition of $\Phi(T)$, the result follows.
\end{proof}

\begin{corollary}[Uniform stability of marginal laws]
\label{cor:stability-laws}

Under the assumptions of Theorem~\ref{thm:stability-robust}, there exists a constant $C_T>0$ such that
\[
\sup_{0\le t\le T} \sup_{P\in\mathcal P} W_2(\mu_t^P,\bar\mu_t^P) \le C_T \Bigl( \|\Delta\xi\|_{L^2_{\mathcal P}} + \|\Delta u\|_{\mathcal H^2_{\mathcal P}} \Bigr).
\]

\end{corollary}

\begin{proof}
For every $P \in \mathcal P$ and every $t\in[0,T]$,
\[
W_2^2(\mu_t^P, \bar\mu_t^P) \le \mathbb E^P\bigl[ |X_t^P - \bar X_t^P|^2 \bigr] \le \mathbb E^P\Bigl[ \sup_{0 \le s \le T} |X_s^P - \bar X_s^P|^2 \Bigr].
\]
Taking suprema and applying Theorem~\ref{thm:stability-robust} yields the result.
\end{proof}

\begin{corollary}[Stability of collective observables]
\label{cor:stability-observables}

Under the assumptions of Theorem~\ref{thm:stability-robust}, there exists a constant $C_T>0$ such that
\[
\sup_{0\le t\le T} \sup_{P\in\mathcal P} \left| \Psi_\varphi(\mu_t^P) - \Psi_\varphi(\bar\mu_t^P) \right| \le C_T \Bigl( \|\Delta\xi\|_{L^2_{\mathcal P}} + \|\Delta u\|_{\mathcal H^2_{\mathcal P}} \Bigr),
\]
and similarly for $\Psi_\beta$, $\Psi_\gamma$ and $\Psi_\delta$.

\end{corollary}

\begin{proof}
By Lemma~\ref{lem:stability-observables},
\[
\big| \Psi_\varphi(\mu_t^P) - \Psi_\varphi(\bar\mu_t^P) \big| \le L_\varphi W_2(\mu_t^P,\bar\mu_t^P).
\]
The result then follows from Corollary~\ref{cor:stability-laws}.
\end{proof}

\begin{remark}
The collective observables thus inherit the Wasserstein stability of the marginal laws. This property will allow controlling the mean-field terms appearing in the backward generator and in the terminal condition.
\end{remark}

\section{Variational analysis of the selected robust backward component}

We now turn to the backward component of the system. Unlike a classical BSVI, the multi-valued monotone operator arising from the ALA functional is evaluated at the forward state \(X^P\), not at the backward variable \(Y^P\). The backward component is therefore treated as a selected backward variational system driven by an exogenous monotone signal. The analysis relies on variational techniques, Yosida regularization and monotonicity arguments for maximal monotone operators.

The delicate point is the non-smooth term arising from $\Phi_{\mathrm{ALA}}$: it requires temporarily replacing the multi-valued signal by its Yosida approximations, then identifying the limit using maximal monotonicity. This section establishes the existence, canonical uniqueness and stability of the selected robust backward variational system.

\subsection{Asymmetric loss aversion functional}

Recall the definition of $\Phi_{\mathrm{ALA}}$ given in Section 2. Let $\chi:\R^d\to\R$ be an $L_\chi$-Lipschitz function (denoted $\psi$ in Section 2). We define the collective observable
\[
\Upsilon(\mu) = \int_{\R^d} \chi(y)\,\mu(dy).
\]

The asymmetric loss aversion functional is then
\[
\Phi_{\mathrm{ALA}}(x,\mu) = \vartheta\bigl(x - \Upsilon(\mu)\mathbf{1}\bigr) + \frac{\varepsilon}{2}|x|^2 + \frac{\beta}{2}W_2^2(\mu,\mu^\star),
\]
where $\mathbf{1}=(1,\ldots,1)\in\mathbb R^d$, $\vartheta:\R^d\to\R$ is defined componentwise by
\[
\vartheta(z) = \sum_{i=1}^d \vartheta_0(z_i),\qquad \vartheta_0(r) = \kappa_+ r_+ + \kappa_- r_-,
\]
with $\kappa_- > \kappa_+ > 0$.

This functional is convex and non-differentiable on the hyperplane \(\{x=\Upsilon(\mu)\mathbf 1\}\). Since the Wasserstein term does not depend on \(x\), the subdifferential with respect to \(x\) depends on the law only through the scalar reference \(r=\Upsilon(\mu)\). For \(r\in\mathbb R\), we define the reduced ALA operator by
\[
A_{\mathrm{ALA}}(x,r)
:=
\prod_{i=1}^d \partial\vartheta_0(x_i-r)+\varepsilon x.
\]
Thus
\[
\partial_x\Phi_{\mathrm{ALA}}(x,\mu)
=
A_{\mathrm{ALA}}\bigl(x,\Upsilon(\mu)\bigr),
\]
where $\partial\vartheta_0$ is the subdifferential of the scalar convex function $\vartheta_0$:
\[
\partial\vartheta_0(r) = 
\begin{cases}
\{\kappa_+\}, & r > 0,\\
[-\kappa_-, \kappa_+], & r = 0,\\
\{-\kappa_-\}, & r < 0.
\end{cases}
\]

Thus, for every fixed \(r\in\mathbb R\), the map \(x\mapsto A_{\mathrm{ALA}}(x,r)\) is a maximal monotone multi-valued operator.

\begin{lemma}[Properties of $A_{\mathrm{ALA}}$]
\label{lem:ALA-properties}
For every fixed \(r\in\mathbb R\), the operator \(A_{\mathrm{ALA}}(\cdot,r):\mathbb R^d\rightrightarrows\mathbb R^d\) satisfies:

\begin{enumerate}
\item \textbf{Linear growth:} For every \(p\in A_{\mathrm{ALA}}(x,r)\),
\[
|p| \le \max(\kappa_+,\kappa_-)\sqrt{d} + \varepsilon|x|.
\]

\item \textbf{Strong monotonicity:} For all \(x,x'\in\mathbb R^d\), \(p\in A_{\mathrm{ALA}}(x,r)\), \(p'\in A_{\mathrm{ALA}}(x',r)\),
\[
\langle p-p',\, x-x'\rangle \ge \varepsilon |x-x'|^2.
\]

\item \textbf{Maximality:} For every \(r\in\mathbb R\), the operator \(x\mapsto A_{\mathrm{ALA}}(x,r)\) is maximal monotone on \(\mathbb R^d\).
\end{enumerate}
\end{lemma}

\begin{proof}
The linear growth follows from the bound \( |\theta|\le\max(\kappa_+,\kappa_-)\) for every \(\theta\in\partial\vartheta_0\) and the fact that the vector \((\theta_1,\ldots,\theta_d)\) has at most \(\sqrt d\) times this bound. For strong monotonicity, we write componentwise. For each \(i=1,\ldots,d\), let \(\theta_i\in\partial\vartheta_0(x_i-r)\) and \(\theta_i'\in\partial\vartheta_0(x_i'-r)\). The monotonicity of \(\partial\vartheta_0\) gives \((\theta_i-\theta_i')(x_i-x_i')\ge 0\). Summing over \(i\),
\[
\langle p-p', x-x'\rangle = \sum_{i=1}^d (\theta_i-\theta_i')(x_i-x_i') + \varepsilon|x-x'|^2 \ge \varepsilon|x-x'|^2.
\]
Maximality is a classical result of Rockafellar: the subdifferential of a proper lower semicontinuous convex function is maximal monotone.
\end{proof}

\subsection{Yosida regularization}

Since the operator is multi-valued, we replace it by a regularized version that preserves its essential properties.

For \(\lambda>0\) and \(r\in\mathbb R\) fixed, we define the resolvent operator
\[
J_\lambda^r = (I + \lambda A_{\mathrm{ALA}}(\cdot,r))^{-1}.
\]
The Yosida approximation is then
\[
A^\lambda_{\mathrm{ALA}}(x,r) = \frac1\lambda\bigl(x - J_\lambda^r(x)\bigr).
\]

\begin{remark}[Raw Yosida approximation and later smoothing]
In the present well-posedness article, the raw Yosida approximation \(A^\lambda_{\mathrm{ALA}}\) is sufficient, because the argument only uses Lipschitz continuity, monotonicity, uniform estimates and graph convergence. In the subsequent control article, where differentiability of the regularized dynamics is needed to derive a Pontryagin maximum principle, one introduces a smoothed approximation \(A_{\mathrm{ALA}}^{\lambda,\varepsilon}\). Thus, \(A^\lambda_{\mathrm{ALA}}\) here should be understood as the raw Yosida approximation, while the smoothed object belongs to the differentiability analysis of the control problem.
\end{remark}

\begin{lemma}[Fundamental properties of the Yosida approximation]
\label{lem:yosida-properties}
For every \(\lambda>0\), \(r\in\mathbb R\) and \(x,y\in\mathbb R^d\), we have:

\begin{enumerate}
\item \textbf{Lipschitz continuity:}
\[
|A^\lambda_{\mathrm{ALA}}(x,r) - A^\lambda_{\mathrm{ALA}}(y,r)| \le \frac1\lambda |x-y|.
\]

\item \textbf{Regularized strong monotonicity:}
\[
\left\langle A^\lambda_{\mathrm{ALA}}(x,r) - A^\lambda_{\mathrm{ALA}}(y,r),\, x-y\right\rangle \ge \frac{\varepsilon}{1+\lambda\varepsilon} |x-y|^2.
\]

\item \textbf{Inclusion:}
\[
A^\lambda_{\mathrm{ALA}}(x,r) \in A_{\mathrm{ALA}}(J_\lambda^r(x),r).
\]

\item \textbf{Uniform linear growth:}
\[
|A^\lambda_{\mathrm{ALA}}(x,r)| \le C(1+|x|),
\]
where $C$ is independent of $\lambda$.

\item \textbf{Pointwise convergence:}
For every \(x\in\mathbb R^d\), \(A^\lambda_{\mathrm{ALA}}(x,r)\to A^0_{\mathrm{ALA}}(x,r)\) as \(\lambda\to0\), where \(A^0_{\mathrm{ALA}}(x,r)\) is the minimal norm element in \(A_{\mathrm{ALA}}(x,r)\).
\end{enumerate}
\end{lemma}

\begin{proof}
Properties (i)--(iii) and (v) follow from the general theory of resolvents of maximal monotone operators (see \cite{Brezis1973}, Proposition 2.6). In particular, the regularized strong monotonicity constant is \(\varepsilon/(1+\lambda\varepsilon)\). For (iv), we use \(A^\lambda_{\mathrm{ALA}}(x,r)\in A_{\mathrm{ALA}}(J_\lambda^r(x),r)\) and Lemma~\ref{lem:ALA-properties}:
\[
|A^\lambda_{\mathrm{ALA}}(x,r)|
\le
\max(\kappa_+,\kappa_-)\sqrt d+\varepsilon|J_\lambda^r(x)|
\le C(1+|x|).
\]
\end{proof}

\begin{lemma}[Strong convergence of Yosida approximations]
\label{lem:yosida-strong-convergence}

Let \(X\in \mathcal S_{\mathcal P}^{2}\) and let \(R=(R_t)_{0\le t\le T}\) be a measurable real-valued reference process, for instance \(R_t=\Upsilon(\mu_t)\). Then

\[
A_{\mathrm{ALA}}^{\lambda}(X,R)
\longrightarrow
A_{\mathrm{ALA}}^{0}(X,R)
\quad\text{in }\mathcal H_{\mathcal P}^{2},
\]

as \(\lambda\to0\), where \(A_{\mathrm{ALA}}^{0}(X_t,R_t)\) denotes the minimal norm element of \(A_{\mathrm{ALA}}(X_t,R_t)\).
\end{lemma}

\begin{proof}
For every pair \((x,r)\in\mathbb R^d\times\mathbb R\), the operator \(A_{\mathrm{ALA}}(\cdot,r)\) is maximal monotone. Hence its Yosida approximation satisfies
\[
A_{\mathrm{ALA}}^\lambda(x,r)
\longrightarrow
A_{\mathrm{ALA}}^0(x,r)
\qquad\text{as }\lambda\to0,
\]
where \(A_{\mathrm{ALA}}^0(x,r)\) is the element of minimal norm in \(A_{\mathrm{ALA}}(x,r)\). Moreover, Lemma~\ref{lem:yosida-properties} gives
\[
|A_{\mathrm{ALA}}^\lambda(x,r)|\le C(1+|x|),
\]
where \(C\) is independent of \(\lambda\) and \(r\). This domination is uniform in \(\lambda\), in \(t\), and in \(P\). Indeed, for every \(P\in\mathcal P\),
\[
\mathbb E^P\int_0^T
\bigl(1+|X_t^P|^2\bigr)\,dt
\le
T
+
T\,\mathbb E^P
\left[
\sup_{0\le t\le T}|X_t^P|^2
\right].
\]
Since \(X\in\mathcal S_{\mathcal P}^{2}\), it follows that
\[
\sup_{P\in\mathcal P}
\mathbb E^P\int_0^T
\bigl(1+|X_t^P|^2\bigr)\,dt
<\infty.
\]
Consequently,
\[
|A_{\mathrm{ALA}}^\lambda(X_t^P,R_t^P)|^2
\le
C\bigl(1+|X_t^P|^2\bigr)
\]
with a constant independent of \(\lambda\), \(t\), and \(P\). Dominated convergence, applied under each \(P\) and then uniformized by the preceding bound, therefore implies
\[
\|A_{\mathrm{ALA}}^\lambda(X,R)-A_{\mathrm{ALA}}^0(X,R)\|_{\mathcal H_{\mathcal P}^{2}}
\longrightarrow 0.
\]
\end{proof}

\begin{lemma}[Strong convergence of resolvents]
\label{lem:resolvent-convergence}
Let \(X\in \mathcal S^2(P)\) and let \(R=(R_t)_{0\le t\le T}\) be a measurable real-valued reference process. Then
\[
J_\lambda^R(X) \to X \quad \text{in } \mathcal H^2(P)
\]
as \(\lambda\to0\).
\end{lemma}

\begin{proof}
By definition, \(X=J_\lambda^R(X)+\lambda A^\lambda_{\mathrm{ALA}}(X,R)\). Thus,
\[
|X-J_\lambda^R(X)|
=
\lambda |A^\lambda_{\mathrm{ALA}}(X,R)|
\le
\lambda C(1+|X|).
\]
Since \(X\in\mathcal S^2(P)\), the right-hand side converges to zero in \(\mathcal H^2(P)\). Hence the announced strong convergence.
\end{proof}

\subsection{Monotone domination assumption}

The estimates close only under a sufficient dissipation assumption.

\begin{hypothesis}[Monotone domination]
\label{H5bis}
The monotonicity constant $\gamma$ of the generator $f$ dominates the Lipschitz constants associated with the variables $z$ and $\Psi_\delta$. More precisely,
\[
2\gamma > 2L_f^2 + L_\delta^2.
\]
This condition will be used in the Cauchy convergence, uniqueness and stability theorems.
\end{hypothesis}

\subsection{Yosida-regularized backward variational system}

The Yosida-regularized backward variational system is written, for each $\lambda>0$,
\begin{equation}\label{eq:regularized}
-dY_t^{\lambda} = \Bigl(f\bigl(t,X_t^P, \Psi_\gamma(\mu_t^P), Y_t^{\lambda}, Z_t^{\lambda}, \Psi_\delta(\mu_t^{P,Y^{\lambda}}), u_t\bigr) + \rho A^\lambda_{\mathrm{ALA}}\bigl(X_t^P,\Upsilon(\mu_t^P)\bigr)\Bigr)dt - Z_t^{\lambda} dB_t,
\end{equation}
with terminal condition $Y_T^{\lambda} = g\bigl(X_T^P, \Upsilon(\mu_T^P)\bigr)$.

\begin{remark}[Role of the backward observable \(\Psi_\delta\)]
The operator \(A_{\mathrm{ALA}}\) acts on the forward state \(X^P\) and on the scalar reference \(\Upsilon(\mu_t^P)\). It represents the asymmetric market signal (loss aversion) that drives the backward variational valuation dynamics. This structure is typical of control problems where the value depends on a collective reference. The term \(\rho A^\lambda_{\mathrm{ALA}}(X_t^P,\Upsilon(\mu_t^P))\) is therefore an exogenous monotone drift, and not an operator acting on the backward variable \(Y\). The observable \(\Psi_\delta(\mu_t^{P,Y^{\lambda}})\) models a collective assessment of future costs and therefore depends naturally on the distribution of the value process \(Y_t^{\lambda}\).

This dependence on the law of \(Y^\lambda\) does not introduce an additional difficulty in the Yosida regularization. Indeed, the regularized monotone term is
\[
A^\lambda_{\mathrm{ALA}}\bigl(X_t^P,\Upsilon(\mu_t^P)\bigr),
\]
and it depends only on the already constructed forward state and on the forward collective reference. The backward observable \(\Psi_\delta(\mu_t^{P,Y^\lambda})\) appears only inside the Lipschitz generator \(f\). It is therefore treated by the usual McKean--Vlasov stability estimate rather than by monotone-operator regularization.

More precisely, if \(Y^1\) and \(Y^2\) are two backward processes, then the Wasserstein stability of collective observables gives
\[
\bigl|
\Psi_\delta(\mathcal L^P(Y_t^1))
-
\Psi_\delta(\mathcal L^P(Y_t^2))
\bigr|
\le
L_\delta\,W_2\bigl(\mathcal L^P(Y_t^1),\mathcal L^P(Y_t^2)\bigr)
\le
L_\delta\,\bigl(\mathbb E^P|Y_t^1-Y_t^2|^2\bigr)^{1/2}.
\]
Thus the \(Y\)-law dependence contributes Lipschitz terms to the energy estimates, while the nonsmoothness is entirely concentrated in the ALA signal evaluated at \(X^P\). This separation is important: the Yosida approximation is used only for the maximal monotone operator \(A_{\mathrm{ALA}}\), whereas \(\Psi_\delta\) is controlled through Wasserstein stability and Gronwall-type arguments.
\end{remark}

For each \(\lambda>0\), \(A^\lambda_{\mathrm{ALA}}(X^P,\Upsilon(\mu^P))\in \mathcal H^2(P)\) (by point (iv) of Lemma~\ref{lem:yosida-properties}). Since \(f\) is Lipschitz in \((y,z)\), the regularized generator satisfies the assumptions of the Pardoux--Peng theorem \cite{PardouxPeng1990}, which guarantees the existence of a unique solution \((Y^\lambda,Z^\lambda)\in \mathcal S^2(P)\times \mathcal H^2(P)\).

\subsection{Uniform estimates for the Yosida-regularized variational system}

In order to pass to the limit, we first establish the following estimate.

\begin{theorem}[Uniform estimates of regularized solutions]
\label{thm:uniform-estimates}

Under assumptions {\rm(H1)--(H5)} and {\rm(H5bis)}, there exists a constant $C>0$, independent of $\lambda>0$ and $P\in\mathcal P$, such that
\[
\sup_{\lambda>0}
\sup_{P\in\mathcal P}
\mathbb E^P
\left[
\sup_{0\le t\le T}|Y_t^\lambda|^2
+
\int_0^T \operatorname{Tr}(Z_t^\lambda a_t (Z_t^\lambda)^\top)\,dt
\right]
\le C.
\]

\end{theorem}

\begin{proof}

Fix \(P\in\mathcal P\) and \(\lambda>0\). To lighten notation, we write \(R_t^P:=\Upsilon(\mu_t^P)\), \(A_t^\lambda := A^\lambda_{\mathrm{ALA}}(X_t^P,R_t^P)\), and \(F_t^\lambda := f(t, X_t^P, \Psi_\gamma(\mu_t^P), Y_t^\lambda, Z_t^\lambda, \Psi_\delta(\mu_t^{P,Y^\lambda}), u_t)\). The Yosida-regularized variational system is then written
\[
-dY_t^\lambda = \bigl(F_t^\lambda + \rho A_t^\lambda\bigr)dt - Z_t^\lambda dB_t,\qquad
Y_T^\lambda = g\bigl(X_T^P, \Upsilon(\mu_T^P)\bigr).
\]

\medskip
\noindent
\textbf{Application of It\^o's formula.}
We apply It\^o's formula to $|Y_t^\lambda|^2$ between $t$ and $T$. We obtain
\[
|Y_t^\lambda|^2
+
\int_t^T \operatorname{Tr}(Z_s^\lambda a_s (Z_s^\lambda)^\top)\,ds
= |Y_T^\lambda|^2
+ 2\int_t^T Y_s^\lambda \bigl(F_s^\lambda + \rho A_s^\lambda\bigr) ds
- 2\int_t^T Y_s^\lambda Z_s^\lambda dB_s.
\]
Taking expectation under $P$, the martingale term vanishes:
\[
\mathbb E^P|Y_t^\lambda|^2
+
\mathbb E^P\int_t^T \operatorname{Tr}(Z_s^\lambda a_s (Z_s^\lambda)^\top)\,ds
= \mathbb E^P|Y_T^\lambda|^2
+ 2\mathbb E^P\int_t^T Y_s^\lambda F_s^\lambda ds
+ 2\rho\mathbb E^P\int_t^T Y_s^\lambda A_s^\lambda ds.
\]

\medskip
\noindent
\textbf{Control of the terminal condition.}
By the growth assumption on $g$ and the Wasserstein stability of collective observables,
\[
|Y_T^\lambda| = \left| g\bigl(X_T^P, \Upsilon(\mu_T^P)\bigr) \right| \le C\bigl(1 + |X_T^P| + |\Upsilon(\mu_T^P)|\bigr).
\]
Since $\chi$ is Lipschitz, $|\Upsilon(\mu_T^P)| \le |\chi(0)| + L_\chi (\mathbb E^P|X_T^P|^2)^{1/2}$. Thus, $\mathbb E^P|Y_T^\lambda|^2 \le C(1 + \mathbb E^P|X_T^P|^2)$. The a priori estimates of the forward component give $\sup_{P\in\mathcal P} \mathbb E^P[\sup_{0\le t\le T}|X_t^P|^2] < \infty$. Therefore $\sup_{P\in\mathcal P} \mathbb E^P|Y_T^\lambda|^2 \le C$, with a constant independent of $\lambda$.

\medskip
\noindent
\textbf{Control of the generator $f$.}
By the monotonicity assumption in $y$, applied with $0$, and by the linear growth of the generator, there exists a constant $C>0$ such that
\[
2Y_s^\lambda F_s^\lambda \le -\gamma |Y_s^\lambda|^2 + C\Bigl(1 + |X_s^P|^2 + |\Psi_\gamma(\mu_s^P)|^2 + |Z_s^\lambda|^2 + |\Psi_\delta(\mu_s^{P,Y^\lambda})|^2 + |u_s|^2\Bigr).
\]
The stability of collective observables gives $|\Psi_\gamma(\mu_s^P)|^2 \le C(1 + \mathbb E^P|X_s^P|^2)$, and $|\Psi_\delta(\mu_s^{P,Y^\lambda})|^2 \le C(1 + \mathbb E^P|Y_s^\lambda|^2)$. Consequently,
\[
2Y_s^\lambda F_s^\lambda \le -\gamma |Y_s^\lambda|^2 + C\Bigl(1 + |X_s^P|^2 + |u_s|^2 + |Z_s^\lambda|^2 + \mathbb E^P|Y_s^\lambda|^2\Bigr).
\]

\medskip
\noindent
\textbf{Control of the regularized monotone term.}
By the uniform linear growth of the Yosida approximation, $|A_s^\lambda| \le C(1 + |X_s^P|)$, with $C$ independent of $\lambda$. Young's inequality gives, for every $\eta>0$,
\[
2\rho |Y_s^\lambda||A_s^\lambda| \le \eta |Y_s^\lambda|^2 + \frac{\rho^2}{\eta}|A_s^\lambda|^2.
\]
Using the growth of $A_s^\lambda$, we obtain
\[
2\rho |Y_s^\lambda||A_s^\lambda| \le \eta |Y_s^\lambda|^2 + C_\eta(1 + |X_s^P|^2).
\]

\medskip
\noindent
\textbf{Integral estimate.}
Substituting the previous estimates into It\^o's identity, then choosing $\eta>0$ sufficiently small, we obtain
\[
\mathbb E^P|Y_t^\lambda|^2
+
\mathbb E^P\int_t^T \operatorname{Tr}(Z_s^\lambda a_s (Z_s^\lambda)^\top)\,ds
\le C
+ C\int_t^T \mathbb E^P|Y_s^\lambda|^2 ds
+ C\int_t^T
\mathbb E^P\operatorname{Tr}(Z_s^\lambda a_s (Z_s^\lambda)^\top)\,ds.
\]
The term in \(Z^\lambda\) on the right-hand side comes from the Lipschitz continuity of \(f\) in \(z\). Since \(a_t\) is uniformly elliptic, the usual norm \(|Z_t|^2\) is controlled by the weighted norm \(\operatorname{Tr}(Z_ta_tZ_t^\top)\). Using Young's inequality with a small coefficient before the \(z\)-contribution, we can absorb part of the weighted integral into the left-hand side. We thus obtain
\[
\mathbb E^P|Y_t^\lambda|^2
+
\frac12
\mathbb E^P\int_t^T \operatorname{Tr}(Z_s^\lambda a_s (Z_s^\lambda)^\top)\,ds
\le C + C\int_t^T \mathbb E^P|Y_s^\lambda|^2 ds.
\]
By the backward Gronwall lemma, $\sup_{0\le t\le T} \mathbb E^P|Y_t^\lambda|^2 \le C$. Returning to the previous inequality, we also obtain
\[
\mathbb E^P\int_0^T \operatorname{Tr}(Z_s^\lambda a_s (Z_s^\lambda)^\top)\,ds \le C.
\]

\medskip
\noindent
\textbf{Estimate of the time supremum of $Y^\lambda$.}
From the backward variational equation,
\[
Y_t^\lambda = Y_T^\lambda + \int_t^T \bigl(F_s^\lambda + \rho A_s^\lambda\bigr) ds - \int_t^T Z_s^\lambda dB_s.
\]
Taking the supremum over $t\in[0,T]$, then expectation under $P$, we obtain
\[
\mathbb E^P\left[ \sup_{0\le t\le T}|Y_t^\lambda|^2 \right] \le C\mathbb E^P|Y_T^\lambda|^2 + C\mathbb E^P\left[ \left( \int_0^T |F_s^\lambda + \rho A_s^\lambda| ds \right)^2 \right] + C\mathbb E^P\left[ \sup_{0\le t\le T} \left| \int_0^t Z_s^\lambda dB_s \right|^2 \right].
\]
By the Burkholder--Davis--Gundy inequality, uniformly in $P\in\mathcal P$,
\[
\mathbb E^P\left[ \sup_{0\le t\le T} \left| \int_0^t Z_s^\lambda dB_s \right|^2 \right]
\le
C \mathbb E^P \int_0^T \operatorname{Tr}(Z_s^\lambda a_s (Z_s^\lambda)^\top)\,ds.
\]
The previous estimates on $Y^\lambda$, $Z^\lambda$, the growth of $f$, the growth of $A^\lambda_{\mathrm{ALA}}$, the bounds of the forward component and the integrability of the control then give
\[
\mathbb E^P\left[ \sup_{0\le t\le T}|Y_t^\lambda|^2 \right] \le C.
\]

\medskip
\noindent
\textbf{Uniformity in $P$ and $\lambda$.}
All constants used depend only on $T$, the Lipschitz and growth constants of the coefficients, the monotonicity constant $\gamma$, $\rho$, the bounds of the forward component and the integrability of the control. They depend neither on $\lambda$, nor on the particular probability $P\in\mathcal P$.

Taking the supremum over $P\in\mathcal P$ and then over $\lambda>0$, we obtain
\[
\sup_{\lambda>0} \sup_{P\in\mathcal P}
\mathbb E^P\left[
\sup_{0\le t\le T}|Y_t^\lambda|^2
+
\int_0^T \operatorname{Tr}(Z_t^\lambda a_t (Z_t^\lambda)^\top)\,dt
\right] \le C.
\]
The proof is complete.

\end{proof}

\subsection{Cauchy convergence of regularized approximations}

\begin{theorem}[Uniform Cauchy convergence of regularized solutions]
\label{thm:robust-cauchy}

Under assumptions {\rm(H1)--(H5)} and {\rm(H5bis)}, the family $\{(Y^\lambda,Z^\lambda)\}_{\lambda>0}$ is Cauchy in $\mathcal S^2(P)\times \mathcal H^2(P)$, uniformly in $P\in\mathcal P$. More precisely,
\[
\sup_{P\in\mathcal P} \mathbb E^P\!\left[
\sup_{0\le t\le T} |Y_t^\lambda - Y_t^\mu|^2
+
\int_0^T
\operatorname{Tr}\bigl((Z_t^\lambda-Z_t^\mu)a_t(Z_t^\lambda-Z_t^\mu)^\top\bigr)\,dt
\right] \longrightarrow 0,
\]
as $\lambda,\mu\to0$.

\end{theorem}

\begin{proof}
Set \(\delta Y = Y^\lambda - Y^\mu\), \(\delta Z = Z^\lambda - Z^\mu\), and
\[
\delta A
=
A^\lambda_{\mathrm{ALA}}(X^P,\Upsilon(\mu^P))
-
A^\mu_{\mathrm{ALA}}(X^P,\Upsilon(\mu^P)).
\]
By difference of the equations,
\[
-d(\delta Y_t) = \bigl(\delta f_t + \rho\delta A_t\bigr)dt - \delta Z_t dB_t,\qquad \delta Y_T = 0,
\]
where $\delta f_t = f(Y_t^\lambda,Z_t^\lambda,\Psi_\delta(\mu_t^{P,Y^\lambda})) - f(Y_t^\mu,Z_t^\mu,\Psi_\delta(\mu_t^{P,Y^\mu}))$.

Decompose $\delta f_t$:
\[
\begin{aligned}
\delta f_t &= \bigl(f(Y_t^\lambda,Z_t^\lambda,\Psi_\delta(\mu_t^{P,Y^\lambda})) - f(Y_t^\mu,Z_t^\lambda,\Psi_\delta(\mu_t^{P,Y^\lambda}))\bigr)\\
&\quad + \bigl(f(Y_t^\mu,Z_t^\lambda,\Psi_\delta(\mu_t^{P,Y^\lambda})) - f(Y_t^\mu,Z_t^\mu,\Psi_\delta(\mu_t^{P,Y^\lambda}))\bigr)\\
&\quad + \bigl(f(Y_t^\mu,Z_t^\mu,\Psi_\delta(\mu_t^{P,Y^\lambda})) - f(Y_t^\mu,Z_t^\mu,\Psi_\delta(\mu_t^{P,Y^\mu}))\bigr).
\end{aligned}
\]

By the monotonicity assumption (H5),
\[
\begin{aligned}
\delta Y_t \Bigl( & f\bigl(t, X_t, \Psi_\gamma(\mu_t), Y_t^\lambda, Z_t, \Psi_\delta(\mathcal L(Y_t^\lambda)), u_t\bigr) \\
& - f\bigl(t, X_t, \Psi_\gamma(\mu_t), Y_t^\mu, Z_t, \Psi_\delta(\mathcal L(Y_t^\mu)), u_t\bigr) \Bigr) \le -\gamma |\delta Y_t|^2.
\end{aligned}
\]

By the Lipschitz continuity of $f$ in $z$,
\[
\begin{aligned}
&\bigl| f(t, X_t^{\mu}, \Psi_\gamma(\mu_t), Y_t^{\mu}, Z_t^{\lambda}, \Psi_\delta(\mathcal L(Y_t^{\mu})), u_t) \\
&\qquad - f(t, X_t^{\mu}, \Psi_\gamma(\mu_t), Y_t^{\mu}, Z_t^{\mu}, \Psi_\delta(\mathcal L(Y_t^{\mu})), u_t) \bigr| \\
&\qquad\qquad \le L_f \, |\delta Z_t|,
\end{aligned}
\]

For the third term, we use the Wasserstein inequality: $W_2^2(\mu_t^{P,Y^\lambda},\mu_t^{P,Y^\mu}) \le \mathbb E^P|\delta Y_t|^2$. By the Lipschitz continuity of $f$ in its last variable and Lemma~\ref{lem:stability-observables} of Section 2,
\[
|\Psi_\delta(\mu_t^{P,Y^\lambda}) - \Psi_\delta(\mu_t^{P,Y^\mu})| \le L_\delta (\mathbb E^P|\delta Y_t|^2)^{1/2}.
\]
Thus,
\[
|\delta Y_t| \, |f(Y_t^\mu,Z_t^\mu,\Psi_\delta(\mu_t^{P,Y^\lambda})) - f(Y_t^\mu,Z_t^\mu,\Psi_\delta(\mu_t^{P,Y^\mu}))| \le C |\delta Y_t| (\mathbb E^P|\delta Y_t|^2)^{1/2} \le \frac{C}{2} |\delta Y_t|^2 + \frac{C}{2} \mathbb E^P|\delta Y_t|^2.
\]

Using the uniform ellipticity of \(a_t\), the Euclidean term \(|\delta Z_t|^2\) is controlled by the weighted quantity \(|\delta Z_t|_{a_t}^2\). Gathering terms, we obtain
\[
2\delta Y_t \delta f_t \le -(2\gamma - 2L_f^2 - L_\delta^2)|\delta Y_t|^2 + \frac12 |\delta Z_t|_{a_t}^2 + C \mathbb E^P|\delta Y_t|^2.
\]
By assumption {\rm(H5bis)}, $2\gamma - 2L_f^2 - L_\delta^2 > 0$.

To control the regularized monotone term, we use Lemma~\ref{lem:yosida-strong-convergence}. Indeed,
\[
\delta A = A_{\mathrm{ALA}}^{\lambda}(X^P,\Upsilon(\mu^P)) - A_{\mathrm{ALA}}^{\mu}(X^P,\Upsilon(\mu^P)),
\]
and by the triangle inequality,
\[
\|\delta A\|_{\mathcal H_{\mathcal P}^2}
\le
\|A_{\mathrm{ALA}}^{\lambda} - A_{\mathrm{ALA}}^{0}\|_{\mathcal H_{\mathcal P}^2}
+
\|A_{\mathrm{ALA}}^{\mu} - A_{\mathrm{ALA}}^{0}\|_{\mathcal H_{\mathcal P}^2}.
\]
Lemma~\ref{lem:yosida-strong-convergence} shows that each of the two terms on the right-hand side tends to zero as $\lambda,\mu\to0$. Consequently, $\|\delta A\|_{\mathcal H_{\mathcal P}^2} \longrightarrow 0$. Thus,
\[
\sup_{P\in\mathcal P} \mathbb E^P\!\left[ \int_0^T |\delta A_t|^2 dt \right] \longrightarrow 0,
\]
uniformly in $P\in\mathcal P$.

We then apply Young's inequality with a parameter $\eta>0$:
\[
2\rho|\delta Y_t||\delta A_t| \le \eta|\delta Y_t|^2 + \frac{\rho^2}{\eta}|\delta A_t|^2.
\]

Apply It\^o's formula to $|\delta Y_t|^2$:
\[
|\delta Y_t|^2 + \int_t^T |\delta Z_s|_{a_s}^2 ds = 2\int_t^T \delta Y_s (\delta f_s + \rho\delta A_s) ds - 2\int_t^T \delta Y_s \delta Z_s dB_s.
\]

Taking expectation and using the estimates above,
\[
\mathbb E^P|\delta Y_t|^2 + \int_t^T \mathbb E^P|\delta Z_s|_{a_s}^2 ds \le C\int_t^T \mathbb E^P|\delta Y_s|^2 ds + \frac12 \int_t^T \mathbb E^P|\delta Z_s|_{a_s}^2 ds + \frac{\rho^2}{\eta} \mathbb E^P\int_t^T |\delta A_s|^2 ds.
\]

Thus,
\[
\mathbb E^P|\delta Y_t|^2 + \frac12 \mathbb E^P\int_t^T |\delta Z_s|_{a_s}^2 ds \le C\int_t^T \mathbb E^P|\delta Y_s|^2 ds + \frac{\rho^2}{\eta} \mathbb E^P\int_t^T |\delta A_s|^2 ds.
\]

By the convergence of $\delta A$ to $0$ in $\mathcal H_{\mathcal P}^2$, we have $\sup_{P\in\mathcal P} \mathbb E^P\int_t^T |\delta A_s|^2 ds \longrightarrow 0$ as $\lambda,\mu\to0$. Set $m(t) = \mathbb E^P|\delta Y_t|^2$. Then
\[
m(t) \le C\int_t^T m(s) ds + o(1),
\]
where $o(1)$ is uniform in $P\in\mathcal P$.

By the backward Gronwall lemma, we obtain \(m(t) \le o(1)e^{C(T-t)}\). In particular, \(\sup_t m(t)\le o(1)\) and \(\mathbb E^P\int_0^T|\delta Z_s|_{a_s}^2ds\le o(1)\). Hence the announced Cauchy convergence.
\end{proof}

\subsection{Identification of the limit and existence of the solution}

The preceding uniform estimates now allow rigorous identification of the limit of the regularized monotone term. This passage constitutes the core of the variational analysis.

The Cauchy theorem guarantees the existence of \((Y,Z)\in \mathcal S^2(P)\times \mathcal H^2(P)\) such that \(Y^\lambda\to Y\) in \(\mathcal S^2(P)\), \(Z^\lambda\to Z\) in \(\mathcal H^2(P)\). Moreover, with \(R_t^P:=\Upsilon(\mu_t^P)\), the family \(\{A^\lambda_{\mathrm{ALA}}(X^P,R^P)\}_{\lambda>0}\) is bounded in \(\mathcal H^2(P)\) (Lemma~\ref{lem:yosida-properties}), hence a subsequence converges weakly to \(\Gamma\in \mathcal H^2(P)\). The weak limit is then identified by Proposition~\ref{prop:robust-graph-closure}.

\begin{proposition}[Robust graph closure of $\partial_x\Phi_{\mathrm{ALA}}$]
\label{prop:robust-graph-closure}

Let \(X^P\) be fixed and set \(R_t^P:=\Upsilon(\mu_t^P)\). Consider a sequence
\[
\eta^\lambda
=
A^\lambda_{\mathrm{ALA}}(X^P,R^P).
\]

Assume that, as $\lambda\to0$,
\[
\eta^\lambda \rightharpoonup \eta \quad\text{weakly in } \mathcal H_{\mathcal P}^2 .
\]

Then, for almost every $(t,\omega)$ and for every admissible probability $P\in\mathcal P$,
\[
\eta_t \in A_{\mathrm{ALA}}\bigl(X_t^P,R_t^P\bigr)
=
\partial_x\Phi_{\mathrm{ALA}}\bigl(X_t^P,\mu_t^P\bigr).
\]

\end{proposition}

\begin{proof}

We use the classical Minty--Brezis argument, applied here to the graph frozen by the collective reference \(R^P\). We first fix \(P\in\mathcal P\). Since the estimates are uniform in \(P\), the resulting identification holds for every admissible probability.

\medskip

\noindent
\textbf{Characterization of the Yosida regularization.}

For every \(\lambda>0\), we introduce the resolvent
\[
J_\lambda^{R^P} := (I + \lambda A_{\mathrm{ALA}}(\cdot,R^P))^{-1},
\]
understood pointwise in \(t\).
By definition of the Yosida regularization,
\[
\eta^\lambda
=
\frac1\lambda\bigl(X^P - J_\lambda^{R^P} X^P\bigr).
\]
Consequently, \(J_\lambda^{R^P} X^P = X^P - \lambda\eta^\lambda\).

Since \((\eta^\lambda)_{\lambda>0}\) is bounded in \(\mathcal H_{\mathcal P}^2\) (which follows from the uniform linear growth of \(A^\lambda_{\mathrm{ALA}}\)), we have \(\lambda\eta^\lambda \to0\) in \(\mathcal H_{\mathcal P}^2\). By Lemma~\ref{lem:resolvent-convergence}, we also have \(J_\lambda^{R^P}X^P\to X^P\) in \(\mathcal H_{\mathcal P}^2\).

\medskip

\noindent
\textbf{Monotonicity relation.}

Let \((v,\zeta)\) be any element of the frozen graph, i.e.
\[
\zeta_t\in A_{\mathrm{ALA}}(v_t,R_t^P)
\quad\text{a.e.}
\]
By the monotonicity of \(A_{\mathrm{ALA}}(\cdot,R_t^P)\), for every \(\lambda>0\),
\[
\bigl\langle \eta^\lambda - \zeta,\; J_\lambda^{R^P} X^P - v \bigr\rangle_{\mathcal H_{\mathcal P}^2} \ge 0.
\]
In other words,
\[
\langle \eta^\lambda,\; J_\lambda^{R^P} X^P - v \rangle \ge \langle \zeta,\; J_\lambda^{R^P} X^P - v \rangle.
\tag{4.6}
\]

\medskip

\noindent
\textbf{Passage to the limit.}

Since \(J_\lambda^{R^P} X^P \to X^P\) strongly in \(\mathcal H_{\mathcal P}^2\) and \(\eta^\lambda \rightharpoonup \eta\) weakly in \(\mathcal H_{\mathcal P}^2\), we can pass to the limit on the left-hand side of (4.6):
\[
\lim_{\lambda\to0} \langle \eta^\lambda,\; J_\lambda^{R^P} X^P - v \rangle = \langle \eta,\; X^P - v \rangle.
\]

The strong convergence of \(J_\lambda^{R^P}X^P\) also gives
\[
\lim_{\lambda\to0} \langle \zeta,\; J_\lambda^{R^P} X^P - v \rangle = \langle \zeta,\; X^P - v \rangle.
\]

Passing to the limit in (4.6), we obtain
\[
\langle \eta - \zeta,\; X^P - v \rangle \ge 0
\qquad\text{for every }(v,\zeta) \in \operatorname{Graph}(A_{\mathrm{ALA}}(\cdot,R^P)).
\tag{4.7}
\]

\medskip

\noindent
\textbf{Strong--weak graph closure.}

It is useful to make explicit the graph-closure mechanism behind (4.7). Set
\[
X^\lambda:=J_\lambda^{R^P}X^P.
\]
By the defining property of the Yosida approximation, one has
\[
\eta^\lambda\in
A_{\mathrm{ALA}}\bigl(X^\lambda,R^P\bigr)
\quad\text{for every }\lambda>0.
\]
The preceding estimates give
\[
X^\lambda\longrightarrow X^P
\quad\text{strongly in }\mathcal H_{\mathcal P}^2,
\qquad
\eta^\lambda\rightharpoonup \eta
\quad\text{weakly in }\mathcal H_{\mathcal P}^2.
\]
In the present frozen formulation the reference \(R^P=\Upsilon(\mu^P)\) is already identified. In a varying-reference formulation, the same argument would require the strong convergence of \(\Upsilon(\mu^\lambda)\) toward \(\Upsilon(\mu)\), which follows from the Wasserstein stability of the observable \(\Upsilon\). Thus all ingredients of the strong--weak closure theorem for maximal monotone graphs are present.

\medskip

\noindent
\textbf{Minty argument.}

The inequality (4.7) is first an integrated inequality in the Hilbert space
\[
L^2\bigl([0,T]\times\Omega,dt\otimes P;\mathbb R^d\bigr).
\]
To pass from this integrated statement to a pointwise graph inclusion, we introduce the Nemytskii operator associated with the frozen graph:
\[
\mathcal A_{R^P}(V)
:=
\Bigl\{
\Xi\in L^2(dt\otimes P;\mathbb R^d)
:\;
\Xi_t(\omega)\in A_{\mathrm{ALA}}\bigl(V_t(\omega),R_t^P(\omega)\bigr)
\text{ for }dt\otimes P\text{-a.e. }(t,\omega)
\Bigr\}.
\]
For \(dt\otimes P\)-almost every \((t,\omega)\), the map \(x\mapsto A_{\mathrm{ALA}}(x,R_t^P(\omega))\) is maximal monotone. Since the graph has measurable selections, the Nemytskii operator \(\mathcal A_{R^P}\) is maximal monotone in the above \(L^2\) space. The standard proof of this fact is based on localization on measurable sets \(E\subset[0,T]\times\Omega\) and on measurable test selections \((v,\zeta)\) of the graph.

Thus, (4.7) exactly means that the pair \((X^P,\eta)\) is monotonically related to the whole graph of \(\mathcal A_{R^P}\). By Minty's theorem for maximal monotone operators in Hilbert spaces, such a pair must belong to the graph itself. Therefore,
\[
\eta_t\in A_{\mathrm{ALA}}(X_t^P,R_t^P)
=
\partial_x\Phi_{\mathrm{ALA}}(X_t^P,\mu_t^P).
\]

\medskip

\noindent
\textbf{Compatibility with law dependence.}

In our construction, the forward process \(X^P\) is already fixed and satisfies \(\mu_t^P=\mathcal L^P(X_t^P)\). Consequently, the only parameter entering the monotone graph is the already identified scalar reference \(R_t^P=\Upsilon(\mu_t^P)\). The classical Minty--Brezis theorem is therefore applied to the frozen maximal monotone operator \(x\mapsto A_{\mathrm{ALA}}(x,R_t^P)\). The law dependence intervenes only through this reference. This is why no convergence of graphs with respect to the full measure variable is needed in the present identification.

The proof is complete.

\end{proof}

\begin{remark}[Interpretation of the Minty--Brezis argument]
\label{rem:minty-interpretation}

The Minty--Brezis argument isolates two points: obtaining sufficient convergences for the regularized monotone terms, then identifying any weak limit as an element of the subdifferential by maximal monotonicity. In the present framework, the law \(\mu_t^P=\mathcal L^P(X_t^P)\) is already fixed by the forward construction; the argument therefore concerns the frozen operator \(A_{\mathrm{ALA}}(\cdot,\Upsilon(\mu_t^P))\). The continuity in measure of Proposition~\ref{prop:Phi-wasserstein} ensures the coherence of the limiting functional, while the graph identification itself is performed through the scalar reference.

\end{remark}

The preceding proposition allows identifying the limiting monotone term from the estimates of the Yosida-regularized variational system.

\begin{theorem}[Robust existence for the selected backward variational system]
\label{thm:existence-backward}

Under assumptions {\rm(H1)--(H5)} and {\rm(H5bis)}, for each $P\in\mathcal P$,
there exists an adapted triplet
\[
(Y^P, Z^P, \Gamma^P) \in \mathcal S^2(P) \times \mathcal H^2(P) \times \mathcal H^2(P)
\]
solution of the selected robust backward variational system with non-smooth monotone signal
\[
-dY_t^P
=
\Bigl[
f\bigl(t,X_t^P, \Psi_\gamma(\mu_t^P), Y_t^P, Z_t^P, \Psi_\delta(\mu_t^{P,Y}), u_t\bigr)
+
\rho\,\Gamma_t^P
\Bigr]dt
- Z_t^P dB_t,
\]
\[
Y_T^P = g\bigl(X_T^P, \Upsilon(\mu_T^P)\bigr),
\qquad
\Gamma_t^P\in
A_{\mathrm{ALA}}\bigl(X_t^P,\Upsilon(\mu_t^P)\bigr)
=
\partial_x\Phi_{\mathrm{ALA}}\bigl(X_t^P,\mu_t^P\bigr).
\]

The process $\Gamma^P$ is a measurable selection of the multi-valued subdifferential evaluated at the forward state and at the scalar collective reference. It is obtained as the weak limit of the Yosida approximations and belongs to $\mathcal H^2(P)$ as the weak limit of the bounded family
\[
\bigl\{A^\lambda_{\mathrm{ALA}}\bigl(X^P,\Upsilon(\mu^P)\bigr)\bigr\}_{\lambda>0}.
\]
Moreover,
\[
\sup_{P\in\mathcal P} \mathbb E^P\!\left[
\sup_{0\le t\le T}|Y_t^P|^2
+
\int_0^T \operatorname{Tr}(Z_t^P a_t (Z_t^P)^\top)\,dt
\right] \le C < \infty.
\]

\end{theorem}

\begin{proof}
The proof is articulated in four complementary arguments.

\medskip
\noindent\textbf{Construction of regularized approximations.}
For each $\lambda>0$, the Yosida-regularized backward variational system \eqref{eq:regularized} admits a unique solution $(Y^\lambda,Z^\lambda)$ by the Pardoux--Peng theorem \cite{PardouxPeng1990}.

\medskip
\noindent\textbf{Cauchy convergence and extraction of a limit.}
By Theorem~\ref{thm:robust-cauchy}, the family $\{(Y^\lambda,Z^\lambda)\}_{\lambda>0}$ is Cauchy in $\mathcal S^2(P)\times \mathcal H^2(P)$ uniformly in $P$. It therefore converges to $(Y,Z)$ in $\mathcal S^2(P)\times \mathcal H^2(P)$. Moreover, the family $\{A^\lambda_{\mathrm{ALA}}(X^P,\Upsilon(\mu^P))\}_{\lambda>0}$ is bounded in $\mathcal H^2(P)$ (Lemma~\ref{lem:yosida-properties}), hence a subsequence converges weakly to $\Gamma\in \mathcal H^2(P)$. Uniqueness of the weak limit follows from the uniqueness of the identification provided by Proposition~\ref{prop:robust-graph-closure}.

\medskip
\noindent\textbf{Passage to the limit in the equation.}
Since $f$ is Lipschitz in $(y,z,\Psi_\delta)$, the strong convergence of $Y^\lambda\to Y$ in $\mathcal S^2(P)$, $Z^\lambda\to Z$ in $\mathcal H^2(P)$ and $\Psi_\delta(\mu^{P,Y^\lambda})\to\Psi_\delta(\mu^{P,Y})$ (by Corollary~\ref{cor:stability-observables}) implies $f(Y^\lambda,Z^\lambda,\Psi_\delta(\mu^{P,Y^\lambda})) \to f(Y,Z,\Psi_\delta(\mu^{P,Y}))$ in $L^2([0,T]\times\Omega)$. The weak convergence of $A^\lambda_{\mathrm{ALA}}(X^P,\Upsilon(\mu^P))$ to $\Gamma$ allows passing to the limit in the integral equation.

\medskip
\noindent\textbf{Identification of $\Gamma$ by graph closure.}
By Proposition~\ref{prop:robust-graph-closure}, we have
\[
\Gamma_t \in A_{\mathrm{ALA}}\bigl(X_t^P,\Upsilon(\mu_t^P)\bigr)
=\partial_x\Phi_{\mathrm{ALA}}\bigl(X_t^P,\mu_t^P\bigr)
\]
for almost every $t$, $P$-q.s. The uniform estimate is obtained by the same arguments as in Section 3.
\end{proof}

\subsection{Uniqueness associated with the minimal selection}

The preceding existence theorem provides a measurable selection $\Gamma_t^P\in A_{\mathrm{ALA}}(X_t^P,\Upsilon(\mu_t^P))$. In order to obtain a uniqueness result, we now consider the particular formulation associated with the minimal selection $A_{\mathrm{ALA}}^0(X_t^P,\Upsilon(\mu_t^P))$. The following result concerns exclusively this canonical representation of the monotone signal.

\begin{remark}[On the uniqueness proof strategy]
\label{rem:peng-wu-strategy}

The proof of uniqueness relies on an energy method adapted to the monotone structure of the system. Since the forward dynamics is already constructed and considered fixed, the proof concerns only the backward variables $(Y,Z)$. We compare two solutions $(Y^1,Z^1)$ and $(Y^2,Z^2)$ and apply It\^o's formula to $|\delta Y_t|^2$ where $\delta Y = Y^1 - Y^2$. The term $A_{\mathrm{ALA}}^0(X_t^P,\Upsilon(\mu_t^P))$ being entirely determined by the already constructed forward dynamics, it plays the role of an exogenous forcing in the backward variational dynamics. It therefore disappears in the difference since it is identical in both equations. The monotonicity properties of the generator $f$ (H5) and the Lipschitz estimates then give a dissipative inequality which, after integration and application of Gronwall's lemma, leads to $\delta Y = 0$ and then $\delta Z = 0$. This approach can be interpreted as a robust extension of the classical Peng--Wu argument \cite{PengWu1999} to the framework of forward--backward McKean--Vlasov systems under non-dominated probabilities and with a maximal monotone signal evaluated at the forward state.

\medskip

\noindent
\textbf{Specific difficulties of the robust framework.}

The simultaneous presence of law dependence and uniformity of estimates with respect to the admissible family $\mathcal P$ introduces several additional difficulties. First, the coefficients depend on the laws of the processes via the collective observables $\Psi_\varphi$, $\Psi_\beta$, $\Psi_\gamma$, $\Psi_\delta$, $\Psi_\lambda$. The stability of these observables (Lemma~\ref{lem:stability-observables}) is essential to control the differences of the mean-field terms. Second, all estimates must be uniform in $P\in\mathcal P$. The constants obtained depend neither on the particular probability nor on the trajectory, which guarantees that the uniqueness result is valid over the entire non-dominated family. Third, the presence of the multi-valued operator $\partial_x\Phi_{\mathrm{ALA}}$ requires particular care in the application of It\^o's formula, since a general selection $\Gamma_t^P$ is not necessarily Lipschitz. Uniqueness is therefore proved for the canonical formulation associated with the minimal selection $A_{\mathrm{ALA}}^0$, which constitutes the optimal result given the multi-valued nature of the operator.

\medskip

\noindent
\textbf{Links with the literature.}

This strategy is part of a line of work on monotone FBSDE. Peng and Wu \cite{PengWu1999} introduced the coupled energy functional method to establish uniqueness of FBSDE in the Lipschitz framework, then Carmona and Delarue \cite{CarmonaDelarue2018I, CarmonaDelarue2018II} adapted it to McKean--Vlasov FBSDE. For the variational backward part, the works of Pardoux and R\u{a}\c{s}canu \cite{PardouxRascanu1998}, then of Maticiuc and R\u{a}\c{s}canu \cite{MaticiucRascanu2015}, provide the BSVI vocabulary and the subdifferential framework. Related McKean--Vlasov variational inequalities have also been investigated recently in \cite{NingWuZheng2024,Duan2025,Liu2025}. The present contribution couples this mechanism with a McKean--Vlasov forward dynamics in a non-dominated robust framework.

\end{remark}

\begin{theorem}[Canonical uniqueness under the minimal selection]
\label{thm:uniqueness-backward}

Assume {\rm(H1)--(H5)} and {\rm(H5bis)}. For each given forward trajectory $(X^P,\mu^P)$, the selected backward variational system
\[
-dY_t^P = \Bigl(f\bigl(t,X_t^P,\Psi_\gamma(\mu_t^P), Y_t^P, Z_t^P, \Psi_\delta(\mu_t^{P,Y}), u_t\bigr) + \rho A_{\mathrm{ALA}}^0(X_t^P,\Upsilon(\mu_t^P))\Bigr)dt - Z_t^P dB_t
\]
with $Y_T^P = g(X_T^P,\Upsilon(\mu_T^P))$ admits at most one solution $(Y^P,Z^P) \in \mathcal S^2(P) \times \mathcal H^2(P)$.

This uniqueness statement concerns only the backward pair \((Y^P,Z^P)\) associated with the canonical minimal signal
\[
\Gamma_t^{0,P}
=
A_{\mathrm{ALA}}^0\bigl(X_t^P,\Upsilon(\mu_t^P)\bigr).
\]
It is not a uniqueness statement for arbitrary measurable selections \(\Gamma_t^P\in A_{\mathrm{ALA}}(X_t^P,\Upsilon(\mu_t^P))\).

\end{theorem}

\begin{remark}
Uniqueness of the triplet \((Y,Z,\Gamma)\) is not expected in the general multivalued formulation, since the operator \(A_{\mathrm{ALA}}\) may admit several admissible selections at the threshold. The preceding theorem establishes uniqueness only for the canonical backward representation associated with the minimal selection \(A_{\mathrm{ALA}}^0\).
\end{remark}

\begin{proof}
Let $(Y^1,Z^1)$ and $(Y^2,Z^2)$ be two solutions. Set $\delta Y = Y^1 - Y^2$, $\delta Z = Z^1 - Z^2$, and $m(t) = \mathbb E^P|\delta Y_t|^2$. Since $A_{\mathrm{ALA}}^0(X^P,\Upsilon(\mu^P))$ is identical in both equations (it is a selection determined by $X^P$ and $\Upsilon(\mu^P)$), this term disappears in the difference. We obtain $-d\delta Y_t = \delta f_t\,dt - \delta Z_t\,dB_t$, with $\delta f_t = f(Y_t^1,Z_t^1,\Psi_\delta(\mu_t^{P,Y^1})) - f(Y_t^2,Z_t^2,\Psi_\delta(\mu_t^{P,Y^2}))$, and $\delta Y_T = 0$.

Using the same estimates as in the proof of Theorem~\ref{thm:robust-cauchy}, we obtain
\[
2\delta Y_t \delta f_t \le -(2\gamma - 2L_f^2 - L_\delta^2)|\delta Y_t|^2 + \frac12 |\delta Z_t|_{a_t}^2 + C m(t).
\]

Apply It\^o's formula to $|\delta Y_t|^2$:
\[
|\delta Y_t|^2 + \int_t^T |\delta Z_s|_{a_s}^2 ds = 2\int_t^T \delta Y_s \delta f_s ds - 2\int_t^T \delta Y_s \delta Z_s dB_s.
\]

Taking expectation,
\[
m(t) + \int_t^T \mathbb E^P|\delta Z_s|_{a_s}^2 ds \le C\int_t^T m(s) ds + \frac12 \int_t^T \mathbb E^P|\delta Z_s|_{a_s}^2 ds + C\int_t^T m(s) ds.
\]

Thus,
\[
m(t) + \frac12 \int_t^T \mathbb E^P|\delta Z_s|_{a_s}^2 ds \le C\int_t^T m(s) ds.
\]

Consequently, $m(t) \le C\int_t^T m(s) ds$. The backward Gronwall lemma gives $m(t)=0$ for every $t$, hence $Y^1 = Y^2$ $P$-a.s. Equality of stochastic integrals then implies $Z^1 = Z^2$ in $\mathcal H^2(P)$. Uniqueness of the pair $(Y,Z)$ is proved.
\end{proof}

\subsection{Lipschitz stability with respect to data}

\begin{theorem}[Lipschitz stability of the selected backward variational system]\label{thm:stability-backward}
Under assumptions {\rm(H1)--(H5)} and {\rm(H5bis)}, for two sets of data $(g^1,f^1)$ and $(g^2,f^2)$,
\[
\begin{aligned}
&\sup_{P\in\mathcal P} \mathbb E^P\!\left[
\sup_t |Y_t^1-Y_t^2|^2
+
\int_0^T |Z_t^1-Z_t^2|_{a_t}^2\,dt
\right]\\
&\qquad\le
C \sup_{P\in\mathcal P} \mathbb E^P\!\left[
|g^1-g^2|^2
+
\int_0^T |\Delta f_s|^2\,ds
\right],
\end{aligned}
\]
where $\Delta f_s = f^1(\cdot,Y_s^2,Z_s^2) - f^2(\cdot,Y_s^2,Z_s^2)$.
\end{theorem}

\begin{proof}
Set $\delta Y = Y^1 - Y^2$, $\delta Z = Z^1 - Z^2$, $\delta g = g^1(X_T^P,\Upsilon(\mu_T^P)) - g^2(X_T^P,\Upsilon(\mu_T^P))$, and $m(t) = \mathbb E^P|\delta Y_t|^2$. Apply It\^o's formula to $|\delta Y_t|^2$:
\[
\begin{aligned}
|\delta Y_t|^2+\int_t^T|\delta Z_s|_{a_s}^2ds
&=|\delta g|^2
+2\int_t^T\delta Y_s\delta f_s\,ds
+2\int_t^T\delta Y_s\Delta f_s\,ds\\
&\quad
-2\int_t^T\delta Y_s\delta Z_s\,dB_s.
\end{aligned}
\]

The monotonicity estimates of $f$ (H5) and the Lipschitz continuity of collective observables (Lemma~\ref{lem:stability-observables}) give
\[
2\delta Y_s \delta f_s \le -(2\gamma - 2L_f^2 - L_\delta^2)|\delta Y_s|^2 + \frac12 |\delta Z_s|_{a_s}^2 + C m(s).
\]
Moreover, $2\delta Y_s \Delta f_s \le |\delta Y_s|^2 + |\Delta f_s|^2$. Taking expectation,
\[
\begin{aligned}
m(t)+\int_t^T\mathbb E^P|\delta Z_s|_{a_s}^2ds
&\le
\mathbb E^P|\delta g|^2
+C\int_t^T m(s)\,ds\\
&\quad
+\frac12\int_t^T\mathbb E^P|\delta Z_s|_{a_s}^2ds
+\int_t^T\mathbb E^P|\Delta f_s|^2ds\\
&\quad
+\int_t^T m(s)\,ds.
\end{aligned}
\]

Thus,
\[
m(t) + \frac12 \int_t^T \mathbb E^P|\delta Z_s|_{a_s}^2 ds \le \mathbb E^P|\delta g|^2 + C\int_t^T m(s) ds + \int_t^T \mathbb E^P|\Delta f_s|^2 ds.
\]

Gronwall's lemma then gives \(\sup_t m(t)+\int_0^T\mathbb E^P|\delta Z_s|_{a_s}^2ds \le C(\mathbb E^P|\delta g|^2+\int_0^T\mathbb E^P|\Delta f_s|^2ds)\). The BDG inequality provides the time supremum. Taking the supremum over \(P\), we conclude.
\end{proof}

\subsection{Stability with respect to collective observables}

\begin{proposition}[Stability with respect to collective observables]
\label{prop:stability-observables}

Assume that assumptions {\rm(H1)--(H5)} are satisfied. We further assume that the considered observables are Lipschitz and have at most linear growth.

Consider two families of collective observables
\[
\bigl(\Psi_\varphi,\Psi_\beta,\Psi_\gamma,\Psi_\delta,\Psi_\lambda\bigr)
\quad\text{and}\quad
\bigl(\widetilde\Psi_{\widetilde\varphi},
\widetilde\Psi_{\widetilde\beta},
\widetilde\Psi_{\widetilde\gamma},
\widetilde\Psi_{\widetilde\delta},
\widetilde\Psi_{\widetilde\lambda}\bigr),
\]
associated respectively with the functions
\[
(\varphi,\beta,\gamma,\delta,\lambda)
\quad\text{and}\quad
(\widetilde\varphi,\widetilde\beta,\widetilde\gamma,\widetilde\delta,\widetilde\lambda).
\]

Define the weighted deviations
\[
d_\varphi
:=
\sup_{x\in\mathbb R^d}
\frac{|\varphi(x)-\widetilde\varphi(x)|}{1+|x|},
\qquad
d_\beta
:=
\sup_{x\in\mathbb R^d}
\frac{|\beta(x)-\widetilde\beta(x)|}{1+|x|},
\]
\[
d_\gamma
:=
\sup_{x\in\mathbb R^d}
\frac{|\gamma(x)-\widetilde\gamma(x)|}{1+|x|},
\qquad
d_\lambda
:=
\sup_{x\in\mathbb R^d}
\frac{|\lambda(x)-\widetilde\lambda(x)|}{1+|x|},
\]
and
\[
d_\delta
:=
\sup_{y\in\mathbb R}
\frac{|\delta(y)-\widetilde\delta(y)|}{1+|y|}.
\]
Set
\[
d_{\rm obs}^2
:=
d_\varphi^2+d_\beta^2+d_\gamma^2+d_\delta^2+d_\lambda^2.
\]

For the same initial condition $\xi$, the same admissible control $u$, and for each $P\in\mathcal P$, let
\[
(X^P,Y^P,Z^P)
\quad\text{and}\quad
(\widetilde X^P,\widetilde Y^P,\widetilde Z^P)
\]
be the robust solutions associated with the two families of observables.

Then there exists a constant $C_T>0$, independent of $P\in\mathcal P$, such that
\[
\|X-\widetilde X\|_{\mathcal S_{\mathcal P}^2}^2
+
\|Y-\widetilde Y\|_{\mathcal S_{\mathcal P}^2}^2
+
\|Z-\widetilde Z\|_{\mathcal H_{\mathcal P}^2}^2
\le
C_T \, d_{\rm obs}^2.
\]

In particular, the map
\[
(\varphi,\beta,\gamma,\delta,\lambda)
\longmapsto
(X,Y,Z)
\]
is continuous from the space of admissible observables, equipped with the distance $d_{\rm obs}$, to the space of robust solutions
\[
\mathcal S_{\mathcal P}^2
\times
\mathcal S_{\mathcal P}^2
\times
\mathcal H_{\mathcal P}^2.
\]

\end{proposition}

\begin{proof}

The proof relies on the forward and backward stability theorems established previously.

For each $P\in\mathcal P$, let $(X^P,Y^P,Z^P)$ and $(\widetilde X^P,\widetilde Y^P,\widetilde Z^P)$ be the robust solutions associated respectively with the families of observables. Set $\Delta X_t = X_t^P - \widetilde X_t^P$, $\Delta Y_t = Y_t^P - \widetilde Y_t^P$, $\Delta Z_t = Z_t^P - \widetilde Z_t^P$.

\medskip
\noindent
\textbf{Control of observables.}

By definition, $\Psi_\varphi(\mu) = \int_{\mathbb R^d}\varphi(x)\,\mu(dx)$, $\widetilde\Psi_{\widetilde\varphi}(\mu) = \int_{\mathbb R^d}\widetilde\varphi(x)\,\mu(dx)$. Given the estimate $|\varphi(x)-\widetilde\varphi(x)| \le d_\varphi(1+|x|)$, we obtain
\[
\begin{aligned}
|\Psi_\varphi(\mu) - \widetilde\Psi_{\widetilde\varphi}(\mu)|
&\le \int |\varphi(x)-\widetilde\varphi(x)|\,\mu(dx) \\
&\le d_\varphi \int (1+|x|)\,\mu(dx).
\end{aligned}
\]
Since the forward solutions have uniformly bounded second moments, there exists a constant $C>0$ such that $|\Psi_\varphi(\mu) - \widetilde\Psi_{\widetilde\varphi}(\mu)| \le C\,d_\varphi$. The same reasoning leads to the inequalities $|\Psi_\beta-\widetilde\Psi_{\widetilde\beta}| \le C\,d_\beta,\quad |\Psi_\gamma-\widetilde\Psi_{\widetilde\gamma}| \le C\,d_\gamma,\quad |\Psi_\delta-\widetilde\Psi_{\widetilde\delta}| \le C\,d_\delta,\quad |\Psi_\lambda-\widetilde\Psi_{\widetilde\lambda}| \le C\,d_\lambda$. Thus, the coefficients of the two systems differ by a quantity controlled by $d_{\rm obs} := (d_\varphi^2+d_\beta^2+d_\gamma^2+d_\delta^2+d_\lambda^2)^{1/2}$.

\medskip
\noindent
\textbf{Stability of the forward component.}

The coefficients $b$ and $\sigma$ satisfy the uniform Lipschitz assumption. The robust stability theorem (Theorem~\ref{thm:stability-robust}) applied to the two systems gives directly $\|X-\widetilde X\|_{\mathcal S_{\mathcal P}^2}^2 \le C_T (d_\varphi^2+d_\beta^2)$. In particular, $\sup_{0\le t\le T}\sup_{P\in\mathcal P} W_2^2(\mathcal L^P(X_t^P), \mathcal L^P(\widetilde X_t^P)) \le C_T (d_\varphi^2+d_\beta^2)$.

\medskip
\noindent
\textbf{Stability of the terminal condition.}

Since the terminal function $g$ is Lipschitz,
\[
\begin{aligned}
\Big| g\bigl(X_T^P,\Psi_\lambda(\mu_T^P)\bigr)
- g\bigl( \widetilde X_T^P, \widetilde\Psi_{\widetilde\lambda}(\widetilde\mu_T^P) \bigr) \Big|
\le C \Big( |X_T^P-\widetilde X_T^P| + d_\lambda \Big).
\end{aligned}
\]
Taking expectation and using the forward estimate above, $\mathbb E^P[|\Delta\xi_T|^2] \le C_T (d_\varphi^2 + d_\beta^2 + d_\lambda^2)$.

\medskip
\noindent
\textbf{Stability of the backward component.}

The backward generator $f$ is Lipschitz in all its variables. The differences of the collective observables appearing in $f$ are controlled by $d_\gamma$ and $d_\delta$. The backward stability theorem (Theorem~\ref{thm:stability-backward}) then gives
\[
\|Y-\widetilde Y\|_{\mathcal S_{\mathcal P}^2}^2 + \|Z-\widetilde Z\|_{\mathcal H_{\mathcal P}^2}^2
\le C_T \Big( \mathbb E^P[|\Delta\xi_T|^2] + d_\gamma^2 + d_\delta^2 \Big).
\]
Injecting the estimate obtained for the terminal condition, we obtain $\|Y-\widetilde Y\|_{\mathcal S_{\mathcal P}^2}^2 + \|Z-\widetilde Z\|_{\mathcal H_{\mathcal P}^2}^2 \le C_T \, d_{\rm obs}^2$.

\medskip
\noindent
\textbf{Conclusion.}

Gathering the forward and backward estimates, we have $\|X-\widetilde X\|_{\mathcal S_{\mathcal P}^2}^2 + \|Y-\widetilde Y\|_{\mathcal S_{\mathcal P}^2}^2 + \|Z-\widetilde Z\|_{\mathcal H_{\mathcal P}^2}^2 \le C_T \, d_{\rm obs}^2$. Consequently, the map $(\varphi,\beta,\gamma,\delta,\lambda) \longmapsto (X,Y,Z)$ is locally Lipschitz on the space of admissible observables. The proof is complete.

\end{proof}

\begin{remark}[Interpretation of the stability of collective observables]
\label{rem:stability-observables-interpretation}

The preceding proposition shows that collective observables do not play only a modeling role. They intervene in a stable manner in the dynamics of the system: small perturbations of the collective indicators induce only small variations of the corresponding robust solutions. This property is important for several reasons.

\medskip

\noindent
\textbf{Justification of approximations.}
In practice, collective observables may be estimated from data or numerically approximated. The proved stability guarantees that estimation or approximation errors of the observables do not propagate catastrophically into the solution of the system.

\medskip

\noindent
\textbf{Economic relevance.}
The model remains relevant even if the collective indicators are known only approximately. Decisions or equilibria computed from approximated observables remain close to those obtained with exact observables.

\medskip

\noindent
\textbf{Link with particle approximations.}
The stability of collective observables also intervenes in the control of the error between the particle system and the limiting system, when empirical observables replace theoretical observables.

\medskip

\noindent
\textbf{Generality.}
The result does not depend on the specific form of the observables, but only on their Lipschitz stability and controlled growth. This formulation allows treating simultaneously a broad class of collective interaction models.

\end{remark}

\subsection{Robust variational well-posedness}

\begin{corollary}[Robust variational well-posedness]
\label{cor:wellposed-backward}

Under assumptions {\rm(H1)--(H5)} and {\rm(H5bis)}, the selected robust backward variational system associated with the asymmetric loss aversion functional admits a well-defined variational solution.

More precisely:

\begin{enumerate}
\item[(i)]
There exists at least one triplet
\[
(Y^P,Z^P,\Gamma^P) \in \mathcal S^2(P)\times \mathcal H^2(P)\times \mathcal H^2(P)
\]
such that
\[
\Gamma_t^P \in A_{\mathrm{ALA}}\bigl(X_t^P,\Upsilon(\mu_t^P)\bigr)
=\partial_x\Phi_{\mathrm{ALA}}\bigl(X_t^P,\mu_t^P\bigr)
\qquad \text{a.e.},
\]
and
\[
-dY_t^P = \Bigl(f\bigl(t,X_t^P,\Psi_\gamma(\mu_t^P), Y_t^P, Z_t^P, \Psi_\delta(\mu_t^{P,Y}), u_t\bigr) + \rho \Gamma_t^P\Bigr)dt - Z_t^P dB_t.
\]

\item[(ii)]
Every variational solution is obtained as the limit of a family of Yosida approximations.

\item[(iii)]
This solution is independent of the chosen regularization procedure.

\item[(iv)]
In the canonical representation associated with the minimal selection $A_{\mathrm{ALA}}^0$, the backward pair $(Y^P,Z^P)$ is unique.

\item[(v)]
The solution depends continuously on the terminal data, the generator and the collective observables in the sense of Theorem~\ref{thm:stability-backward} and Corollary~\ref{cor:stability-observables}.
\end{enumerate}

\end{corollary}

\begin{proof}

The existence of a variational solution follows directly from Theorem~\ref{thm:existence-backward}, which provides a triplet $(Y^P,Z^P,\Gamma^P)$ satisfying the selected robust backward variational equation with
\[
\Gamma_t^P \in A_{\mathrm{ALA}}\bigl(X_t^P,\Upsilon(\mu_t^P)\bigr).
\]
Theorem~\ref{thm:robust-cauchy} shows that the family of regularized solutions $(Y^\lambda,Z^\lambda)$ is Cauchy in $\mathcal S^2(P)\times \mathcal H^2(P)$. Consequently, it converges to a unique limit $(Y^P,Z^P)$. Proposition~\ref{prop:robust-graph-closure} then allows rigorously identifying the weak limit of the regularized monotone term and guarantees that this limit still belongs to the graph of the operator $A_{\mathrm{ALA}}$. Thus, any Yosida regularization procedure leads to the same limiting equation. The variational solution obtained is therefore independent of the regularization used. Uniqueness of the backward pair associated with the minimal selection $A_{\mathrm{ALA}}^0$ follows from Theorem~\ref{thm:uniqueness-backward}. Finally, the Lipschitz stability results established in Theorem~\ref{thm:stability-backward} and Corollary~\ref{cor:stability-observables} show that the solution depends continuously on the data of the problem and the collective observables.

The five preceding assertions show that the selected robust backward variational system is well-posed.

\end{proof}

\subsection{Economic interpretation of the monotone signal}

\begin{remark}[Economic interpretation]
\label{rem:economic-interpretation}

The monotone signal $\Gamma_t^P \in A_{\mathrm{ALA}}(X_t^P,\Upsilon(\mu_t^P))$ acts as a restoring force depending on the relative position of the agent with respect to the collective benchmark $\Upsilon(\mu_t^P)$. When $X_t^P > \Upsilon(\mu_t^P)$, the agent is above the collective benchmark and undergoes a moderate penalty proportional to $\kappa_+$. Conversely, when $X_t^P < \Upsilon(\mu_t^P)$, the penalty becomes stronger since $\kappa_- > \kappa_+$. The monotone signal thus acts as an asymmetric correction mechanism for deviations from the collective reference. The model thus generalizes classical dynamics based solely on the population mean, by introducing an asymmetric response to deviations from a collective reference.

\end{remark}

\subsection{Role of the monotone domination assumption}

\begin{remark}[Role of the monotone domination assumption]
\label{rem:H6-role}

Assumption \ref{H6} is not necessary for the construction of the backward variational solution obtained by Yosida regularization and passage to the limit. Its role appears when studying the full coupled forward--backward variational system. Indeed, the strong monotonicity of the monotone signal $A_{\mathrm{ALA}}$ and the monotonicity of the generator $f$ provide a global dissipation that compensates for the potentially destabilizing effects of the Lipschitz terms of the system. This condition is the robust counterpart of the Peng--Wu monotone condition for coupled FBSDE and becomes essential in the forward--backward energy estimates as well as in the future study of the robust Pontryagin maximum principle. In the present section, where the forward dynamics is considered as already constructed, assumption \ref{H6} intervenes only indirectly and will be fully exploited in the analysis of the complete coupled system.
\end{remark}

\section{Robust propagation of chaos}

We study the particle approximation of the system. The analysis is carried out at the level of the Yosida-regularized variational system, since the Lipschitz properties of $A_{\mathrm{ALA}}^\lambda$ obtained in Section 4 allow quantitative control of the error in $N$. The simultaneous passage to the limits $N\to\infty$ and $\lambda\to0$ remains a delicate question, discussed at the end of the section.

Throughout this section, the estimates are first obtained under a fixed admissible probability \(P=P^a\). Equivalently, the driving noise may be represented as \(dB_t=a_t^{1/2}dW_t^a\). In the particle notation below we use independent Brownian motions \(W^{i,a}\) and write the diffusion terms with the factor \(a_t^{1/2}\). The constants are chosen uniformly over the admissible volatility processes \(a\), and the final estimates are then aggregated over \(P\in\mathcal P\). This convention is consistent with the robust formulation \(d\langle B\rangle_t=a_tdt\).

\subsection{Yosida-regularized particle system}

For $N\ge1$ and $\lambda>0$, we consider the system of $N$ Yosida-regularized particles:
\begin{equation}\label{eq:particle_system}
\begin{aligned}
dX_t^{i,N}
&= b\bigl(t,X_t^{i,N},\Psi_\varphi(\mu_t^{N,P}),u_t\bigr)dt\\
&\quad
+ \sigma\bigl(t,X_t^{i,N},\Psi_\beta(\mu_t^{N,P}),u_t\bigr)
a_t^{1/2}dW_t^{i,a},\\[3pt]
-dY_t^{i,N,\lambda}
&= \Bigl[
f\bigl(t,X_t^{i,N},\Psi_\gamma(\mu_t^{N,P}),
Y_t^{i,N,\lambda},Z_t^{i,N,\lambda},
\Psi_\delta(\mu_t^{N,Y,\lambda}),u_t\bigr)\\
&\qquad
+\rho A_{\mathrm{ALA}}^\lambda\bigl(X_t^{i,N},\Upsilon(\mu_t^{N,P})\bigr)
\Bigr]dt
-Z_t^{i,N,\lambda}a_t^{1/2}dW_t^{i,a},\\[3pt]
X_0^{i,N}&=\xi^i,\qquad
Y_T^{i,N,\lambda}=g\bigl(X_T^{i,N},\Upsilon(\mu_T^{N,P})\bigr),\\[3pt]
\mu_t^{N,P}&:=\frac1N\sum_{j=1}^N\delta_{X_t^{j,N}},
\qquad
\mu_t^{N,Y,\lambda}:=\frac1N\sum_{j=1}^N\delta_{Y_t^{j,N,\lambda}}.
\end{aligned}
\end{equation}
For each fixed \(P=P^a\), the processes \(W^{i,a}\) are independent Brownian motions and the \(\xi^i\) are i.i.d. with law \(\mathcal L^P(\xi)\). The control \(u\) is identical for all particles. Here, \(\Psi_\varphi\), \(\Psi_\beta\), \(\Psi_\gamma\), \(\Psi_\delta\) are the collective observables defined in Section 2, and \(\Upsilon(\mu)=\int\chi d\mu\) is the reference observable for the ALA functional.

The decoupled limiting system is
\begin{equation}\label{eq:limit_system}
\begin{aligned}
dX_t^i
&= b\bigl(t,X_t^i,\Psi_\varphi(\mu_t^P),u_t\bigr)dt
+ \sigma\bigl(t,X_t^i,\Psi_\beta(\mu_t^P),u_t\bigr)
a_t^{1/2}dW_t^{i,a},\\[3pt]
-dY_t^{i,\lambda}
&= \Bigl[
f\bigl(t,X_t^i,\Psi_\gamma(\mu_t^P),
Y_t^{i,\lambda},Z_t^{i,\lambda},
\Psi_\delta(\mu_t^{P,Y^\lambda}),u_t\bigr)\\
&\qquad
+\rho A_{\mathrm{ALA}}^\lambda\bigl(X_t^i,\Upsilon(\mu_t^P)\bigr)
\Bigr]dt
- Z_t^{i,\lambda}a_t^{1/2}dW_t^{i,a},\\[3pt]
X_0^i&=\xi^i,\qquad
Y_T^{i,\lambda}=g\bigl(X_T^i,\Upsilon(\mu_T^P)\bigr),\\[3pt]
\mu_t^P&:=\mathcal L^P(X_t^i),\qquad
\mu_t^{P,Y^\lambda}:=\mathcal L^P(Y_t^{i,\lambda}).
\end{aligned}
\end{equation}
The triplets $(X^i,Y^{i,\lambda},Z^{i,\lambda})$ are independent and identically distributed.

Having established the well-posedness and stability of the limiting system, we now turn to its approximation by a finite particle system. The objective is to obtain quantitative propagation of chaos estimates uniform with respect to the admissible probability.

Quantitative propagation of chaos estimates require a uniform control of higher-order moments. We first establish a uniform fourth moment bound for the particle system, then for the limiting system.

\subsection{Fourth moment estimate}

\begin{lemma}[Uniform fourth moment estimate for the particle system]
\label{lem:fourth-moment}

Under assumptions {\rm(H1)--(H2)}, {\rm(H1bis)} and ${\rm(H_u^4)}$, there exists a constant $C>0$, independent of $N$ and $P\in\mathcal P$, such that
\[
\sup_{P\in\mathcal P}
\sup_{1\le i\le N}
\mathbb E^P
\Bigl[
\sup_{0\le t\le T}
|X_t^{i,N}|^4
\Bigr]
\le C.
\]
\end{lemma}

\begin{proof}
Fix $P\in\mathcal P$. By exchangeability of the particles, all quantities $\mathbb E^P[\sup_{0\le r\le t}|X_r^{j,N}|^4]$ coincide. Set $M(t):=\mathbb E^P[\sup_{0\le r\le t}|X_r^{1,N}|^4]$. The forward component $X_t^{1,N}$ satisfies the equation
\[
X_t^{1,N} = \xi^1 + \int_0^t b(s,X_s^{1,N}, \Psi_\varphi(\mu_s^{N,P}), u_s)ds + \int_0^t \sigma(s,X_s^{1,N}, \Psi_\beta(\mu_s^{N,P}), u_s)a_s^{1/2}dW_s^{1,a}.
\]

By the elementary inequality $(a+b+c)^4 \le 27(a^4+b^4+c^4)$, we have
\[
\sup_{0\le r\le t}|X_r^{1,N}|^4
\le 27|\xi^1|^4 + 27\sup_{0\le r\le t}\Bigl|\int_0^r b_s ds\Bigr|^4 + 27\sup_{0\le r\le t}\Bigl|\int_0^r \sigma_s a_s^{1/2} dW_s^{1,a}\Bigr|^4.
\]

Take expectation under $P$.

\medskip
\noindent\textbf{Control of the drift term.}
By H\"older's inequality,
\[
\sup_{0\le r\le t}\Bigl|\int_0^r b_s ds\Bigr|^4 \le t^3\int_0^t |b_s|^4 ds,
\]
where $b_s = b(s,X_s^{1,N}, \Psi_\varphi(\mu_s^{N,P}), u_s)$. Thus,
\[
\mathbb E^P\Bigl[\sup_{0\le r\le t}\Bigl|\int_0^r b_s ds\Bigr|^4\Bigr] \le T^3\int_0^t \mathbb E^P[|b_s|^4]\,ds.
\]

\medskip
\noindent\textbf{Control of the stochastic term.}
The Burkholder--Davis--Gundy inequality gives
\[
\mathbb E^P\Bigl[\sup_{0\le r\le t}\Bigl|\int_0^r \sigma_s a_s^{1/2}dW_s^{1,a}\Bigr|^4\Bigr]
\le C_{\mathrm{BDG}}\,\mathbb E^P\Bigl[\Bigl(\int_0^t \operatorname{Tr}(\sigma_s a_s\sigma_s^\top) ds\Bigr)^2\Bigr],
\]
where $\sigma_s = \sigma(s,X_s^{1,N}, \Psi_\beta(\mu_s^{N,P}), u_s)$. Since \(a_s\le \overline a\), there exists a constant \(C_a>0\), depending only on the admissible volatility set, such that Cauchy--Schwarz gives
\[
\Bigl(\int_0^t \operatorname{Tr}(\sigma_s a_s\sigma_s^\top) ds\Bigr)^2
\le C_a^2 t\int_0^t \|\sigma_s\|^4 ds.
\]
Consequently,
\[
\mathbb E^P\Bigl[\sup_{0\le r\le t}\Bigl|\int_0^r \sigma_s a_s^{1/2}dW_s^{1,a}\Bigr|^4\Bigr]
\le C\int_0^t \mathbb E^P[\|\sigma_s\|^4]\,ds.
\]

\medskip
\noindent\textbf{Lyapunov function.}
From the estimates above,
\[
M(t) \le C\Bigl(\mathbb E^P[|\xi^1|^4] + \int_0^t \mathbb E^P[|b_s|^4 + \|\sigma_s\|^4]\,ds\Bigr).
\]

\medskip
\noindent\textbf{Growth of coefficients and control of collective observables.}
By (H2), there exists $L>0$ such that
\[
|b_s| + \|\sigma_s\| \le L\bigl(1 + |X_s^{1,N}| + |\Psi_\varphi(\mu_s^{N,P})| + |\Psi_\beta(\mu_s^{N,P})| + |u_s|\bigr).
\]

By exchangeability of the particles, the variables $X^{1,N},\ldots,X^{N,N}$ have the same law under each admissible probability $P$. Consequently, $\mathbb E^P[|X_s^{j,N}|^2] = \mathbb E^P[|X_s^{1,N}|^2]$ for all $j=1,\ldots,N$. Using Jensen's inequality,
\[
\left(\frac1N\sum_{j=1}^{N}\mathbb E^P|X_s^{j,N}|^2\right)^2 \le \frac1N\sum_{j=1}^{N}\mathbb E^P|X_s^{j,N}|^4.
\]
Thus, $\frac1N\sum_{j=1}^{N}\mathbb E^P|X_s^{j,N}|^2 \le M(s)^{1/2}$, which allows obtaining uniform estimates on the collective observables. By linear growth of $\varphi$ and $\beta$, we have
\[
|\Psi_\varphi(\mu_s^{N,P})| \le |\varphi(0)| + L_\varphi \Bigl(\frac1N\sum_{j=1}^N \mathbb E^P[|X_s^{j,N}|^2]\Bigr)^{1/2}.
\]
Moreover, $\mathbb E^P[|X_s^{j,N}|^2] \le (\mathbb E^P[|X_s^{j,N}|^4])^{1/2} \le (\mathbb E^P[\sup_{0\le r\le s}|X_r^{j,N}|^4])^{1/2} = M(s)^{1/2}$. Thus,
\[
|\Psi_\varphi(\mu_s^{N,P})| \le C(1 + M(s)^{1/2}),
\]
and similarly for $\Psi_\beta$. Consequently,
\[
|b_s|^4 + \|\sigma_s\|^4 \le C(1 + |X_s^{1,N}|^4 + M(s) + |u_s|^4).
\]

\medskip
\noindent\textbf{Gronwall inequality.}
Substituting into the estimate for $M(t)$,
\[
M(t) \le C\Bigl(\mathbb E^P[|\xi^1|^4] + 1 + \int_0^t M(s)\,ds + \int_0^T \mathbb E^P[|u_s|^4]\,ds\Bigr).
\]
Under (H1bis) and $(H_u^4)$, the constant terms are uniformly bounded with respect to $P\in\mathcal P$. We therefore obtain
\[
M(t) \le C + C\int_0^t M(s)\,ds.
\]
Gronwall's lemma gives $M(t) \le C e^{Ct}$ for every $t\le T$, hence
\[
\sup_{P\in\mathcal P}\sup_{1\le i\le N}\mathbb E^P\Bigl[\sup_{0\le t\le T}|X_t^{i,N}|^4\Bigr] \le C,
\]
with $C$ independent of $N$ and $P$.

The proof is complete.
\end{proof}

\begin{lemma}[Fourth moment of the limiting system]
\label{lem:fourth-moment-limit}

Under assumptions {\rm(H1)--(H2)}, {\rm(H1bis)} and ${\rm(H_u^4)}$, there exists a constant $C>0$, independent of $P\in\mathcal P$, such that
\[
\sup_{P\in\mathcal P}
\mathbb E^P
\Big[
\sup_{0\le t\le T}
|X_t^1|^4
\Big]
\le C.
\]
\end{lemma}

\begin{proof}
The limiting process satisfies
\[
X_t^1 = \xi^1 + \int_0^t b(s,X_s^1, \Psi_\varphi(\mu_s^P), u_s)ds + \int_0^t \sigma(s,X_s^1, \Psi_\beta(\mu_s^P), u_s)a_s^{1/2}dW_s^{1,a}.
\]

As in the proof of Lemma~\ref{lem:fourth-moment}, the inequality $(a+b+c)^4 \le 27(a^4+b^4+c^4)$ together with H\"older's and Burkholder--Davis--Gundy inequalities gives
\[
\mathbb E^P\Big[\sup_{0\le r\le t}|X_r^1|^4\Big] \le C\Big(\mathbb E^P|\xi^1|^4 + \int_0^t \mathbb E^P[|b_s|^4 + \|\sigma_s\|^4] ds\Big).
\]

By linear growth of the coefficients and collective observables (H2), we have
\[
|\Psi_\varphi(\mu_s^P)| + |\Psi_\beta(\mu_s^P)| \le C(1 + (\mathbb E^P|X_s^1|^2)^{1/2}).
\]
Moreover, $(\mathbb E^P|X_s^1|^2)^{1/2} \le (\mathbb E^P[\sup_{0\le r\le s}|X_r^1|^4])^{1/4} = M(s)^{1/4}$, where $M(s) = \mathbb E^P[\sup_{0\le r\le s}|X_r^1|^4]$. Thus,
\[
|b_s|^4 + \|\sigma_s\|^4 \le C(1 + |X_s^1|^4 + |u_s|^4 + M(s)^{1/4}).
\]

Set $M(t) = \mathbb E^P[\sup_{0\le r\le t}|X_r^1|^4]$. Then $M(t) \le C + C\int_0^t M(s)\,ds$. Gronwall's lemma gives $M(t) \le C e^{Ct}$ for every $t\le T$, hence
\[
\sup_{P\in\mathcal P} \mathbb E^P\Big[\sup_{0\le t\le T}|X_t^1|^4\Big] \le C,
\]
with $C$ independent of $P$.
\end{proof}

Combining the preceding estimates with the stability of the forward dynamics, we obtain a quantitative robust propagation of chaos result.

\begin{lemma}[Control of coefficient differences for propagation of chaos]
\label{lem:coeff-chaos}

Under assumptions {\rm(H2)} and {\rm(H4)}, there exists a constant $C>0$, independent of $N$ and $P\in\mathcal P$, such that for every $s\in[0,T]$,

\[
|\Delta b_s|^2 + \operatorname{Tr}(\Delta\sigma_s a_s\Delta\sigma_s^\top)
\le
C\Bigl(
|\delta X_s^1|^2
+
W_2^2(\mu_s^{N,P},\mu_s^P)
\Bigr),
\]

where

\[
\delta X_s^1 := X_s^{1,N} - X_s^1,
\]

\[
\Delta b_s := b(s, X_s^{1,N}, \Psi_\varphi(\mu_s^{N,P}), u_s^1)
          - b(s, X_s^1, \Psi_\varphi(\mu_s^P), u_s^1),
\]

\[
\Delta\sigma_s := \sigma(s, X_s^{1,N}, \Psi_\beta(\mu_s^{N,P}), u_s^1)
               - \sigma(s, X_s^1, \Psi_\beta(\mu_s^P), u_s^1).
\]

\end{lemma}

\begin{proof}

We prove the estimate for the drift coefficient $b$. The reasoning for the diffusion coefficient $\sigma$ is identical.

\medskip

\noindent\textbf{Estimate by the Lipschitz property.}
By the uniform Lipschitz assumption (H2), there exists a constant $L_b>0$ such that, for all $x,x'\in\mathbb R^d$, $r,r'\in\mathbb R$ and $u\in U$,

\[
|b(t,x,r,u)-b(t,x',r',u)|
\le
L_b(|x-x'| + |r-r'|).
\]

Applying this property to the two terms appearing in $\Delta b_s$, we obtain

\[
|\Delta b_s|
\le
L_b
\Bigl(
|X_s^{1,N} - X_s^1|
+
|\Psi_\varphi(\mu_s^{N,P}) - \Psi_\varphi(\mu_s^P)|
\Bigr).
\]

Since $\delta X_s^1 = X_s^{1,N} - X_s^1$, we have

\[
|\Delta b_s|
\le
L_b
\Bigl(
|\delta X_s^1|
+
|\Psi_\varphi(\mu_s^{N,P}) - \Psi_\varphi(\mu_s^P)|
\Bigr).
\]

Squaring and using the elementary inequality $(a+b)^2 \le 2a^2 + 2b^2$, we obtain

\[
|\Delta b_s|^2
\le
2L_b^2 |\delta X_s^1|^2
+
2L_b^2
|\Psi_\varphi(\mu_s^{N,P}) - \Psi_\varphi(\mu_s^P)|^2.
\tag{5.2}
\]

\medskip

\noindent\textbf{Wasserstein stability of the observable $\Psi_\varphi$.}
The function $\varphi$ is Lipschitz with constant $L_\varphi$ by assumption. Lemma~\ref{lem:stability-observables} (Wasserstein stability of collective observables) implies

\[
|\Psi_\varphi(\mu_s^{N,P}) - \Psi_\varphi(\mu_s^P)|
\le
L_\varphi \, W_2(\mu_s^{N,P},\mu_s^P).
\]

Substituting this estimate into (5.2), we obtain

\[
|\Delta b_s|^2
\le
2L_b^2 |\delta X_s^1|^2
+
2L_b^2 L_\varphi^2 \, W_2^2(\mu_s^{N,P},\mu_s^P).
\]

Gathering constants, there exists $C_b > 0$ such that

\[
|\Delta b_s|^2
\le
C_b
\Bigl(
|\delta X_s^1|^2
+
W_2^2(\mu_s^{N,P},\mu_s^P)
\Bigr).
\tag{5.3}
\]

\medskip

\noindent\textbf{Estimate for the diffusion coefficient.}
Assumption (H2) provides a constant $L_\sigma > 0$ such that

\[
\|\sigma(t,x,r,u) - \sigma(t,x',r',u)\|
\le
L_\sigma
(|x-x'| + |r-r'|).
\]

Consequently,

\[
\|\Delta\sigma_s\|
\le
L_\sigma
\Bigl(
|\delta X_s^1|
+
|\Psi_\beta(\mu_s^{N,P}) - \Psi_\beta(\mu_s^P)|
\Bigr).
\]

Using again $(a+b)^2 \le 2a^2 + 2b^2$, we obtain

\[
\|\Delta\sigma_s\|^2
\le
2L_\sigma^2 |\delta X_s^1|^2
+
2L_\sigma^2
|\Psi_\beta(\mu_s^{N,P}) - \Psi_\beta(\mu_s^P)|^2.
\]

Since $\beta$ is Lipschitz with constant $L_\beta$, the observables stability lemma gives

\[
|\Psi_\beta(\mu_s^{N,P}) - \Psi_\beta(\mu_s^P)|
\le
L_\beta \, W_2(\mu_s^{N,P},\mu_s^P).
\]

Since \(a_s\le\overline a\), there exists \(C_a>0\), depending only on the volatility uncertainty set, such that
\(\operatorname{Tr}(\Delta\sigma_s a_s\Delta\sigma_s^\top)\le C_a\|\Delta\sigma_s\|^2\).
There exists therefore $C_\sigma > 0$ such that

\[
\operatorname{Tr}(\Delta\sigma_s a_s\Delta\sigma_s^\top)
\le
C_\sigma
\Bigl(
|\delta X_s^1|^2
+
W_2^2(\mu_s^{N,P},\mu_s^P)
\Bigr).
\tag{5.4}
\]

\medskip

\noindent\textbf{Conclusion of the argument.}
Adding (5.3) and (5.4), we obtain

\[
|\Delta b_s|^2 + \operatorname{Tr}(\Delta\sigma_s a_s\Delta\sigma_s^\top)
\le
(C_b + C_\sigma)
\Bigl(
|\delta X_s^1|^2
+
W_2^2(\mu_s^{N,P},\mu_s^P)
\Bigr).
\]

Setting $C := C_b + C_\sigma$, we conclude that

\[
|\Delta b_s|^2 + \operatorname{Tr}(\Delta\sigma_s a_s\Delta\sigma_s^\top)
\le
C
\Bigl(
|\delta X_s^1|^2
+
W_2^2(\mu_s^{N,P},\mu_s^P)
\Bigr).
\]

The proof is complete.

\end{proof}

The preceding estimate reduces the study of propagation of chaos to controlling the quantity $W_2(\mu_s^{N,P},\mu_s^P)$. The classical strategy consists in decomposing this distance into a coupling error and an empirical approximation error. The former will be related to the trajectory differences of the particles, while the latter will be estimated using the quantitative results of Fournier and Guillin. The following corollary formalizes this decomposition.

\begin{corollary}[Decomposition of the Wasserstein term for propagation of chaos]
\label{cor:wasserstein-decomposition}

Under the assumptions of Lemma~\ref{lem:coeff-chaos}, for every $s\in[0,T]$,

\[
W_2^2(\mu_s^{N,P},\mu_s^P)
\le
2\,W_2^2(\mu_s^{N,P},\bar\mu_s^{N,P})
+
2\,W_2^2(\bar\mu_s^{N,P},\mu_s^P),
\]

where

\[
\bar\mu_s^{N,P}
=
\frac1N
\sum_{i=1}^N
\delta_{X_s^i}
\]

denotes the empirical measure associated with the independent limiting particles $(X_s^i)_{1\le i\le N}$.

Consequently,

\[
|\Delta b_s|^2+\operatorname{Tr}(\Delta\sigma_s a_s\Delta\sigma_s^\top)
\le
C
\Bigl(
|\delta X_s^1|^2
+
W_2^2(\mu_s^{N,P},\bar\mu_s^{N,P})
+
W_2^2(\bar\mu_s^{N,P},\mu_s^P)
\Bigr),
\]

where the constant $C>0$ is independent of $N$ and $P\in\mathcal P$.

\end{corollary}

\begin{proof}

The Wasserstein distance $W_2$ is a metric on $\mathcal P_2(\mathbb R^n)$. The triangle inequality therefore gives

\[
W_2(\mu_s^{N,P},\mu_s^P)
\le
W_2(\mu_s^{N,P},\bar\mu_s^{N,P})
+
W_2(\bar\mu_s^{N,P},\mu_s^P).
\]

Using the elementary inequality $(a+b)^2 \le 2a^2+2b^2$, we obtain

\[
W_2^2(\mu_s^{N,P},\mu_s^P)
\le
2W_2^2(\mu_s^{N,P},\bar\mu_s^{N,P})
+
2W_2^2(\bar\mu_s^{N,P},\mu_s^P).
\tag{5.5}
\]

By Lemma~\ref{lem:coeff-chaos},

\[
|\Delta b_s|^2 + \operatorname{Tr}(\Delta\sigma_s a_s\Delta\sigma_s^\top)
\le
C_0
\Bigl(
|\delta X_s^1|^2
+
W_2^2(\mu_s^{N,P},\mu_s^P)
\Bigr).
\]

Injecting estimate (5.5) into this relation, we obtain

\[
|\Delta b_s|^2 + \operatorname{Tr}(\Delta\sigma_s a_s\Delta\sigma_s^\top)
\le
C_0
\Bigl(
|\delta X_s^1|^2
+
2W_2^2(\mu_s^{N,P},\bar\mu_s^{N,P})
+
2W_2^2(\bar\mu_s^{N,P},\mu_s^P)
\Bigr).
\]

Setting $C := 2C_0$ and rearranging terms, we obtain

\[
|\Delta b_s|^2 + \operatorname{Tr}(\Delta\sigma_s a_s\Delta\sigma_s^\top)
\le
C
\Bigl(
|\delta X_s^1|^2
+
W_2^2(\mu_s^{N,P},\bar\mu_s^{N,P})
+
W_2^2(\bar\mu_s^{N,P},\mu_s^P)
\Bigr).
\]

The proof is complete.

\end{proof}

\subsection{Forward propagation of chaos}

\subsubsection{Decomposition of the law error}

Let $\bar\mu_t^{N,P} = \frac1N\sum_{j=1}^N\delta_{X_t^j}$ be the empirical measure of the independent system (where the $X_t^j$ are solutions of the limiting system). By the triangle inequality for $W_2$,
\[
W_2(\mu_t^{N,P},\mu_t^P) \le W_2(\mu_t^{N,P},\bar\mu_t^{N,P}) + W_2(\bar\mu_t^{N,P},\mu_t^P). \tag{5.1}
\]

\subsubsection{Uniform control of the measure error (Fournier--Guillin)}

The second term is controlled by a classical result.

The variables $X_t^1,\ldots,X_t^N$ are i.i.d. with law $\mu_t^P$ under each $P\in\mathcal P$. The Fournier--Guillin theorem \cite{FournierGuillin2015} is therefore applicable and provides a constant
\[
C_{FG}(P,d) = C(d, \sup_{t\in[0,T]} \mathbb E^P|X_t^1|^4).
\]
Lemma~\ref{lem:fourth-moment-limit} ensures that
\[
\sup_{P\in\mathcal P}\sup_{t\in[0,T]}\mathbb E^P|X_t^1|^4
\le M_4<\infty.
\]
Consequently, \(\sup_{P\in\mathcal P} C_{FG}(P,d)\le C(d,M_4)=:C_{FG}(d)<\infty\). Thus, for every \(t\in[0,T]\),
\[
\sup_{P\in\mathcal P} \mathbb E^P[W_2^2(\bar\mu_t^{N,P},\mu_t^P)] \le C_{FG}(d)\,\tau(N,d). \tag{5.2}
\]
The rate $\tau(N,d)$ is given by
\[
\tau(N,d) = \begin{cases}
N^{-1/2}, & d\le 3,\\
N^{-1/2}(\log N)^{1/2}, & d=4,\\
N^{-2/d}, & d\ge 5.
\end{cases}
\]

\subsubsection{Forward propagation of chaos lemma}

\begin{lemma}[Forward propagation of chaos]
\label{lem:propagation-forward}

Under assumptions {\rm(H1)--(H2)}, {\rm(H1bis)} and ${\rm(H_u^4)}$, there exists a constant $C_{\mathrm{prop}}>0$, independent of $N$ and $P\in\mathcal P$, such that
\[
\sup_{P\in\mathcal P}
\sup_{0\le t\le T}
\mathbb E^P
\Bigl[
W_2^2
\bigl(
\mu_t^{N,P},
\bar\mu_t^{N,P}
\bigr)
\Bigr]
\le
C_{\mathrm{prop}}
\,\tau(N,d).
\]
\end{lemma}

\begin{proof}
Since both systems are driven by the same control $u$, the contributions associated with $u$ cancel in the difference. Set $\delta X_t^i = X_t^{i,N} - X_t^i$.

The empirical coupling
\[
\pi_t^N = \frac1N \sum_{i=1}^N \delta_{(X_t^{i,N}, X_t^i)}
\]
is an admissible coupling between $\mu_t^{N,P}$ and $\bar\mu_t^{N,P}$.

By definition of the Wasserstein distance,
\[
W_2^2(\mu_t^{N,P},\bar\mu_t^{N,P})
\le
\int |x-y|^2\,\pi_t^N(dx,dy).
\]

We then obtain
\[
W_2^2(\mu_t^{N,P},\bar\mu_t^{N,P})
\le
\frac1N
\sum_{i=1}^{N}
|X_t^{i,N}-X_t^{i}|^2.
\]

Taking expectation and using exchangeability,
\[
\mathbb E^P
\Bigl[
W_2^2(\mu_t^{N,P},\bar\mu_t^{N,P})
\Bigr]
\le
\mathbb E^P
\Bigl[
|X_t^{1,N}-X_t^{1}|^2
\Bigr].
\]

This estimate constitutes the fundamental link between the measure error and the trajectory error.

By exchangeability of the particles, define
\[
\Delta(t)
:=
\sup_{P\in\mathcal P}\mathbb E^P[|\delta X_t^1|^2].
\]
Then
\[
\sup_{P\in\mathcal P}
\mathbb E^P\!\left[
W_2^2(\mu_t^{N,P},\bar\mu_t^{N,P})
\right]
\le \Delta(t).
\]

Applying It\^o's formula to $|\delta X_t^1|^2$, we obtain
\[
\begin{aligned}
d|\delta X_t^1|^2
&= 2\delta X_t^1\cdot\Delta b_t\,dt
+ 2\delta X_t^1\cdot\Delta\sigma_t a_t^{1/2}dW_t^{1,a}\\
&\quad
+ \operatorname{Tr}(\Delta\sigma_t a_t\Delta\sigma_t^\top)\,dt.
\end{aligned}
\]

Taking expectation and integrating from $0$ to $t$:
\[
\mathbb E^P[|\delta X_t^1|^2] \le 2\int_0^t\mathbb E^P[|\delta X_s^1||\Delta b_s|]ds + \int_0^t\mathbb E^P[\operatorname{Tr}(\Delta\sigma_s a_s\Delta\sigma_s^\top)]ds.
\]

Applying Young's inequality $2ab \le a^2 + b^2$, we obtain
\[
2|\delta X_s^1||\Delta b_s| \le |\delta X_s^1|^2 + |\Delta b_s|^2.
\]
Thus,
\[
\mathbb E^P[|\delta X_t^1|^2] \le C\int_0^t \mathbb E^P[|\delta X_s^1|^2 + |\Delta b_s|^2 + \operatorname{Tr}(\Delta\sigma_s a_s\Delta\sigma_s^\top)]ds.
\]

By the Lipschitz assumption (H2) and Lemma~\ref{lem:stability-observables}, we have
\[
|\Psi_\varphi(\mu_s^{N,P}) - \Psi_\varphi(\mu_s^P)| \le L_\varphi W_2(\mu_s^{N,P},\mu_s^P),
\]
and similarly for $\Psi_\beta$. Consequently,
\[
|\Delta b_s|^2 + \operatorname{Tr}(\Delta\sigma_s a_s\Delta\sigma_s^\top) \le C(|\delta X_s^1|^2 + W_2^2(\mu_s^{N,P},\mu_s^P)).
\]

Using (5.1) and the inequality $(a+b)^2 \le 2a^2 + 2b^2$,
\[
W_2^2(\mu_s^{N,P},\mu_s^P) \le 2W_2^2(\mu_s^{N,P},\bar\mu_s^{N,P}) + 2W_2^2(\bar\mu_s^{N,P},\mu_s^P).
\]

Thus,
\[
\mathbb E^P[|\delta X_t^1|^2] \le C\int_0^t\mathbb E^P[|\delta X_s^1|^2]ds + C\int_0^t\mathbb E^P[W_2^2(\mu_s^{N,P},\bar\mu_s^{N,P})]ds + C\int_0^t\mathbb E^P[W_2^2(\bar\mu_s^{N,P},\mu_s^P)]ds.
\]

Taking the supremum over $P$. By (5.2), the last term is bounded by $C\tau(N,d)$. Thus,
\[
\Delta(t) \le C\int_0^t\Delta(s)ds + C\int_0^t \sup_{P}\mathbb E^P[W_2^2(\mu_s^{N,P},\bar\mu_s^{N,P})]ds + C\tau(N,d).
\]
But $\sup_{P}\mathbb E^P[W_2^2(\mu_s^{N,P},\bar\mu_s^{N,P})] \le \Delta(s)$, hence
\[
\Delta(t) \le C\int_0^t\Delta(s)ds + C\tau(N,d).
\]

Gronwall's lemma gives $\Delta(t)\le C\tau(N,d)$ for every $t\in[0,T]$. Since the obtained estimate is uniform in $t$, we can take the supremum over $[0,T]$:
\[
\sup_{0\le t\le T}\Delta(t) \le C\tau(N,d).
\]
Consequently,
\[
\sup_{0\le t\le T}\sup_{P\in\mathcal P}\mathbb E^P[W_2^2(\mu_t^{N,P},\bar\mu_t^{N,P})] \le C\tau(N,d).
\]
Hence the result.
\end{proof}

\subsubsection{Robust propagation of chaos theorem}

\begin{theorem}[Robust propagation of chaos]
\label{thm:robust-chaos}

Under assumptions {\rm(H1)--(H2)}, {\rm(H1bis)} and ${\rm(H_u^4)}$, there exists a constant $C>0$, independent of $N$ and $P\in\mathcal P$, such that
\[
\sup_{P\in\mathcal P}
\sup_{0\le t\le T}
\mathbb E^P
\Bigl[
W_2^2
\bigl(
\mu_t^{N,P},
\mu_t^P
\bigr)
\Bigr]
\le
C\,\tau(N,d).
\]

In particular,
\[
\sup_{P\in\mathcal P}
\sup_{0\le t\le T}
\mathbb E^P
\Bigl[
W_2^2(\mu_t^{N,P},\mu_t^P)
\Bigr]
\longrightarrow 0,
\qquad N\to\infty.
\]

\end{theorem}

\begin{proof}
By the triangle inequality for $W_2$,
\[
W_2(\mu_t^{N,P},\mu_t^P)
\le
W_2(\mu_t^{N,P},\bar\mu_t^{N,P})
+
W_2(\bar\mu_t^{N,P},\mu_t^P).
\]

Using $(a+b)^2\le 2a^2+2b^2$, we obtain
\[
W_2^2(\mu_t^{N,P},\mu_t^P)
\le
2W_2^2(\mu_t^{N,P},\bar\mu_t^{N,P})
+
2W_2^2(\bar\mu_t^{N,P},\mu_t^P).
\]

Taking expectation under $P$:
\[
\mathbb E^P[W_2^2(\mu_t^{N,P},\mu_t^P)]
\le
2\mathbb E^P[W_2^2(\mu_t^{N,P},\bar\mu_t^{N,P})]
+
2\mathbb E^P[W_2^2(\bar\mu_t^{N,P},\mu_t^P)].
\]

Lemma~\ref{lem:propagation-forward} provides
\[
\sup_{P\in\mathcal P}
\sup_{0\le t\le T}
\mathbb E^P[W_2^2(\mu_t^{N,P},\bar\mu_t^{N,P})]
\le
C_1\tau(N,d).
\]

The application of the Fournier--Guillin theorem requires a uniform control of fourth moments. Lemma~\ref{lem:fourth-moment-limit} shows that
\[
\sup_{P\in\mathcal P}
\sup_{t\in[0,T]}
\mathbb E^P|X_t^{1}|^{4}
\le
M_4
<\infty.
\]

The constant appearing in the quantitative Fournier--Guillin estimate depends only on the dimension and this moment bound. There exists therefore a constant $C_{FG}=C(d,M_4)$ independent of $P\in\mathcal P$ such that
\[
\sup_{P\in\mathcal P}
\mathbb E^P[W_2^2(\bar\mu_t^{N,P},\mu_t^{P})]
\le
C_{FG}\tau(N,d).
\]

Combining these two estimates,
\[
\sup_{P\in\mathcal P}
\sup_{0\le t\le T}
\mathbb E^P[W_2^2(\mu_t^{N,P},\mu_t^P)]
\le
2(C_1+C_{FG})\tau(N,d) = C\,\tau(N,d).
\]

\end{proof}

\subsection{Stability of collective observables for the particle system}

\begin{proposition}[Stability of forward collective observables]
\label{prop:stability-observables-particle}

Under the assumptions of Theorem~\ref{thm:robust-chaos}, there exists a constant $C>0$, independent of $N$ and $P\in\mathcal P$, such that
\[
\sup_{P\in\mathcal P}
\sup_{0\le t\le T}
\mathbb E^P
\Bigl[
|\Psi_\varphi(\mu_t^{N,P}) - \Psi_\varphi(\mu_t^P)|^2
\Bigr]
\le
C\,\tau(N,d),
\]
and similarly for $\Psi_\beta$, $\Psi_\gamma$, $\Upsilon$.
\end{proposition}

\begin{proof}
By the Wasserstein stability lemma of collective observables (Lemma~\ref{lem:stability-observables}) and Theorem~\ref{thm:robust-chaos},
\[
\sup_{P\in\mathcal P}
\sup_{0\le t\le T}
\mathbb E^P
\Bigl[
|\Psi_\varphi(\mu_t^{N,P}) - \Psi_\varphi(\mu_t^P)|^2
\Bigr]
\le
L_\varphi^2
\sup_{P\in\mathcal P}
\sup_{0\le t\le T}
\mathbb E^P
\Bigl[
W_2^2(\mu_t^{N,P},\mu_t^P)
\Bigr]
\le
C\,\tau(N,d).
\]
The proofs for $\Psi_\beta$, $\Psi_\gamma$, $\Upsilon$ are identical.
\end{proof}

\subsection{Regularized backward propagation of chaos}

The next theorem is the second propagation-of-chaos result of the paper. It is stated as an autonomous theorem in order to clearly separate the robust forward propagation result, which is non-regularized and uniform over \(\mathcal P\), from the backward propagation result, which is obtained for a fixed Yosida parameter \(\lambda>0\).

\begin{theorem}[Backward propagation of chaos for fixed \(\lambda\)]
\label{thm:backward-chaos-fixed-lambda}

Assume {\rm(H1)--(H5)}, {\rm(H1bis)}, {\rm(H5bis)} and ${\rm(H_u^4)}$. Fix \(\lambda>0\). Then there exists a constant \(C_\lambda>0\), independent of \(N\) and \(P\in\mathcal P\), such that, for every \(i=1,\ldots,N\),
\[
\sup_{P\in\mathcal P}
\mathbb E^P
\left[
\sup_{0\le t\le T}
\bigl|Y_t^{i,N,\lambda}-Y_t^{i,\lambda}\bigr|^2
+
\int_0^T
\bigl|Z_t^{i,N,\lambda}-Z_t^{i,\lambda}\bigr|_{a_t}^2\,dt
\right]
\le
C_\lambda\,\tau(N,d).
\]
Moreover, if
\[
\bar\mu_t^{N,Y,\lambda}:=\frac1N\sum_{j=1}^N\delta_{Y_t^{j,\lambda}},
\]
then the empirical law of the regularized backward component satisfies
\[
\sup_{P\in\mathcal P}
\sup_{0\le t\le T}
\mathbb E^P
\left[
W_2^2
\bigl(
\mu_t^{N,Y,\lambda},
\mu_t^{P,Y^\lambda}
\bigr)
\right]
\le
C_\lambda\,\tau(N,d).
\]
\end{theorem}

\begin{proof}
Fix \(P\in\mathcal P\), \(i\in\{1,\ldots,N\}\), and set
\[
\Delta Y_t^i:=Y_t^{i,N,\lambda}-Y_t^{i,\lambda},
\qquad
\Delta Z_t^i:=Z_t^{i,N,\lambda}-Z_t^{i,\lambda},
\qquad
\Delta X_t^i:=X_t^{i,N}-X_t^i.
\]
The two backward equations differ only through the forward state, the collective observables, the empirical law of the backward component and the Yosida-regularized ALA signal. We write the difference of the full generators as
\[
\Delta F_s^i
=
\Delta f_s^i
+
\rho\,\Delta A_s^{i,\lambda},
\]
where
\[
\begin{aligned}
\Delta f_s^i
&=
f\bigl(s,X_s^{i,N},\Psi_\gamma(\mu_s^{N,P}),
Y_s^{i,N,\lambda},Z_s^{i,N,\lambda},
\Psi_\delta(\mu_s^{N,Y,\lambda}),u_s\bigr)\\
&\quad
-
f\bigl(s,X_s^{i},\Psi_\gamma(\mu_s^{P}),
Y_s^{i,\lambda},Z_s^{i,\lambda},
\Psi_\delta(\mu_s^{P,Y^\lambda}),u_s\bigr),
\end{aligned}
\]
and
\[
\Delta A_s^{i,\lambda}
=
A^\lambda_{\mathrm{ALA}}
\bigl(X_s^{i,N},\Upsilon(\mu_s^{N,P})\bigr)
-
A^\lambda_{\mathrm{ALA}}
\bigl(X_s^i,\Upsilon(\mu_s^P)\bigr).
\]
By the Lipschitz assumptions on \(f\), the Wasserstein stability of the collective observables, and the uniform ellipticity of \(a_t\), we have
\[
\begin{aligned}
|\Delta f_s^i|
&\le
C
\Bigl(
|\Delta X_s^i|
+ W_2(\mu_s^{N,P},\mu_s^P)
+ |\Delta Y_s^i|
+ |\Delta Z_s^i|_{a_s}\\
&\qquad\qquad
+ W_2(\mu_s^{N,Y,\lambda},\mu_s^{P,Y^\lambda})
\Bigr).
\end{aligned}
\]
Moreover, for fixed \(\lambda>0\), the Yosida approximation is Lipschitz and the reference observable \(\Upsilon\) is Wasserstein-Lipschitz; hence
\[
\begin{aligned}
|\Delta A_s^{i,\lambda}|
&\le
C_\lambda
\Bigl(
|\Delta X_s^i|
+
|\Upsilon(\mu_s^{N,P})-\Upsilon(\mu_s^P)|
\Bigr)\\
&\le
C_\lambda
\Bigl(
|\Delta X_s^i|
+
W_2(\mu_s^{N,P},\mu_s^P)
\Bigr).
\end{aligned}
\]
Combining the two estimates yields
\[
\begin{aligned}
|\Delta F_s^i|
&\le
C_\lambda
\Bigl(
|\Delta X_s^i|
+ W_2(\mu_s^{N,P},\mu_s^P)
+ |\Delta Y_s^i|
+ |\Delta Z_s^i|_{a_s}\\
&\qquad\qquad
+ W_2(\mu_s^{N,Y,\lambda},\mu_s^{P,Y^\lambda})
\Bigr),
\end{aligned}
\]
where \(C_\lambda\) may depend on \(\lambda\), but is independent of \(N\) and \(P\).

Applying It\^o's formula to \(|\Delta Y_t^i|^2\) between \(t\) and \(T\), and using \(d\langle B\rangle_t=a_tdt\), gives
\[
\begin{aligned}
|\Delta Y_t^i|^2
+\int_t^T|\Delta Z_s^i|_{a_s}^2\,ds
&=
|\Delta Y_T^i|^2
+2\int_t^T
\langle\Delta Y_s^i,\Delta F_s^i\rangle\,ds\\
&\quad
-2\int_t^T
\bigl\langle
\Delta Y_s^i,\Delta Z_s^i\,dB_s
\bigr\rangle .
\end{aligned}
\]
Young's inequality allows the term involving \(|\Delta Z_s^i|_{a_s}^2\) to be absorbed into the left-hand side:
\[
\begin{aligned}
2\langle \Delta Y_s^i,\Delta F_s^i\rangle
&\le
\frac12|\Delta Z_s^i|_{a_s}^2
+C_\lambda
\Bigl(
|\Delta Y_s^i|^2
+|\Delta X_s^i|^2\\
&\qquad\qquad
+W_2^2(\mu_s^{N,P},\mu_s^P)
+W_2^2(\mu_s^{N,Y,\lambda},\mu_s^{P,Y^\lambda})
\Bigr).
\end{aligned}
\]
Taking expectation under \(P\), the stochastic integral has zero mean. Therefore, for every \(t\in[0,T]\),
\[
\begin{aligned}
\mathbb E^P\bigl[|\Delta Y_t^i|^2\bigr]
&+
\frac12\mathbb E^P
\int_t^T|\Delta Z_s^i|_{a_s}^2\,ds\\
&\le
\mathbb E^P\bigl[|\Delta Y_T^i|^2\bigr]
+
C_\lambda
\int_t^T
\mathbb E^P
\Bigl[
|\Delta Y_s^i|^2
+|\Delta X_s^i|^2\\
&\qquad\qquad
+W_2^2(\mu_s^{N,P},\mu_s^P)
+W_2^2(\mu_s^{N,Y,\lambda},\mu_s^{P,Y^\lambda})
\Bigr]\,ds .
\end{aligned}
\]
This is the basic energy inequality. Applying the Burkholder--Davis--Gundy inequality to the stochastic integral in the Itô identity, then using Young's inequality once more, gives the corresponding estimate with the time supremum:
\[
\begin{aligned}
\mathbb E^P\Bigl[
\sup_{0\le r\le t}|\Delta Y_r^i|^2
+
\int_0^t|\Delta Z_s^i|_{a_s}^2\,ds
\Bigr]
&\le
C_\lambda
\mathbb E^P\Bigl[
|\Delta Y_T^i|^2\\
&\quad
+
\int_0^t
\Bigl(
|\Delta X_s^i|^2
+
W_2^2(\mu_s^{N,P},\mu_s^P)\\
&\qquad
+
W_2^2(\mu_s^{N,Y,\lambda},\mu_s^{P,Y^\lambda})
\Bigr)ds
\Bigr].
\end{aligned}
\]

The terminal difference is controlled by the Lipschitz continuity of \(g\) and \(\Upsilon\):
\[
|\Delta Y_T^i|^2
\le
C\Bigl(
|\Delta X_T^i|^2
+
W_2^2(\mu_T^{N,P},\mu_T^P)
\Bigr).
\]
By Theorem~\ref{thm:robust-chaos}, the forward Wasserstein error is bounded by \(C\tau(N,d)\), uniformly in \(P\). The same theorem and Lemma~\ref{lem:propagation-forward} also control \(\Delta X^i\). More explicitly,
\[
\sup_{P\in\mathcal P}
\sup_{0\le s\le T}
\mathbb E^P
\Bigl[
|\Delta X_s^i|^2
+
W_2^2(\mu_s^{N,P},\mu_s^P)
\Bigr]
\le
C\,\tau(N,d).
\]

It remains to control the backward empirical term. By the triangle inequality,
\[
W_2^2(\mu_s^{N,Y,\lambda},\mu_s^{P,Y^\lambda})
\le
2W_2^2(\mu_s^{N,Y,\lambda},\bar\mu_s^{N,Y,\lambda})
+
2W_2^2(\bar\mu_s^{N,Y,\lambda},\mu_s^{P,Y^\lambda}).
\]
The first term is bounded by the coupling error:
\[
W_2^2(\mu_s^{N,Y,\lambda},\bar\mu_s^{N,Y,\lambda})
\le
\frac1N\sum_{j=1}^N
|Y_s^{j,N,\lambda}-Y_s^{j,\lambda}|^2.
\]
The second term is the empirical approximation error for the i.i.d. variables \(Y_s^{j,\lambda}\). For fixed \(\lambda\), the a priori estimates for the regularized backward system, together with the fourth moment estimate for the forward component and the linear growth of the generator, give a uniform fourth moment bound for \(Y^{j,\lambda}\). The Fournier--Guillin estimate therefore yields
\[
\sup_{P\in\mathcal P}
\sup_{0\le s\le T}
\mathbb E^P
\bigl[
W_2^2(\bar\mu_s^{N,Y,\lambda},\mu_s^{P,Y^\lambda})
\bigr]
\le
C_\lambda\,\tau(N,d).
\]
Combining the preceding inequalities and using exchangeability, we obtain, for
\[
\Theta(t):=
\sup_{P\in\mathcal P}
\sup_{1\le j\le N}
\mathbb E^P
\left[
\sup_{0\le r\le t}|Y_r^{j,N,\lambda}-Y_r^{j,\lambda}|^2
+
\int_0^t|Z_s^{j,N,\lambda}-Z_s^{j,\lambda}|_{a_s}^2ds
\right],
\]
the estimate
\[
\Theta(t)
\le
C_\lambda\,\tau(N,d)
+
C_\lambda\int_0^t\Theta(s)\,ds.
\]
Gronwall's lemma gives \(\Theta(T)\le C_\lambda\tau(N,d)\). The Wasserstein estimate for the empirical backward law follows from the same decomposition with \(\bar\mu_t^{N,Y,\lambda}\). This proves the theorem.
\end{proof}

\begin{remark}[Why the rate is not uniform in \((N,\lambda)\)]
\label{rem:nonuniform-backward-chaos}

The preceding theorem is stated for a fixed regularization parameter \(\lambda>0\). This is the natural level at which the regularized backward dynamics is Lipschitz. The constant \(C_\lambda\) may blow up as \(\lambda\downarrow0\), because the Lipschitz constant of the raw Yosida approximation \(A^\lambda_{\mathrm{ALA}}\) is of order \(1/\lambda\).

A rate that is uniform simultaneously in \(N\) and \(\lambda\) would require an additional non-contact condition at the threshold. A typical assumption is the existence of constants \(c,\alpha>0\) such that, for all sufficiently small \(r>0\),
\[
\sup_{P\in\mathcal P}
\sup_{0\le t\le T}
P\!\left(
\operatorname{dist}
\bigl(
X_t^P,
\{x\in\mathbb R^d:\ x=\Upsilon(\mu_t^P)\mathbf 1\}
\bigr)
\le r
\right)
\le
c\,r^\alpha.
\]
Equivalently, one may impose a stronger separation condition preventing the limiting forward state from spending mass on the non-smooth threshold. Without such a non-contact hypothesis, no uniform rate in \((N,\lambda)\) should be expected.
\end{remark}

\subsection{Remark on the convergence rate}

The rate $\tau(N,d)$ comes from \cite{FournierGuillin2015}. For a classical introduction to propagation of chaos, see \cite{Sznitman1991}; for general convergence methods for Markov processes and interacting particle systems, see \cite{EthierKurtz2009}.

\section{Discussion, limitations and future perspectives}
\label{sec:discussion-perspectives}

This section clarifies the exact scope of the results and the natural limitations of the framework developed in this paper. Without an additional structural assumption, the present analysis does not provide a complete propagation of chaos result that is uniform for the non-regularized backward component. The robust result established here is more precise: the forward propagation of chaos is quantitative and uniform over the non-dominated family \(\mathcal P\), whereas the backward propagation result is obtained for every fixed regularization parameter \(\lambda>0\).

To the best of our knowledge, no propagation-of-chaos result currently provides a rate that is uniform simultaneously in \((N,\lambda)\) for multivalued McKean--Vlasov systems governed by maximal monotone operators and a nonsmooth collective reference of the type considered here. This observation explains why we carefully separate the forward result, the fixed-\(\lambda\) backward result, and the open question of the simultaneous limit.

\subsection{Why the limit \(\lambda\downarrow0\) is delicate}

The main difficulty comes from the nonsmooth nature of the ALA functional. The threshold
\[
\{x\in\mathbb R^d:\ x=\Upsilon(\mu_t^P)\mathbf 1\}
\]
is exactly the set on which the subdifferential becomes multivalued. The Yosida approximation temporarily replaces this multivalued signal by a Lipschitz map, but its Lipschitz constant is typically of order \(1/\lambda\). Consequently, the backward estimates obtained at the regularized level naturally contain a constant \(C_\lambda\), which may deteriorate as \(\lambda\downarrow0\).

This phenomenon is not merely technical. Near the threshold, very small perturbations of the forward state or of the collective observable may change the effective selection of the subdifferential. In a particle approximation, such perturbations are unavoidable: they come both from the mean-field approximation error and from the empirical Wasserstein error. As long as the limiting dynamics may touch the threshold with non-negligible probability, a uniform control in \(\lambda\) should not be expected.

\subsection{Possible role of a non-contact condition}

A natural way to recover uniformity would be to impose a non-contact condition. Such a condition means that the limiting forward state does not spend too much time, or too much probability mass, near the nonsmooth threshold. For instance, one may assume that there exist constants \(c,\alpha>0\) such that
\[
\sup_{P\in\mathcal P}\sup_{0\le t\le T}
P\!\left(
\operatorname{dist}
\bigl(
X_t^P,
\{x\in\mathbb R^d:\ x=\Upsilon(\mu_t^P)\mathbf 1\}
\bigr)
\le r
\right)
\le c r^\alpha,
\qquad r>0.
\]
This hypothesis should not be imposed artificially in the main theorems of the present paper, because it restricts the scope of the model. Rather, it should be viewed as an additional structural assumption for a future analysis of the simultaneous limit \(N\to\infty\), \(\lambda\downarrow0\). Under such a condition, one may hope to control the contribution of the critical region and obtain sharper estimates for the non-regularized backward error.

\subsection{Exact scope of the results}

The results of this paper should therefore be read as follows. First, the robust forward dynamics is well-posed and stable under natural Lipschitz assumptions. Second, the selected backward variational component is constructed through a monotone signal
\[
\Gamma_t^P\in A_{\mathrm{ALA}}\bigl(X_t^P,\Upsilon(\mu_t^P)\bigr),
\]
and the canonical representation associated with the minimal norm selection gives a unique backward pair. Third, the forward propagation of chaos is quantitative and robust. Fourth, the regularized backward component propagates chaos for every fixed \(\lambda>0\), with a constant that may depend on \(\lambda\).

This formulation is intentionally precise: it avoids confusing the original multivalued problem with its Yosida approximation. It also makes visible the remaining mathematical program, namely the study of backward stability when the regularization vanishes while the number of particles tends to infinity.

\subsection{Future perspectives}

Several natural extensions emerge from this work. The first is to establish non-regularized backward propagation of chaos under a non-contact condition. The second concerns robust mean field games with asymmetric loss aversion, where the variational signal \(\Gamma\) would enter the equilibrium conditions. The third direction is robust optimal control and the associated Pontryagin maximum principle, especially when the collective observables are estimated from data. Finally, financial applications include robust portfolio management, behavioral criteria under volatility ambiguity, collective regulation, and systemic-risk models with asymmetric loss thresholds.

\section*{Conclusion}

We have developed a robust variational framework for forward--backward McKean--Vlasov systems under a non-dominated family of probabilities. The model combines law dependence, nonlinear collective observables, and asymmetric loss aversion, leading to a selected backward variational component driven by a distribution-dependent maximal monotone signal evaluated at the forward state.

Well-posedness is obtained in two steps: first for the forward dynamics, through a McKean--Vlasov fixed point, and then for the selected backward variational component, through Yosida regularization, uniform estimates, and Minty--Brezis graph closure. We also establish stability results with respect to data and collective observables.

The particle approximation provides quantitative forward propagation of chaos with constants that remain uniform over the admissible family. For every fixed regularization parameter \(\lambda>0\), the regularized backward component also propagates chaos with rate \(\tau(N,d)\), although the corresponding constant may deteriorate as \(\lambda\downarrow0\). This limitation is intrinsic to the non-smooth ALA threshold and motivates future work under non-contact assumptions.

Natural perspectives concern the simultaneous limit \(N\to\infty\), \(\lambda\to0\), robust mean field games with asymmetric loss aversion, partial observation, learning-based estimation of collective observables, and financial applications involving ambiguity-sensitive behavioral preferences.



\section*{Declarations}

\subsection*{Conflict of interest}

The authors declare that they have no known competing financial interests or personal relationships that could have appeared to influence the work reported in this article.

\subsection*{Use of artificial intelligence}

The authors declare that artificial intelligence tools were used only for linguistic editing, translation support, formatting, and clarity improvements. The authors reviewed and approved the final manuscript and remain fully responsible for its scientific content.

\subsection*{Data availability}

No datasets were generated or analyzed during the current study. This article is theoretical, and all results are derived from the assumptions and mathematical arguments presented in the manuscript.


\begin{thebibliography}{99}

\bibitem{AmbrosioGigliSavare2005} Ambrosio, L., Gigli, N. \& Savaré, G. (2005). \emph{Gradient Flows in Metric Spaces and in the Space of Probability Measures}. Birkhäuser.

\bibitem{BayraktarZhang2023} Bayraktar, E. \& Zhang, X. (2023). McKean-Vlasov forward--backward stochastic differential equations with jumps. \emph{Stochastic Processes and their Applications}, 155, 1-35.

\bibitem{Brezis1973} Brézis, H. (1973). \emph{Opérateurs maximaux monotones et semi-groupes de contractions dans les espaces de Hilbert}. North-Holland.

\bibitem{Brezis2011} Brezis, H. (2011). \emph{Functional Analysis, Sobolev Spaces and Partial Differential Equations}. Springer.

\bibitem{CarmonaDelarue2015} Carmona, R. \& Delarue, F. (2015). Forward-backward stochastic differential equations and controlled McKean-Vlasov dynamics. \emph{The Annals of Probability}, 43(5), 2647-2700.

\bibitem{CarmonaDelarue2018I} Carmona, R. \& Delarue, F. (2018). \emph{Probabilistic Theory of Mean Field Games with Applications I}. Springer.

\bibitem{CarmonaDelarue2018II} Carmona, R. \& Delarue, F. (2018). \emph{Probabilistic Theory of Mean Field Games with Applications II}. Springer.

\bibitem{DenisHuPeng2011} Denis, L., Hu, M. \& Peng, S. (2011). Function spaces and capacity related to a sublinear expectation: application to $G$-Brownian motion paths. \emph{Potential Analysis}, 34(2), 139-161.

\bibitem{Duan2025} Duan, S. (2025). McKean-Vlasov stochastic variational inequalities. \emph{Preprint}.

\bibitem{EthierKurtz2009} Ethier, S. N. \& Kurtz, T. G. (2009). \emph{Markov Processes: Characterization and Convergence}. Wiley.

\bibitem{FournierGuillin2015} Fournier, N. \& Guillin, A. (2015). On the rate of convergence in Wasserstein distance of the empirical measure. \emph{Probability Theory and Related Fields}, 162(3), 707-738.

\bibitem{HuaLuo2024} Hua, Z. \& Luo, P. (2024). McKean-Vlasov forward--backward stochastic differential equations with time-delayed generators. \emph{arXiv preprint}.

\bibitem{KahnemanTversky1979} Kahneman, D. \& Tversky, A. (1979). Prospect theory: An analysis of decision under risk. \emph{Econometrica}, 47(2), 263-291.

\bibitem{KaratzasShreve1991} Karatzas, I. \& Shreve, S. E. (1991). \emph{Brownian Motion and Stochastic Calculus}. Springer.

\bibitem{LasryLions2007} Lasry, J. M. \& Lions, P. L. (2007). Mean field games. \emph{Japanese Journal of Mathematics}, 2(1), 229-260.

\bibitem{Liu2025} Liu, Y. (2025). Backward stochastic variational inequalities of McKean-Vlasov type. \emph{Preprint}.

\bibitem{McKean1966} McKean, H. P. (1966). A class of Markov processes associated with nonlinear parabolic equations. \emph{Proceedings of the National Academy of Sciences}, 56(6), 1907-1911.

\bibitem{MaticiucRascanu2015} Maticiuc, L. \& R\u{a}\c{s}canu, A. (2015). Backward stochastic variational inequalities on random interval. \emph{Bernoulli}, 21(2), 1166-1199.

\bibitem{NingWuZheng2024} Ning, J., Wu, Z. \& Zheng, W. (2024). McKean-Vlasov stochastic variational inequalities. \emph{arXiv preprint}.

\bibitem{NutzVanHandel2013} Nutz, M. \& van Handel, R. (2013). Constructing sublinear expectations on path space. \emph{Stochastic Processes and their Applications}, 123(8), 3100-3121.

\bibitem{PardouxPeng1990} Pardoux, E. \& Peng, S. (1990). Adapted solution of a backward stochastic differential equation. \emph{Systems \& Control Letters}, 14(1), 55-61.

\bibitem{PardouxRascanu1998} Pardoux, E. \& R\u{a}\c{s}canu, A. (1998). Backward stochastic differential equations with subdifferential operator and real valued noise. \emph{Bernoulli}, 4(4), 441-461.

\bibitem{Peng2019} Peng, S. (2019). \emph{Nonlinear Expectations and Stochastic Calculus under Uncertainty}. Springer.

\bibitem{PengWu1999} Peng, S. \& Wu, Z. (1999). Fully coupled forward--backward stochastic differential equations and applications to optimal control. \emph{SIAM Journal on Control and Optimization}, 37(3), 825-843.

\bibitem{Rockafellar1970} Rockafellar, R. T. (1970). \emph{Convex Analysis}. Princeton University Press.

\bibitem{SonerTouziZhang2011} Soner, H. M., Touzi, N. \& Zhang, J. (2011). Quasi-sure stochastic analysis through aggregation. \emph{Electronic Journal of Probability}, 16, 1844-1879.

\bibitem{SonerTouziZhang2012} Soner, H. M., Touzi, N. \& Zhang, J. (2012). Wellposedness of second order backward SDEs. \emph{Probability Theory and Related Fields}, 153(1), 149-190.

\bibitem{Sznitman1991} Sznitman, A. S. (1991). Topics in propagation of chaos. \emph{Lecture Notes in Mathematics}, 1464, 165-251.

\bibitem{Villani2009} Villani, C. (2009). \emph{Optimal Transport: Old and New}. Springer.

\bibitem{Yosida1965} Yosida, K. (1965). \emph{Functional Analysis}. Springer.

\end{thebibliography}
\end{document}